\documentclass[10pt]{smfart}

\RequirePackage[T1]{fontenc}
\RequirePackage{amsfonts,latexsym,amssymb}

\RequirePackage{smfenum}
\RequirePackage[frenchb]{babel}
\addto\extrasfrenchb{\bbl@nonfrenchitemize
\bbl@nonfrenchspacing}

\RequirePackage{mathrsfs}
\let\mathcal\mathscr

\usepackage[matrix,arrow]{xy}
\usepackage{url}
\usepackage{accents}
\usepackage{enumitem}
\usepackage{fge}
\newcommand{\moins}{\mathbin{\fgebackslash}}

\theoremstyle{plain}
\numberwithin{equation}{section}
\newtheorem{prop}[equation]{\propname}
\newtheorem{theo}[equation]{\theoname}

\newtheorem{coro}[equation]{\coroname}

\newtheorem{lemm}[equation]{\lemmname}
\theoremstyle{definition}
\theoremstyle{remark}
\newtheorem{defi}[equation]{\definame}
\newtheorem{rema}[equation]{\remaname}

\newtheorem{conv}[equation]{Convention}

\let\emptyset\varnothing

\def\paskunas{Pa\v{s}k\={u}nas}

\let\cal\mathcal
\let\goth\mathfrak

\def\pp{\goth p}
\def\Q{{\bf Q}} \def\Z{{\bf Z}}
\def\zp{\Z_p} \def\qp{\Q_p}

\def\N{{\bf N}}
\def\O{{\cal O}}
\def\G{{\cal G}} 

\def\dual{{\boldsymbol *}}

\def\Qbar{\overline{\bf Q}}

\def\epsilon{\varepsilon}

 \def\ddr{{\bf D}_{\rm dR}}

\def\dpst{{\bf D}_{\rm pst}}

\def\ainf{{\rm A}_{{\rm inf}}}
\def\bst{{\rm B}_{{\rm st}}}
\def\bcris{{\rm B}_{{\rm cris}}} 
\def\bdr{{\rm B}_{{\rm dR}}}

\def\Bcris{{\mathbb B}_{{\rm cris}}} 
\def\Bdr{{\mathbb B}_{{\rm dR}}}

\def\piqp{{\bf P}^1}

 \def\A{{\bf A}}

\def\matrice#1#2#3#4{{\big(\begin{smallmatrix}#1&#2\\ #3&#4\end{smallmatrix}\big)}}

\def\rg{{\rm R}\Gamma}
\def\ocirc#1{\accentset{\circ}{#1}}

\newcommand{\R}{\mathrm {R} }

\newcommand{\Aut}{\operatorname{Aut} }

  \newcommand{\proet}{\operatorname{pro\acute{e}t}  }

 \newcommand{\eet}{\operatorname{\acute{e}t} }

 \newcommand{\dlog}{\operatorname{dlog} }
 
 \newcommand{\an}{\operatorname{an} }

 \newcommand{\Spf}{\operatorname{Spf} }

 \newcommand{\Gal}{\operatorname{Gal} }

 \newcommand{\sy}{{\mathcal{Y}}}

 \newcommand{\sn}{{\mathcal N}}
 \newcommand{\so}{{\mathcal O}}

 \newcommand{\sx}{{\mathcal{X}}}

      \def\A{{\bf A}}
\def\jcdot{{\scriptscriptstyle\bullet}}

\def\invlim{\mathop{\vtop{\ialign{##\crcr$\hfill{\lim}\hfil$\crcr
\noalign{\kern1pt\nointerlineskip}\leftarrowfill\crcr\noalign
{\kern -3pt}}}}\limits}
\def\dirlim{\mathop{\vtop{\ialign{##\crcr$\hfill{\lim}\hfil$\crcr
\noalign{\kern1pt\nointerlineskip}\rightarrowfill\crcr\noalign
{\kern -3pt}}}}\limits}

\def\epsilon{\varepsilon}
\let\mathcal\mathscr

\def\tHK{\widetilde{\rm HK}}
\def\mv{{\cal M}^\varpi}
\def\fn{\Phi{\rm N}}

\begin{document}
\title[Cohomologie des rev\^etements du demi-plan de Drinfeld]
{Cohomologie $p$-adique de la tour de Drinfeld: le cas de la dimension~$1$}
\author{Pierre Colmez}
\address{CNRS, IMJ-PRG, Sorbonne Universit\'e, 4 place Jussieu,
75005 Paris, France}
\email{pierre.colmez@imj-prg.fr}
\author{Gabriel Dospinescu}
\address{CNRS, UMPA, \'Ecole Normale Sup\'erieure de Lyon, 46 all\'ee d'Italie, 69007 Lyon, France}
\email{gabriel.dospinescu@ens-lyon.fr}
\author{Wies{\l}awa Nizio{\l}}
\address{CNRS, UMPA, \'Ecole Normale Sup\'erieure de Lyon, 46 all\'ee d'Italie, 69007 Lyon, France}
\email{wieslawa.niziol@ens-lyon.fr}
\thanks{Les trois auteurs sont membres du projet Percolator de l'ANR
(projet ANR-14-CE25).}
\begin{abstract}
Nous calculons la cohomologie \'etale g\'eom\'etrique $p$-adique des rev\^etements du demi-plan de Drinfeld,
et, dans le cas o\`u le corps de base est $\Q_p$,  montrons qu'elle r\'ealise la correspondance de Langlands locale
$p$-adique pour les repr\'esentations de de Rham de dimension~$2$ (\`a poids $0$ et $1$).
\end{abstract}
\begin{altabstract}
We compute the $p$-adic geometric \'etale cohomology of the coverings of Drinfeld half-plane,
and we show that, if the base field is $\Q_p$, this cohomology encodes the
$p$-adic local Langlands correspondence for $2$-dimensional de Rham representations
(of weight $0$ and $1$).
\end{altabstract}
\setcounter{tocdepth}{1}

\maketitle

{\Small
\tableofcontents
}

\section*{Introduction}
Soient $F$ une extension finie de $\Q_p$ et $C$ le compl\'et\'e d'une cl\^oture alg\'ebrique de~$F$.
On note $\G_F$ le groupe de Galois absolu de $F$
et ${\rm W}_F$ (resp.~${\rm WD}_F$) son groupe de Weil (resp.~de Weil-Deligne).

Gr\^ace aux travaux de~\cite{Faltings, Faltings2tours, Fargues, Harris, HT}, on sait que la cohomologie \'etale
$\ell$-adique de la tour de Drinfeld, pour $\ell\neq p$,
encode les correspondances de Langlands et Jacquet-Langlands
locales classiques pour ${\rm GL}_n(F)$.  Le but de cet article
est d'expliquer que:

$\bullet$ (vraisemblablement)
la cohomologie 
\'etale $p$-adique encode l'hypoth\'etique correspondance de Langlands $p$-adique,

$\bullet$ 
(accessoirement) la
cohomologie \'etale $p$-adique d'objets du genre des rev\^etements \'etales
des espaces de Drinfeld n'est pas aussi abominable que ce que l'on aurait pu penser. 

Une premi\`ere indication que ceci pourrait \^etre le cas est le r\'esultat suivant
de Drinfeld~\cite{Drinfeld}, dans lequel $\Omega_{\rm Dr}=\piqp_C-\piqp(F)$ d\'esigne le demi-plan de Drinfeld,
${\rm St}_{\Q_p}^{\rm cont}$ la steinberg continue\footnote{C'est le quotient de l'espace ${\cal C}(\piqp(F),\Q_p)$
des fonctions continues sur $\piqp(F)$, \`a valeurs dans~$\Q_p$, par celui des fonctions
constantes.}
\`a coefficients dans $\Q_p$ 
et $({\rm St}_{\Q_p}^{\rm cont})^\dual$ son dual continu.
\begin{prop}
On a un isomorphisme de ${\rm GL}_2(F)\times\G_{F}$-repr\'esentations
$$H^1_{\eet}(\Omega_{\rm Dr},\Q_p(1))\cong ({\rm St}_{\Q_p}^{\rm cont})^\dual,$$
l'action de $\G_{F}$ sur le membre de droite \'etant, par d\'efinition,
triviale.
\end{prop}
Ce r\'esultat est encourageant\footnote{Il a jou\'e un r\^ole d\'eterminant pour nous persuader
de regarder ce qui se passe pour les \'etages sup\'erieurs de la tour mais, comme nous l'ont
signal\'e Grosse-Kl\"onne et Berkovich, la preuve de Drinfeld de ce r\'esultat pour $\ell=p$ laisse
beaucoup \`a d\'esirer... (Drinfeld \'enonce le r\'esultat pour la cohomologie \'etale $\ell$-adique,
pour tout nombre premier $\ell$, y compris $\ell=p$.)} car il montre que
la cohomologie \'etale $p$-adique de $\Omega_{\rm Dr}$ est un objet
de taille raisonnable (c'est une repr\'esentation coadmissible de ${\rm GL}_2(F)$).  
Il est quand m\^eme un peu trompeur car, comme nous le verrons,
la cohomologie \'etale $p$-adique des rev\^etements du demi-plan de Drinfeld
est loin d'\^etre aussi simple: en particulier, elle n'est pas coadmissible.

 Un peu plus pr\'ecis\'ement, nous montrons
que, si $F=\Q_p$ et en dimension~$1$, la cohomologie \'etale $p$-adique
de la tour ${\cal M}_\infty$ de Drinfeld encode
la correspondance de Langlands locale $p$-adique pour les repr\'esentations
de de Rham de $\G_{\Q_p}$, de dimension~$2$, \`a poids de Hodge-Tate $0$ et $1$,
dont la repr\'esentation de Weil-Deligne associ\'ee est irr\'eductible,
ce qui fournit une construction g\'eom\'etrique de cette correspondance (pour ces
repr\'esentations particuli\`eres): si $V$ est une telle repr\'esentation,
$${\rm Hom}_{{\rm W}_{\Q_p}}(V,H^1_{\eet}({\cal M}_\infty,\Q_p(1)))={\rm JL}(V)\otimes\Pi(V)^\dual,$$
o\`u
${\rm JL}(V)$ et $\Pi(V)$ sont les objets associ\'es \`a $V$ via
les correspondances de Jacquet-Langlands locale classique~\cite{JL} et de Langlands
locale $p$-adique~\cite{gl2,CDP}.
L'\'enonc\'e analogue pour
la cohomologie \'etale $\ell$-adique est valable pour tout $F$, toute dimension~$n$,
et toute $\Qbar_\ell$-repr\'esentation irr\'eductible $V$ de $W_F$,
de dimension~$n$.

Un obstacle pour \'etendre nos r\'esultats \`a d'autres cas est l'absence\footnote{En dehors
des candidats de~\cite{6auteurs}.}
de correspondance de Langlands locale $p$-adique pour d'autres groupes que ${\bf GL}_2(\Q_p)$,
mais la formule ci-dessus a un sens en g\'en\'eral,
et on peut esp\'erer que ce qui en sort a un lien avec l'hypoth\'etique
correspondance de Langlands locale $p$-adique pour ${\bf GL}_n(F)$.

Pour \'enoncer pr\'ecis\'ement nos r\'esultats, nous allons avoir besoin
d'introduire un certain nombre d'objets et de notations.

\subsection{Les rev\^etements du demi-plan de Drinfeld}
Soient:

$\bullet$
$\O_F$ l'anneau de ses entiers et $\varpi$
une uniformisante de $F$,

$\bullet$ $G={\bf GL}_2(F)$, 

$\bullet$ $\check G$ le groupe des \'el\'ements inversibles de l'alg\`ebre de quaternions~$D$
de centre $F$, $\O_D$ l'ordre maximal de $D$, et $\varpi_D$ une uniformisante de $\O_D$.

\smallskip
Le demi-plan $p$-adique (de Drinfeld)
$\piqp_F-\piqp(F)$ admet une structure naturelle d'espace analytique
rigide $\Omega_{{\rm Dr},F}$ sur $F$, et une action de $G$ par homographies,
qui respecte cette structure.
Drinfeld a d\'efini \cite{Drinfeld1}
une tour de rev\^etements $\breve{\cal M}_n$, pour $n\in\N$,
de ce demi-plan, v\'erifiant les propri\'et\'es suivantes:

$\bullet$ $\breve{\cal M}_n$ est d\'efini sur $\breve{F}=\widehat{F^{\rm nr}}$
et muni d'une action de ${\rm W}_F$ compatible avec l'action naturelle sur $\breve{F}$.

$\bullet$ $\breve{\cal M}_n$ est muni d'actions de $G$ et de $\check G$
commutant entre elles ainsi qu'avec l'action de ${\rm W}_F$,
et les fl\`eches de transition $\breve{\cal M}_{n+1}\to \breve{\cal M}_n\to
\Omega_{{\rm Dr},F}$ sont ${\rm W}_F$, $\check G$ et $G$-\'equivariantes (l'action de
$\check G$ sur $\Omega_{{\rm Dr},F}$ \'etant l'action triviale).

$\bullet$ Si $n\geq 1$, alors $\breve{\cal M}_n$ est un rev\^etement galoisien
de $\breve{\cal M}_0$, de groupe de Galois $\O_D^\dual/(1+\varpi_D^n\O_D)$.

On note simplement ${\cal M}_n$ l'extension des scalaires
de $\breve{\cal M}_n$ \`a $C$:
$${\cal M}_n=C\times_{\breve{F}} \breve{\cal M}_n.$$ 

On note ${\cal M}_\infty$ le syst\`eme projectif (non compl\'et\'e)
des ${\cal M}_n$: si $H^{\bullet}$ est un th\'eorie cohomologique
raisonnable, contravariante, alors $H^\bullet({\cal M}_\infty)=
\varinjlim_n H^\bullet({\cal M}_n)$; par exemple,
$\O({{\cal M}_\infty})=H^0({\cal M}_\infty,\O)=
\varinjlim_nH^0({\cal M}_n,\O)=\varinjlim_n \O({{\cal M}_n})$.

Si $H=G,\check G,{\rm W}_F$, on dispose d'un morphisme
de groupes naturel $\nu_H:H\to F^\dual,$
o\`u $\nu_G=\det$, $\nu_{\check G}$ est la norme r\'eduite,
et $\nu_{{\rm W}_F}$ est le compos\'e de ${\rm W}_F\to {\rm W}_F^{\rm ab}$ et
de l'isomorphisme ${\rm W}_F^{\rm ab}\cong F^\dual$ de la th\'eorie locale du corps
de classes.  
L'ensemble $\pi_0({\cal M}_\infty)$
des composantes connexes de ${\cal M}_\infty$ est un espace homog\`ene principal
sous l'action de $F^\dual$ et $H=G,\check G,{\rm W}_F$ agit sur $\pi_0({\cal M}_\infty)$
\`a travers $\nu_H:H\to F^\dual$.  En particulier, $\check G$ agit sur
$\pi_0({\cal M}_\infty)$ \`a travers la norme r\'eduite (ces assertions se d\'eduisent, par exemple, de 
\cite{Strauchgeocc} et de la comparaison avec la tour de Lubin-Tate \cite{Faltings, Fargues}).

\Subsection{Cohomologie \'etale de ${\cal M}_\infty$ et correspondance de Langlands locale}
Soit $L$ une extension finie de $\Q_p$.
Si $V$ est une $L$-repr\'esentation de de Rham de~$\G_{\Q_p}$, de dimension~$2$, \`a poids $0$ et $1$,
on associe \`a $V$ les objets suivants\footnote{
${\rm WD}(V)$ est obtenue \`a partir de $\dpst(V)[1]$ par la recette de Fontaine~\cite{FonAst},
${\rm LL}(V)$ \`a partir de ${\rm WD}(V)$ par
la correspondance de Langlands locale et ${\rm JL}(V)$ \`a partir de
${\rm LL}(V)$ par
la correspondance de Jacquet-Langlands locale, $\Pi(V)$ est la repr\'esentation
associ\'ee \`a $V$ par la correspondance de Langlands
locale $p$-adique~\cite{gl2,CDP}.
 
Les caract\`eres centraux de ${\rm LL}(V)$ et ${\rm JL}(V)$ sont \'egaux et co\"{\i}ncident
avec $\det{\rm WD}(V)\cdot|\ |$ (vu comme caract\`ere
de $W_{\Q_p}^{\rm ab}\cong\Q_p^{\dual}$, le frobenius arithm\'etique s'envoyant sur $p$)
et avec la restriction de $\det V\cdot \chi_{\rm cyclo}^{-1}$ \`a $W_{\Q_p}$
(on a ${\rm WD}(V)={\rm WD}(\dpst(V))\otimes |\ |^{-1}$).}

$\bullet$ un $L$-$(\varphi,N,\G_{\Q_p})$-module $\dpst(V)$, de rang~$2$ sur $L\otimes_{\Q_p}\Q_p^{\rm nr}$,

$\bullet$ une $L$-repr\'esentation\footnote{Le $[1]$ signifie que
l'on multiplie par $p$ l'action de $\varphi$, i.e.~$\dpst(V)[1]=\dpst(V(-1))$ si $V(-1)$
d\'esigne la tordue de $V$ par l'inverse du caract\`ere cyclotmique.}
 ${\rm WD}(V):={\rm WD}(\dpst(V)[1])$ de ${\rm WD}_{\Q_p}$, de dimension~$2$,

$\bullet$ une $L$-repr\'esentation lisse irr\'eductible\footnote{De dimension infinie.} ${\rm LL}(V):={\rm LL}({\rm WD}(V))$ de $G$,

$\bullet$ une $L$-repr\'esentation lisse irr\'eductible\footnote{Donc de dimension finie.}
 ${\rm JL}(V):={\rm JL}({\rm LL}(V))$ de $\check G$,

$\bullet$ une $L$-repr\'esentation unitaire continue $\Pi(V)$ de $G$.

Les repr\'esentations ${\rm WD}(V)$, ${\rm LL}(V)$ et ${\rm JL}(V)$ ne d\'ependent
que du $L$-$(\varphi,N,\G_{\Q_p})$-module $\dpst(V)$,
mais pas
$\Pi(V)$ 
qui, elle, d\'epend vraiment
de~$V$.  On retrouve ${\rm LL}(V)$ \`a partir de $\Pi(V)$ en prenant les vecteurs localement
constants sous l'action de $G$ (cf.~\cite{gl2,Em08,Dosp1,poids}):
$${\rm LL}(V)=\Pi(V)^{\rm lisse}.$$

On dit que $V$ est {\it supercuspidale} si 
${\rm WD}(V)$ est irr\'eductible (auquel cas ${\rm LL}(V)$ est supercuspidale), ce qui implique,
en particulier, que $N=0$ sur $\dpst(V)$ (i.e.\,$V$ est potentiellement cristalline).

\begin{theo}\label{intro1}
Soit $V$ une $L$-repr\'esentation absolument irr\'eductible de $\G_{\Q_p}$, de dimension~$\geq 2$.

{\rm (i)}
Si $V$ est supercuspidale, de dimension~$2$, \`a poids $0$ et $1$,
$${\rm Hom}_{{\rm W}_{\Q_p}}(V,L\otimes_{\Q_p}H^1_{\eet}({\cal M}_\infty,\Q_p(1)))=
{\rm JL}(V)\otimes_L \Pi(V)^\dual.$$

{\rm (ii)} Dans le cas contraire, ${\rm Hom}_{{\rm W}_{\Q_p}}(V,L\otimes_{\Q_p}H^1_{\eet}({\cal M}_\infty,\Q_p(1)))=0$.
\end{theo}

\begin{rema}
{\rm (i)}  Le th.\,\ref{intro1} devrait pouvoir s'\'etendre aux repr\'esentations supercuspidales
de poids quelconques en rempla\c cant $\Q_p(1)$
par les puissances sym\'etriques du module de Tate du groupe $p$-divisible universel
au-dessus de ${\cal M}_\infty$.

{\rm (ii)} Il serait int\'eressant de d\'eterminer la structure compl\`ete
du $G\times\check G\times {\rm W}_{\Q_p}$-module $H^1_{\eet}({\cal M}_\infty,\Q_p(1))$.
On peut peut-\^etre esp\'erer une r\'eponse \`a la Emerton~\cite{Em08}, i.e. une d\'ecomposition
suivant les repr\'esentations r\'esiduelles $\bar\rho$ de dimension~$2$ de $\G_{\Q_p}$ faisant intervenir
l'anneau des d\'eformations universelles de $\bar\rho$ \`a poids $0$ et $1$ et type fix\'e.
Nous esp\'erons pouvoir revenir sur ce probl\`eme dans un travail ult\'erieur.

{\rm (iii)} Pour faire le lien avec les r\'esultats de Scholze~\cite{ScENS}, il faudrait
consid\'erer la cohomologie de la tour compl\'et\'ee $\widehat{\cal M}_\infty$
(qui est un espace perfecto\"{\i}de).  Si $V$ est une $L$-repr\'esentation de $\G_{\Q_p}$, de dimension~$2$,
absolument irr\'eductible, alors ${\rm Hom}_{L[G]}(\Pi(V)^\dual,H^1_{\eet}(\widehat{\cal M}_\infty,L(1)))$
est une repr\'esentation de $\check G\times W_{\Q_p}$, admissible~\cite{ScENS}
 en tant que repr\'esentation de $\check G$,
qui s'\'etend par continuit\'e en une repr\'esentation
de $\check G\times G_{\Q_p}$, et que l'on peut esp\'erer pouvoir factoriser sous la forme ${\rm JL}_p(\Pi(V))\otimes_L V$,
o\`u ${\rm JL}_p(\Pi(V))$ est une repr\'esentation de $\check G$.  Si c'est le cas, on peut
esp\'erer que $${\rm Hom}_{L[\G_{\Q_p}]}(V,H^1_{\eet}(\widehat{\cal M}_\infty,L(1)))=
{\rm JL}_p(\Pi(V))\widehat\otimes_L \Pi(V)^\dual$$ (i.e. $H^1_{\eet}(\widehat{\cal M}_\infty,L(1))$
r\'ealise simultan\'ement la correspondance de Langlands locale $p$-adique et une correspondance
de Jacquet-Langlands locale $p$-adique), et que ${\rm JL}(V)$ soit l'ensemble des vecteurs lisses de ${\rm JL}_p(\Pi(V))$
si $V$ est supercuspidale, \`a poids $0$ et $1$. On peut de plus esp\'erer une d\'ecomposition
\`a la Emerton pour le module $H^1_{\eet}(\widehat{\cal M}_\infty,L(1))$ tout entier...  
\end{rema}

\subsection{Cohomologies de de Rham et de Hyodo-Kato}
La preuve du th.\,\ref{intro1} fait intervenir la cohomologie pro\'etale de ${\cal M}_\infty$:
on r\'ecup\`ere la cohomologie \'etale en consid\'erant les classes de cohomologie pro\'etale
dont l'orbite sous l'action de $G$ est born\'ee (\'etudier directement la cohomologie \'etale
a l'air un peu d\'elicat, cf.~rem.\,\ref{entier}).
Les th\'eor\`emes de comparaison relient la cohomologie pro\'etale aux cohomologies
de de Rham et de Hyodo-Kato et nous allons commencer par d\'ecrire ces derni\`eres.
Pour \'enoncer le r\'esultat, nous allons privil\'egier le $(\varphi,N,\G_{\Q_p})$-module
$\dpst(V)[1]$ plut\^ot que la repr\'esentation $V$.  

On ne suppose plus que $F=\Q_p$.
Les recettes utilis\'ees plus haut fournissent, 
si $M$ est un $L$-$(\varphi,N,\G_F)$-module (i.e. un $L\otimes \Q_p^{\rm nr}$-module
muni d'actions d'un frobenius semi-lin\'eaire $\varphi$, d'un op\'erateur
$N$ tel que $N\varphi=p\varphi N$ et d'une action semi-lin\'eaire
lisse de $\G_F$), de rang $2$:

$\bullet$ une $L$-repr\'esentation ${\rm WD}(M)$ de ${\rm WD}_{F}$,
de dimension~$2$,

$\bullet$ une $L$-repr\'esentation lisse irr\'eductible ${\rm LL}(M)$ de $G$,

$\bullet$ une $L$-repr\'esentation lisse irr\'eductible (et donc de dimension finie) ${\rm JL}(M)$ de~$\check G$.

\noindent (Si $M=\dpst(V)[1]$, on a ${\rm WD}(V)={\rm WD}(M)$, ${\rm LL}(V)={\rm LL}(M)$ et ${\rm JL}(V)={\rm JL}(M)$.)

On dit
que $M$ est {\it supercuspidal},
si ${\rm WD}(M)$ est irr\'eductible.
Notons que, si $M$ est supercuspidal, les pentes de $\varphi$ sont toutes
\'egales \`a un m\^eme nombre rationnel que nous appellerons {\it la pente de $M$}.

\smallskip
Si $n\in\N$, nous montrons~(cf.~\S\,\ref{GRAB15}) que
${\cal M}_n$ poss\`ede un mod\`ele semi-stable $G\times \check G\times {\rm W}_F$-\'equivariant sur
l'anneau $\O_K$ des entiers d'une extension finie $K$ de $\breve F$:
on choisit un sous-groupe cocompact $\Gamma$ de $G$ op\'erant librement et sans point fixe
sur l'arbre de Bruhat-Tits, de telle sorte que $X_n(\Gamma)=\Gamma\backslash{\cal M}_n$ soit une courbe propre et lisse;
on choisit alors $K$ tel que $X_n(\Gamma)$ ait un mod\`ele semi-stable sur $\O_K$ 
et on fait le produit fibr\'e des mod\`eles semi-stables minimaux de 
$X_n(\Gamma)$ et $\Omega_{\rm Dr}$ sur
$\O_K$ au-dessus de celui de $X(\Gamma)=\Gamma\backslash\Omega_{\rm Dr}$.
 
On dispose des groupes de cohomologie de de Rham $H^1_{\rm dR}({\cal M}_\infty)$
et, gr\^ace \`a ce qui pr\'ec\`ede, de Hyodo-Kato $H^1_{\rm HK}({\cal M}_\infty)$,
de ${\cal M}_\infty$: le premier est un $C$-espace vectoriel,
le second un $(\varphi,N,\G_F)$-module sur $\breve\Q_p$, les deux sont
des ind-fr\'echets, et on a
un isomorphisme naturel~\cite{GK2} ({\it de Hyodo-Kato})
$$\iota_{\rm HK}:C\widehat\otimes_{\breve\Q_p} H^1_{\rm HK}({\cal M}_\infty)\overset{\sim}{\to}
H^1_{\rm dR}({\cal M}_\infty).$$ 

Si $M$ est supercuspidal, de rang $2$,
on pose $$\breve M=\breve\Q_p\otimes_{\Q_p^{\rm nr}}M,\quad M_{\rm dR}=(\Qbar_p\otimes_{\Q_p^{\rm nr}}M)^{\G_F}.$$
Alors $M_{\rm dR}$ est un $L\otimes_{\Q_p}F$-module libre de rang~$2$.
\begin{theo} \label{gabriel}
Il existe un diagramme commutatif naturel de $G$-fr\'echets
$$\xymatrix@R=.6cm@C=.8cm{{\rm Hom}_{L[\check{G}]}({\rm JL(M)}, L\otimes_{\Q_p} H^1_{\rm HK}({\cal M}_{\infty}))\ar[r]^-{{\sim}}\ar[d]^{\iota_{\rm HK}}&
\breve{M}\widehat{\otimes}_{L} {\rm LL}(M)^\dual \ar[d]\\
{\rm Hom}_{L[\check{G}]}\big({\rm JL}(M), L\otimes_{\Q_p} H^1_{\rm dR}({\cal M}_\infty)\big)\ar[r]^-{{\sim}}&C\widehat\otimes_F M_{\rm dR}\widehat{\otimes}_L {\rm LL}(M)^\dual }$$
  la fl\`eche \`a gauche \'etant induite par l'isomorphisme de 
Hyodo-Kato et celle \`a droite par l'identification 
  $M_{\rm dR}\otimes_{F}  C=M\otimes_{\Q_p^{\rm nr}} C=\breve{M}\otimes_{\breve{\Q}_p} C$.
Toutes les fl\`eches commutent \`a l'action de ${\rm W}_F$ et la fl\`eche horizontale du haut commute
en plus \`a $\varphi$.        
\end{theo}
        
\begin{rema} 
{\rm (i)}
La partie pour la cohomologie
de de Rham et pour $F=\Q_p$ a \'et\'e \'etablie dans~\cite{DL}, et notre preuve
du th\'eor\`eme
est une adaptation de celle de~\cite{DL}:
des m\'ethodes globales (i.e.~l'utilisation de courbes de Shimura obtenues comme
quotients de ${\cal M}_n$) montrent 
l'existence d'un plongement du terme de droite dans celui de gauche 
(prop.\,\ref{weak}) et il s'agit de prouver
qu'il n'y a rien de plus dans le membre de gauche.  Dans~\cite{DL}, cela se fait en utilisant
la correspondance de Langlands $p$-adique pour ${\rm GL}_2(\Q_p)$.
Comme une telle correspondance est encore hypoth\'etique si $F\neq\Q_p$,
nous utilisons (\no\ref{GAB7}), \`a la place, l'isomorphisme avec la cohomologie
de la tour de Lubin-Tate (th.~\ref{intro-comp} ci-dessous) pour laquelle
on peut utiliser la th\'eorie de Lubin-Tate non ab\'elienne \cite{Carayol2, HT}
pour contr\^oler la dimension des invariants par des sous-groupes ouverts compacts de $G$.

{\rm (ii)} Lue Pan~\cite{Pan} a prouv\'e l'\'enonc\'e ci-dessus, par des m\'ethodes
purement locales, pour le premier \'etage ${\cal M}_1$ de la tour.  La preuve repose
sur la construction d'un mod\`ele semi-stable explicite.
\end{rema}

Soient $({\rm LT}_j)_{j\geq 0}$ la tour de Lubin-Tate et ${\rm LT}_\infty$ le syst\`eme projectif
des ${\rm LT}_j$.
Fixons une uniformisante $\varpi$ de $F$ et notons
simplement
${\rm LT}_\infty^\varpi$ et ${\cal M}_\infty^\varpi$ les quotients de
${\rm LT}_\infty$ et ${\cal M}_\infty$ par
       l'action de $\varpi$ vu comme \'el\'ement du centre de $\check G$.

\begin{theo}\label{intro-comp}
On a
un isomorphisme naturel 
$$ H^1_{\rm dR, c}({\rm LT}_\infty^\varpi)\simeq  H^1_{\rm dR, c}({\cal M}_\infty^\varpi)$$
qui munit les deux membres d'une structure de
$C[G\times \check{G}]$-module lisse, admissible d\'ej\`a en tant que $G$-module.
\end{theo}

\begin{rema}
{\rm (i)} 
Les limites projectives compl\'et\'ees $\widehat{\rm LT}_\infty$ et $\widehat{\cal M}_{\infty}$
des tours de Lubin-Tate et Drinfeld 
sont des espaces perfecto\"{\i}des~\cite[th. 6.5.4, 7.2.3]{SW},
isomorphes, 
mais l'isomorphisme du th.\,\ref{intro-comp}
 est quand m\^eme un peu surprenant car les espaces perfecto\"{\i}des n'ont pas
de cohomologie de de Rham digne de ce nom (extraire des racines d'ordre $p^\infty$ et compl\'eter
rend la d\'erivation probl\'ematique).

{\rm (ii)}
Il est raisonnable de penser que ce r\'esultat s'\'etend en dimension quelconque.
 Si on remplace la cohomologie de de Rham \`a support compact par la cohomologie $\ell$-adique 
\`a support compact (avec $\ell\neq p$), on dispose d'un r\'esultat tr\`es g\'en\'eral de 
dualit\'e d\^u \`a Scholze \cite[prop. 5.4]{survey}. 
\end{rema}

\subsection{Cohomologie pro\'etale}
Soit $M$ un $L$-$(\varphi,N,\G_F)$-module supercuspidal, de rang $2$. 
Si $Z$ est un $\Q_p[\check G]$-module, on pose 
$$Z[M]={\rm Hom}_{L[\check G]}({\rm JL}(M),L\otimes_{\Q_p}Z).$$

Le r\'esultat suivant fournit une description de $H^1_{\proet}({\cal M}_\infty)$
comme $G\times W_F$ repr\'esentation.  C'est le point de d\'epart de la preuve du th.\,\ref{intro1}.
On pose $$X_{\rm st}^+(M)=(\bst^+\otimes_{\Q_p^{\rm nr}} M)^{N=0,\varphi=p}.$$

\begin{theo}\label{diagfond}
Il existe un diagramme commutatif de $G\times W_F$-fr\'echets, \`a lignes exactes,
$$\xymatrix@C=.6cm@R=.6cm{0 \ar[r]&\ \mathcal{O}({\cal M}_\infty)[M]\ar[r]^-{{\rm exp}}\ar@{=}[d]&
H^1_{\proet}({\cal M}_\infty, L(1))[M]\ar[d]^{{\rm dlog}}\ar[r]&
X_{\rm st}^+(M)\widehat{\otimes}_L{\rm LL}(M)^\dual \ar[d]^{\theta}\ar[r]&0\\
0\ar[r]& \mathcal{O}({\cal M}_\infty)[M] \ar[r]^-d&\Omega^1({\cal M}_\infty)[M]\ar[r]&
(C \otimes_F M_{\rm dR})\widehat{\otimes}_L {\rm LL}(M)^\dual \ar[r]&0
}$$
  \end{theo}

Le lecteur trouvera des indications sur la preuve de cet \'enonc\'e au \S\,\ref{preuves}.

\subsection{La conjecture de Breuil-Strauch}
On suppose $F=\Q_p$ dans ce paragraphe.  Si $M$ est supercuspidal de rang $2$,
alors 
$M_{\rm dR}=(\Qbar_p\otimes_{\Q_p^{\rm nr}} M)^{\G_{\Q_p}}$ 
est un $L$-espace de rang~$2$.
Si ${\cal L}$ est une $L$-droite de $M_{\rm dR}$, on d\'efinit
la repr\'esentation $V_{M,{\cal L}}$ de $\G_{\Q_p}$ par $$V_{M,{\cal L}}={\rm Ker}\big((\bcris^+\otimes_{\Q_p^{\rm nr}} M)^{\varphi=p}\to
C\otimes_{\Q_p} (M_{\rm dR}/{\cal L})\big).$$
Il r\'esulte de~\cite{CF} que $V_{M,{\cal L}}$
est une repr\'esentation supercuspidale \`a poids $0$ et $1$ et toute telle repr\'esentation
est de la forme $V_{M,{\cal L}}$ pour un unique couple $(M,{\cal L})$.

On a alors le r\'esultat suivant~\cite{DL}, conjectur\'e par Breuil et Strauch (non publi\'e),
et qui constitue le premier r\'esultat tangible reliant la correspondance de Langlands locale
$p$-adique \`a la tour de Drinfeld.
\begin{prop}\label{dl10}
L'inclusion de ${\cal L}$ dans $M_{\rm dR}$ donne naissance \`a un diagramme commutatif
de $G$-fr\'echets, \`a lignes exactes, o\`u $\Pi(V_{M,{\cal L}})^{\rm an}$ est l'espace
des vecteurs localement analytiques de $\Pi(V_{M,{\cal L}})${\rm :}
$$\xymatrix@C=.5cm@R=.6cm{0 \ar[r]&\ \mathcal{O}({\cal M}_n)[M]\ar[r]\ar@{=}[d]&
C\widehat\otimes_{\Q_p}(\Pi(V_{M,{\cal L}})^{\rm an})^\dual\ar[r]\ar[d]&
(C\otimes_{\Q_p}{\cal L})\widehat\otimes{\rm LL}(M)^\dual \ar[d]\ar[r]&0\\
0\ar[r]& \mathcal{O}({\cal M}_n)[M] \ar[r]^-d&\Omega^1({\cal M}_n)[M]\ar[r]&
(C \otimes_{\Q_p} M_{\rm dR})\widehat{\otimes}_L{\rm LL}(M)^\dual \ar[r]&0
}$$
\end{prop}

Pour d\'eduire le th.\,\ref{intro1} de cet \'enonc\'e,
on utilise le fait que ${\rm Hom}_{L[\G_{\Q_p}]}(V_{M,{\cal L}},L\otimes C)$ et
${\rm Hom}_{L[\G_{\Q_p}]}(V_{M,{\cal L}},X_{\rm st}^+(M))$ sont tous les deux de dimension~$1$ sur $L$,
et on compare le diagramme obtenu en appliquant ${\rm Hom}_{L[\G_{\Q_p}]}(V_{M,{\cal L}},-)$ au
diagramme du th.\,\ref{diagfond} \`a celui
de la prop.\,\ref{dl10}.
On en d\'eduit que
$${\rm Hom}_{L[\check G\times {\rm W}_{\Q_p}]}
({\rm JL}(M)\otimes V_{M,{\cal L}},H^1_{\proet}({\cal M}_\infty,\Q_p(1)))=
(\Pi(V_{M,{\cal L}})^{\rm an})^\dual.$$
On utilise alors
le fait (prop.\,\ref{fonda1})
que $H^1_{\eet}({\cal M}_\infty,\Q_p(1)))$ est l'ensemble des \'el\'ements
de $H^1_{\proet}({\cal M}_\infty,\Q_p(1)))$ dont l'orbite sous l'action de $G$ est born\'ee,
et que $\Pi(V_{M,{\cal L}})$ est le compl\'et\'e universel de $\Pi(V_{M,{\cal L}})^{\rm an}$ d'apr\`es~\cite{CD},
ce qui, par dualit\'e, se traduit par le fait que $\Pi(V_{M,{\cal L}})^\dual$ est l'ensemble des
\'el\'ements $G$-born\'es de $(\Pi(V_{M,{\cal L}})^{\rm an})^\dual$;
on en d\'eduit que
$${\rm Hom}_{L[\check G\times {\rm W}_{\Q_p}]}
({\rm JL}(M)\otimes V_{M,{\cal L}},H^1_{\eet}({\cal M}_\infty,\Q_p(1)))=
\Pi(V_{M,{\cal L}})^\dual,$$
ce qui fournit la moiti\'e du th.\,\ref{intro1}; le reste repose sur une combinatoire difficile
\`a r\'esumer (cf.~\no\ref{emer5.2}), 
mais qui utilise le fait que, si $V$ est irr\'eductible et n'est pas de la forme $V_{M,{\cal L}}$, alors
$\dim_L{\rm Hom}_{L[\G_{\Q_p}]}(V,L\otimes C)>\dim_L{\rm Hom}_{L[\G_{\Q_p}]}(V,X_{\rm st}^+(M))$ sauf si les deux termes
sont nuls auquel cas il n'y a rien \`a faire.

\subsection{Le diagramme fondamental}\label{preuves}
Il y a plusieurs mani\`eres d'\'etablir l'existence
du diagramme du th.\,\ref{diagfond}.  On en pr\'esente deux ci-dessous:
la premi\`ere est d\'evelopp\'ee dans~\cite{CDN1} o\`u l'on trouvera aussi une troisi\`eme approche
utilisant l'int\'egration $p$-adique sur les courbes~\cite{ColemanRCC, CI, Cz-periodes, Cz-inte},
 et la seconde est celle utilis\'ee dans
l'article.
\subsubsection{Cohomologie syntomique}
L'approche la plus naturelle est probablement de passer par la cohomologie syntomique
car le lien avec le complexe de de Rham est inscrit dans la d\'efinition
m\^eme de la cohomologie syntomique.
Cette approche se g\'en\'eralise bien en dimension sup\'erieure~\cite{CDN2}.

Soit $X$ une courbe analytique sur $C$ avec un mod\`ele semi-stable ${\cal X}$ sur l'anneau des entiers $\O_K$
d'un sous-corps ferm\'e de $C$ de valuation discr\`ete.
On suppose que $X$ n'a qu'un nombre fini de composantes connexes.
On note $\rg_{\rm cris}(X)$ le complexe calculant la cohomologie cristalline absolue
de ${\cal X}\times \O_C$.
Par d\'efinition de la cohomologie syntomique \cite{FM, NN}, 
on dispose
d'un diagramme commutatif
de triangles distingu\'es:
$$
\xymatrix@R=.6cm{
\rg_{\rm syn}(X,\Q_p(1))\ar[r]\ar[d]^{{\beta}} & \rg_{\rm cris}(X)^{\varphi=p}
\ar[d]^-{\gamma}\ar[r] & \rg_{\rm cris}(X)/{\rm Fil}^1\ar@{=}[d]\\
(0\to\Omega^1) \ar[r] & (\O\to\Omega^1) \ar[r] & (\O\to 0)
}
$$
dans lequel:

$\bullet$  la ligne du haut est la d\'efinition
de $\rg_{\rm syn}(X,\Q_p(1))$, celle du bas est \'evidente.

$\bullet$ L'application $\gamma$ est induite par la compos\'ee
de $\iota_{\rm can}:\rg_{\rm cris}(X)^{\varphi=p}\to \rg_{\rm cris}(X)$ et de $\theta:\rg_{\rm cris}(X)\to
\rg_{\rm dR} (X)$.

$\bullet$ L'application $\beta$ est induite par la compos\'ee de
l'application naturelle
$\tilde\beta:\rg_{\rm syn}(X,\Q_p(1))\to {\rm Fil}^1\rg_{\rm cris}(X)$ provenant
de la d\'efinition de $\rg_{\rm syn}(X,\Q_p(1))$
et de 
$\theta:{\rm Fil}^1\rg_{\rm cris}(X)\to (0\to\Omega^1)$.

En passant \`a la cohomologie, et en utilisant l'isomorphisme
entre la cohomologie (pro)\'etale et la cohomologie syntomique~\cite{CN,Ts},
cela fournit le diagramme commutatif:
$$
\xymatrix@C=.5cm@R=.5cm{
0\ar[r] & C\otimes\Z[\pi_0(X)]\ar@{=}[d]\ar[r]&
\O(X)\ar[r]\ar@{=}[d] & H^1_{\proet}(X,\Q_p(1))\ar[d]^{\beta} \ar[r] & 
{\rm HK}^1_1(X)\ar[r]\ar[d]^{\gamma} & H^1(X,\O)\ar@{=}[d]\\
0\ar[r]&C\otimes\Z[\pi_0(X)]\ar[r]&
 \O(X) \ar[r]^-d & \Omega^1(X) \ar[r] & H^1_{\rm dR}(X)\ar[r] & H^1(X,\O)
}
$$
o\`u l'on a not\'e ${\rm HK}^1_1(X)$ le groupe $H^1(\rg_{\rm cris}(X)^{\varphi=p})$.

Si $X$ est un affino\"{\i}de ou une courbe de Stein (cas qui nous int\'eresse), on a
$H^1(X,\O)=0$. Si $X$ est une courbe propre, le th\'eor\`eme
de comparaison semi-stable fournit un isomorphisme ${\rm HK}^1_1(X)\cong
(\bst^+\otimes H^1_{\rm HK}(X))^{N=0,\varphi=p}$. Par contre,
si $X$ est un affino\"{\i}de, le groupe ${\rm HK}^1_1(X)$ est fort peu sympathique: en particulier,
c'est un espace topologique non s\'epar\'e.
Heureusement, une courbe de Stein se comporte plus comme une courbe propre que comme un
affino\"{\i}de, et on peut montrer~\cite{CDN1}, en faisant un peu de gymnastique, que l'on a
encore un isomorphisme
$${\rm HK}^1_1(X)\cong
(\bst^+\widehat\otimes H^1_{\rm HK}(X))^{N=0,\varphi=p},$$ 
si $X$ est une courbe de Stein
(la diff\'erence avec le cas propre est que $H^1_{\rm HK}(X)$ est une limite projective
d'espaces de dimension finie, et donc qu'il faut prendre un produit tensoriel compl\'et\'e).

Pour en d\'eduire le th.\,\ref{diagfond}, il suffirait
d'appliquer ce qui pr\'ec\`ede \`a $X=\mv_n$ (cf.~\S\,\ref{notasup} pour le passage
de $\mv_n$ \`a ${\cal M}_n$), d'appliquer le foncteur $Z\mapsto Z[M]$,
et d'utiliser le th.\,\ref{gabriel} pour faire appara\^{\i}tre la colonne de droite
(et le fait que $\check G$ agit par la norme r\'eduite sur $\pi_0(\mv_n)$ 
et donc que $(\Z[\pi_0(X)])[M]=0$).

\begin{rema}\label{entier}
Les th\'eor\`emes de comparaison \'etale-syntomique sont valables (\`a torsion born\'ee pr\`es)
pour des coefficients entiers (i.e.~$\Z_p(1)$ au lieu de $\Q_p(1)$).
On peut donc les utiliser pour \'etudier la cohomologie \'etale.  La premi\`ere ligne
du diagramme du th.\,\ref{diagfond}
devient alors:
$$0\to H^1_{\eet}({\cal M}_\infty,L(1))[M]\to X_{\rm st}^+(M)\otimes_L \widehat{{\rm LL}(M)}^\dual\to
\big(L\otimes H^1({\cal M}_n^\circ,\O)\big)[M],$$
o\`u $\widehat{{\rm LL}(M)}$ est le compl\'et\'e unitaire universel de ${\rm LL}(M)$,
$n$ est tel que $1+\varpi_D^n\O_D$ agisse trivialement sur ${\rm JL}(M)$, et
${\cal M}_n^\circ$ est un mod\`ele semi-stable $\check G\times G$-\'equivariant de
${\cal M}_n$. Comme il existe des repr\'esentations irr\'eductibles $V$ de $\G_{\Q_p}$,
 de dimension arbitraire, telles que ${\rm Hom}_{\G_{\Q_p}}(V,X_{\rm st}^+(M))\neq 0$,
le th.\,\ref{intro1} montre que
$H^1({\cal M}_n^\circ,\O)$ est tr\`es loin d'\^etre nul (au moins si $F=\Q_p$) contrairement \`a
$H^1({\cal M}_n,\O)$ (et $\O({\cal M}_n^\circ)=\O_C\otimes\Z[\pi_0({\cal M}_n)]$, ce qui explique le $0$ 
\`a gauche dans la suite exacte ci-dessus).
\end{rema}  

\subsubsection{Cohomologie des anneaux de Fontaine relatifs}
Une autre approche possible, et c'est celle utilis\'ee dans l'article car elle se
marie bien avec les m\'ethodes globales utilis\'ees pour prouver le th.\,\ref{gabriel},
consiste \`a utiliser les anneaux de Fontaine relatifs et les th\'eor\`emes de
comparaison de Faltings et Scholze \cite{Fa1,RAV}

On a un diagramme de faisceaux pour la topologie pro\'etale, \`a lignes exactes:
$$\xymatrix@R=.5cm@C=.6cm{0\ar[r]&\Q_p(1)\ar[r]\ar[d]&({\mathbb B}_{\rm cris}^+)^{\varphi=p}
\ar[r]\ar[d]&\widehat\O\ar@{=}[d]\ar[r]&0\\
0\ar[r]&\widehat\O(1)\ar[r]&{\mathbb B}_{\rm dR}^+/t^2\ar[r]&\widehat\O\ar[r]&0}$$
On d\'eduit de~\cite[Lemma\,3.24]{survey} que $\Omega^1(X)(-1)$ est naturellement
un quotient de $H^1_{\proet}(X,\widehat\O)$; si $X$ est une courbe
de Stein, on a le r\'esultat plus pr\'ecis suivant:
$$H^0_{\proet}(X,\widehat\O)=\O(X)\quad{\rm et}\quad H^1_{\proet}(X,\widehat\O)=\Omega^1(X)(-1).$$
D\'efinissons le groupe $\tHK(X)$ par:
\begin{align*}
\widetilde{\rm HK}(X)
={\rm Ker}\big(H^1_{\proet}(X,({\mathbb B}_{\rm cris}^+)^{\varphi=p})\to
H^1_{\proet}(X,\widehat\O)\to \Omega^1(X)(-1)\big).
\end{align*}
Si $X$ est une courbe propre,
on d\'eduit des th\'eor\`emes de comparaison
un isomorphisme (cf. prop. \ref{Shimura})
$$\tHK(X)\cong (\bst^+\otimes H^1_{\rm HK}(X))^{N=0,\varphi=p}.$$
Si $X$ est une courbe de Stein, on d\'eduit de ce qui pr\'ec\`ede
un diagramme commutatif
\`a lignes exactes: 
$$\xymatrix@R=.5cm@C=.6cm{
 0\ar[r]&C\otimes\Z[\pi_0(X)]\ar@{=}[d]\ar[r]&
\mathcal{O}(X)\ar[r]^-{\rm exp}\ar@{=}[d]& H^1_{\proet}(X, \Q_p(1))\ar[r]\ar[d]^-{\rm dlog}&
\tHK(X)\ar[r]\ar[d]^-{\rm \iota_{\rm can}}&0\\
0\ar[r]&C\otimes\Z[\pi_0(X)]\ar[r]&
\mathcal{O}(X)\ar[r]^-{d} & \Omega^1(X) \ar[r]^-{\pi_{\rm dR}}&
H^1_{\rm dR}(X)\ar[r]&0
}$$           

Pour en d\'eduire le th.\,\ref{diagfond}, on utilise ce qui pr\'ec\`ede pour $X=\mv_n$
et on est r\'eduit \`a calculer $\tHK(\mv_n)[M]$.
Pour cela, on commence par prouver que $\iota_{\rm can}$ identifie $\tHK(\mv_n)[M]$ 
\`a un sous-espace ferm\'e de $H^1_{\rm dR}(\mv_n)[M]$; cela implique, en utilisant le th.\,\ref{gabriel}
et la prop.\,\ref{isot},
qu'il existe $Z$ tel que $\tHK(\mv_n)[M]\cong Z\widehat\otimes {\rm LL}(M)^\dual$.
Pour calculer $Z$, on utilise les quotients de $\mv_n$ par des sous-groupes cocompacts
bien choisis de $G$ (de congruence dans une alg\`ebre de quaternions sur un corps de nombres dont un
compl\'et\'e est $F$), les th\'eor\`emes de compatibilit\'e local-global de Carayol~\cite{Carayol}
 et Saito~\cite{Saito},
 et le fait que 
$\tHK(X)\cong (\bst^+\otimes H^1_{\rm HK}(X))^{N=0,\varphi=p}$,
si $X$ est une courbe propre.

\subsubsection*{Remerciements}
Certaines parties de ce projet ont \'et\'e r\'ealis\'ees lors de s\'ejours au BICMR de P\'ekin (P.C.),
au SCMS de l'universit\'e de Fudan \`a Shanghai (W.N. et P.C.), au Tata Institute de Bombay (W.N. et P.C.), \`a l'institut Mittag-Leffler (W.N.) ou au KIAS de S\'eoul (G.D.), 
et nous voudrions remercier ces institutions pour leur hospitalit\'e.
Nous voudrions aussi remercier 
Matt Baker et Michael Temkin pour leurs r\'eponses d\'etaill\'ees
\`a nos questions sur les mod\`eles des courbes analytiques, ainsi que 
Antoine Ducros, Laurent Fargues, Luc Illusie, Arthur-C\'esar Le Bras, Qing Liu, Lue Pan, Takeshi Saito et Shanwen Wang
pour leurs remarques, r\'ef\'erences ou explications.

\section{La cohomologie du demi-plan de Drinfeld et la steinberg}\label{DEMI}
Soient $K\subset L$ des sous-corps ferm\'es de $C$.  Si $X_K$ est une
vari\'et\'e analytique d\'efinie sur $K$, on note $X_L$ son extension des scalaires \`a $L$,
et simplement $X$ l'extension des scalaires \`a $C$.
Alors $\G_K={\rm Aut}_{\rm cont}(C/K)$ agit sur $X$ et sur tous les objets
qui s'en d\'eduisent.
\subsection{Affino\"{\i}des de $\piqp$}

Si $Y$ est un affino\"{\i}de\footnote{Tous nos affino\"{\i}des sont r\'eduits et munis de la valuation
spectrale.} 
d\'efini sur $C$, on note:

$\bullet$ $\O^+(Y)$ le sous-anneau de $\O(Y)$ des fonctions \`a valeurs enti\`eres,

$\bullet$ $\O^{++}(Y)$ l'id\'eal ${\goth m}_C\otimes_{\O_C}\O^+(Y)$ de $\O^+(Y)$,

$\bullet$
 $\O(Y)^{\dual\dual}$ le sous-groupe $1+\O^{++}(Y)$ de $\O(Y)^\dual$.

$\bullet$ $v_Y$ la valuation spectrale sur $\O(Y)$.

\noindent
On a donc 
\begin{align*}\O^+(Y)=\{f\in \O(Y),\ v_Y(f)\geq 0\},\quad&
\O^{++}(Y)=\{f\in \O(Y),\ v_Y(f)>0 \},\\
\O(Y)^{\dual\dual}= \{f\in \O(&Y),\ v_Y(f-1)> 0\}.
\end{align*}

     Si $x\in \mathbf{P}^1(C)$ et $r\geq 0$, on note $B^-(x,r)$
  la boule ouverte de centre $x$ et de valuation $r$, avec la convention que 
    $B^-(x,r)=\{z,\  v_p(z-x)>r\}$ si $x\in \O_C$ et 
$B^-(x,r)=\{z,\  v_p(z^{-1}-x^{-1})>r\}$ pour 
    $x\in \mathbf{P}^1(C)-\O_C$.
Pour avoir des formules uniformes, 
notons 
$z-x$ le param\`etre local $z^{-1}-x^{-1}$ de $B^-(x,r)$,
si $x\in\mathbf{P}^1(C)-\O_C$.
   
Soit $U$ un affino\"{\i}de connexe de $\piqp$.  Il existe 
des boules ouvertes $B^-(x,r_x)$, pour $x\in X$, o\`u $X$ est fini et $X\not\ni\infty$, telles que
$$U=\mathbf{P}^1-\bigcup_{x\in X} B^-(x, r_x).$$
   D'apr\`es la proposition 2.5.10 de \cite{FDP},
on a le r\'esultat suivant.
 \begin{prop}\label{unites0}
Tout $f\in\O(U)^\dual$ peut s'\'ecrire
sous la forme
\begin{center}
\hfill$f=c\,u\prod_{x\in X}(z-x)^{m_x},\hskip.4cm
 {\text {avec $c\in C^\dual$, $u\in \O(U)^{\dual\dual}$,
 $m_x\in\Z$ et $\sum_{x\in X} m_x=0$.}}$
\end{center}
Cette \'ecriture est unique \`a $c\mapsto u_0c$ et $u\mapsto u_0^{-1}u$ pr\`es, avec
$u_0\in 1+{\goth m}_C$.
 \end{prop}  
   
Notons ${\rm LC}(X,\Z)$ le $\Z$-module des fonctions $\phi:X\to \Z$ et
${\rm LC}(X,\Z)^\dual$ son $\Z$-dual (c'est l'espace des mesures sur $X$
\`a valeurs dans $\Z$). Un \'el\'ement $\mu$ de
${\rm LC}(X,\Z)^\dual$ est \'equivalent \`a la donn\'ee de $m_x\in\Z$, pour $x\in X$:
la mesure associ\'ee est $\phi\mapsto\int_X\phi\,\mu=\sum_{x\in X}m_x\phi(x)$ ou, de mani\`ere
\'equivalente, $\mu=\sum_xm_x{\rm Dir}_x$ o\`u ${\rm Dir}_x$ est la masse de Dirac en $x$.
Si $\mu\in {\rm LC}(X,\Z)^\dual$, on pose $\int_X\mu=\int_X{\bf 1}_X\,\mu$.
Le r\'esultat suivant est alors une simple traduction
de la proposition~\ref{unites0} ci-dessus.
   \begin{coro}\label{unites}
On a une suite exacte de groupes ab\'eliens
$$\xymatrix@C=.5cm{
0\ar[r]&\O(U)^{\dual\dual}/(1+{\goth m}_C)\ar[r]& \O(U)^\dual/C^\dual
\ar[r]& {\rm LC}(X,\Z)^\dual\ar[rr]^-{\mu\mapsto \int_X\mu}&&\Z\ar[r]&0},$$
o\`u, si $x,y\in X$, l'image de $\frac{z-x}{z-y}\in\O(U)^\dual$ dans ${\rm LC}(X,\Z)^\dual$
est ${\rm Dir}_x-{\rm Dir}_y$.
        \end{coro}
   
\subsection{Le demi-plan de Drinfeld}
    Rappelons que $\Omega_{{\rm Dr},F}$ d\'esigne l'espace analytique rigide 
 $\piqp_{F}-\piqp(F)$, muni de l'action naturelle de $G$ par homographies,
et que $\Omega_{{\rm Dr}}=C\times_F\Omega_{{\rm Dr},F}$.

Si $n\geq 1$, soient
 $$\mathcal{P}_n=\mathbf{P}^1(\O_F/\varpi^n)\quad{\rm et}\quad
    U_n=\piqp-\bigcup_{x\in \mathcal{P}_{n}} B^-(x, nv_p(\varpi)).$$
Les $U_n$ sont des ouverts affino\"{\i}des de $\Omega_{\rm Dr}$
dont la r\'eunion (croissante) est $\Omega_{\rm Dr}$.
On dispose d'une application de r\'eduction $G$-\'equivariante 
 $r: \Omega_{\rm Dr}\to \mathcal{T}$, o\`u $\mathcal{T}$ est l'arbre de Bruhat-Tits de $G$. 
 Alors $U_n$ est aussi 
l'image inverse par 
 $r$ du sous-arbre de $\mathcal{T}$ form\'e des sommets et ar\^etes \`a distance au plus $n$ du sommet central. 
Notons que $U_n$ est l'extension des scalaires \`a $C$ d'un affino\"{\i}de
$U_{n,F}$ de $\piqp_F$, et que
$\Omega_{{\rm Dr},F}$ est la r\'eunion des $U_{n,F}$.

Soit $({\cal X}_n)_{n\in\N}$ la suite de mod\`eles formels semi-stables
de $\piqp_{F}$ sur $\O_F$ d\'efinie de la mani\`ere suivante:
${\cal X}_0=\piqp_{\O_F}$ et ${\cal X}_{n}$ est obtenu \`a partir de ${\cal X}_{n-1}$
en \'eclatant les points lisses de ${\cal X}_{n-1}(k_F)$.
La fibre sp\'eciale de ${\cal X}_n$ est un arbre de $\piqp$, de rayon $n$, centr\'e en
le $\piqp$ initial, chaque $\piqp$ sauf les $\piqp$ extr\'emaux (ce sont ceux obtenus
par \'eclatement des points lisses de ${\cal X}_{n-1}(k_F)$, ce sont aussi ceux
\`a distance $n$ du centre) rencontrant $q+1$ autres $\piqp$.

On note ${\cal U}_n$ le sous-sch\'ema formel de ${\cal X}_n$ obtenu
en retirant les points lisses de ${\cal X}_{n}(k_F)$.  
La fibre g\'en\'erique de ${\cal U}_n$ est l'affino\"{\i}de $U_{n,F}$ ci-dessus.
Alors
${\cal U}_n$ est un ouvert de ${\cal U}_{n+1}$ pour tout $n$ et 
la r\'eunion croissante ${\cal U}_\infty$ de ${\cal U}_n$ est un mod\`ele semi-stable
de $\Omega_{{\rm Dr},F}$. L'action de $G$ sur $\Omega_{{\rm Dr},F}$ se prolonge
\`a ce mod\`ele (car $G$ permute les pr\'eimages des composantes irr\'eductibles
de la fibre sp\'eciale de ${\cal U}_\infty$).

\begin{lemm}\label{gab1}
Si $f\in\O^+(U_{n+1})$ s'annule sur $U_n(C)$, alors
$f\in\varpi \O^+(U_n)$.
\end{lemm}
\begin{proof}
${\cal U}_n$ est inclus dans la r\'eunion des composantes
irr\'eductibles non extr\'emales de ${\cal U}_{n+1}$ qui sont des $\piqp$.
On en d\'eduit que $f$ est constante sur chacune de ces composantes irr\'eductibles,
et donc aussi sur ${\cal U}_n$.  Comme on a suppos\'e que $f$
s'annule sur $U_n(C)$, cela implique que $f=0$ sur ${\cal U}_n$
et donc $f\in \varpi \O^+(U_n)$.
\end{proof}

 Notons $A_n=\mathcal{O}(U_{n,C})$. 
Le lemme suivant rassemble un certain nombre de r\'esultats concernant ces anneaux. On renvoie le lecteur \`a la section 2.7 de \cite{FDP}
 (en particulier la preuve du th\'eor\`eme~2.7.6)  pour plus de d\'etails. 
       
\begin{lemm}\label{technique} On a 
       
{\rm (i)} ${\rm Pic}(A_n)=0$ pour tout $n$. 
       
{\rm (ii)} $\R^1\varprojlim\,A_n^\dual =0$. 
       
{\rm (iii)} $\varprojlim\, A_n^{**}=1+{\goth m}_C$ et 
$\R^1\varprojlim\, A_n^{**}=0$.       
\end{lemm}
       
       \begin{proof} (i) est une simple cons\'equence du fait que $U_{n, C}$ sont des affino\"{\i}des connexes de $\mathbf{P}^1_{C}$, donc les $A_n$ sont des anneaux principaux. 
       
       (ii) Cela d\'ecoule de la surjectivit\'e de l'application 
     $\prod_{n} A_n^\dual \to \prod_{n} A_n^\dual $ envoyant $(f_1,f_2,...)$ sur 
     $(f_1/f_2, f_2/f_3,...)$, qui elle-m\^eme est une cons\'equence de la description de 
     $A_n^\dual $ fournie par la proposition \ref{unites0}. 
     
(iii)  Fixons un point
$Q\in U_1$. Si $(f_n)_n\in \varprojlim A_n^{**}$.
on peut \'ecrire $f_n=a\cdot (1+g_n)$ avec 
$a=f_1(Q)=f_n(Q)$ et donc $g_n(Q)=0$, et $v_{U_{n}}(g_n)>0$ pour tout $n$. 
On en d\'eduit, en utilisant le lemme~\ref{gab1}, que 
$v_{U_{n}}(g_n)\geq v_p(\varpi)$ pour tout $n$. Une r\'ecurrence imm\'ediate montre que 
$v_{U_{n}}(g_n)\geq kv_p(\varpi)$ pour tous $n$ et $k$, d'o\`u 
$g_n=0$ pour tout $n$ et $\varprojlim A_n^{**}=1+{\goth m}_C$. Un argument similaire montre que si 
$f_n\in A_n^{**}$, le produit $\prod_{k\geq n} \frac{f_k}{f_k(Q)}$ converge dans 
$A_n^{**}$, ce qui permet de d\'emontrer la nullit\'e de $\R^1\varprojlim\, A_n^{**}$.
\end{proof}

\Subsection{La steinberg et ses variantes}
\subsubsection{Fonctions sur $\piqp(F)$}
Si $\Lambda$ est un anneau topologique,
on note ${\rm LC}(\piqp(F),\Lambda)$ (resp.~${\cal C}(\piqp(F),\Lambda)$) l'espace 
des fonctions localement constantes (resp.~continues) 
sur $\piqp(F)$ \`a valeurs dans $\Lambda$.
Comme $\piqp(F)$ est la limite projectives des ${\cal P}_n=\piqp(\O_F/\varpi^n)$ qui sont des ensembles finis,
on a ${\rm LC}(\piqp(F),\Lambda)=\varinjlim {\rm LC}({\cal P}_n,\Lambda)$, et
${\rm LC}(\piqp(F),\Lambda)$ est dense dans ${\cal C}(\piqp(F),\Lambda)$ si $\Lambda$ est un
anneau topologique.

Si $L$ est un sous-corps ferm\'e de $C$ contenant $F$, on note
${\rm LA}(\piqp(F),L)$ l'espace
des fonctions localement $F$-analytiques sur $\piqp(F)$ \`a valeurs dans $L$.

Comme $\piqp(F)$ est muni d'une action de $G$ (par homographies), cela
munit les espaces ci-dessus d'une action de $G$ qui est lisse, (i.e.~localement constante)
sur ${\rm LC}(\piqp(F),\Lambda)$, continue sur~${\cal C}(\piqp(F),\Lambda)$
et localement $F$-analytique sur ${\rm LA}(\piqp(F),L)$.  
On note ${\rm St}_\Lambda^{\rm lisse}$, ${\rm St}^{\rm cont}_\Lambda$
et ${\rm St}^{\rm an}_L$ les steinbergs lisse, continue et localement $F$-analytique,
quotients respectifs des espaces ci-dessus par les fonctions constantes: ce sont
des $\Lambda$ (resp. $L$)-repr\'esentations de $G$.

Si $L$ est un sous-corps ferm\'e de $C$, alors 
${\rm St}^{\rm cont}_L$ est un banach, et si $L$ contient~$F$, alors ${\rm St}^{\rm an}_L$
est une limite inductive compacte de banachs. De plus:

$\bullet$
${\rm St}^{\rm an}_L$
(resp.~${\rm St}^{\rm lisse}_L$)
est l'ensemble des vecteurs localement $F$-analytiques (resp.~lisses)
de ${\rm St}^{\rm cont}_L$,

$\bullet$ 
${\rm St}^{\rm cont}_L$ est le compl\'et\'e unitaire universel
de ${\rm St}^{\rm an}_L$ et de ${\rm St}^{\rm lisse}_L$,

$\bullet$
${\rm St}^{\rm lisse}_L$ est ferm\'ee dans ${\rm St}^{\rm an}_L$
et la topologie de ${\rm St}^{\rm lisse}_L$ induite par celle de ${\rm St}^{\rm an}_L$
est la topologie convexe la plus fine.

\subsubsection{Dualit\'e}
Si $\Lambda$ est un anneau topologique (la topologie discr\`ete n'est pas exclue), et si
$M$ est un $\Lambda$-module topologique, on note $M^\dual$ le $\Lambda$-module
${\rm Hom}_{\rm cont}(M,\Lambda)$, que l'on munit de la topologie faible.

Alors $({\rm St}^{\rm lisse}_\Lambda)^\dual$
est la limite projective des
${\rm Ker}:{\rm LC}({\cal P}_n,\Lambda)^\dual\to\Lambda$, chacun
des ${\rm LC}({\cal P}_n,\Lambda)^\dual$ \'etant muni de la topologie discr\`ete.

Si $L$ est un sous-corps ferm\'e de $C$, alors
$({\rm St}^{\rm cont}_L)^\dual$, $({\rm St}^{\rm an}_L)^\dual$ et $({\rm St}^{\rm lisse}_L)^\dual$
sont respectivement les espaces des mesures, distributions et distributions alg\'ebriques, de masse
totale $0$ sur $\piqp(F)$.  De plus:

$\bullet$
$({\rm St}^{\rm cont}_L)^\dual$ est un dual faible de banach
et $({\rm St}^{\rm an}_L)^\dual$ et $({\rm St}^{\rm lisse}_L)^\dual$ sont des fr\'echets\footnote{
Plus pr\'ecis\'ement, $({\rm St}^{\rm lisse}_L)^\dual$ est une limite projective d\'enombrable d'espaces
de dimension finie.},

$\bullet$ les ${\rm Dir}_a-{\rm Dir}_b$, pour $a,b\in F$, engendrent un sous-espace dense 
des trois espaces,

$\bullet$ La restriction $({\rm St}^{\rm an}_L)^\dual\to ({\rm St}^{\rm lisse}_L)^\dual$ est surjective,
et $({\rm St}^{\rm cont}_L)^\dual\to ({\rm St}^{\rm an}_L)^\dual$ est injective, d'image dense,
et identifie $({\rm St}^{\rm cont}_L)^\dual$ \`a l'espace des vecteurs $G$-born\'es 
de $({\rm St}^{\rm an}_L)^\dual$.

\Subsection{Cohomologie de de Rham du demi-plan de Drinfeld}
\begin{prop}\label{steinb1}
On a un isomorphisme naturel
    $$\O(\Omega_{\rm Dr})^\dual/C^\dual\simeq ({\rm St}^{\rm lisse}_{\mathbf{Z}})^\dual$$
envoyant $\frac{z-a}{z-b}$ sur ${\rm Dir}_a-{\rm Dir}_b$.
\end{prop}
\begin{proof}
Posons $$A=\O(\Omega_{\rm Dr})=\varprojlim_n A_n,$$
avec $A_n=\O(U_n)$. 
Le cor.\,\ref{unites} fournit des suites exactes
     $$0\to A_n^{**}/(1+{\goth m}_C)\to A_n^\dual/C^\dual
\to {\rm LC}( {\cal P}_n,\Z)^\dual\to \Z\to 0,$$
et donc, en passant \`a limite et en utilisant le lemme \ref{technique},
une suite exacte
$$0\to A^\dual/C^\dual\to {\rm LC}( \mathbf{P}^1(F),\Z)^\dual\to \Z\to 0,$$
qui fournit l'isomorphisme annonc\'e.
Que $\frac{z-a}{z-b}$ soit envoy\'e sur ${\rm Dir}_a-{\rm Dir}_b$ se voit facilement en revenant
aux d\'efinitions dans le cor.~\ref{unites}
\end{proof}

\begin{prop}\label{steinb2}
{\rm (i)}
L'application
$\mu\mapsto\int_{\piqp(F)}\frac{dz}{z-x}\,\mu(x)$ induit 
un isomorphisme $$({\rm St}^{\rm an}_C)^\dual\cong \Omega^1(\Omega_{\rm Dr}).$$

{\rm (ii)} L'application qui envoie $f\in A^\dual$ sur son symbole
$(f)_{\rm dR}=\frac{df}{f}$ en cohomologie de de Rham induit un isomorphisme
$$C\widehat\otimes(\O(\Omega_{\rm Dr})^\dual/C^\dual)\cong H^1_{\rm dR}(\Omega_{\rm Dr}).$$

{\rm (iii)} Le diagramme 
$$\xymatrix@C=.5cm@R=.6cm{&({\rm St}^{\rm an}_C)^\dual\ar[r]^-{\sim}\ar[dl]
&\Omega^1(\Omega_{{\rm Dr}})\ar[dr]\\
({\rm St}^{\rm lisse}_C)^\dual\ar@{=}[r]&C\widehat\otimes({\rm St}^{\rm lisse}_\Z)^\dual& 
C\widehat\otimes(\O(\Omega_{\rm Dr})^\dual/C^\dual)\ar[l]_-{\sim}
\ar[r]^-{\sim}&H^1_{\rm dR}(\Omega_{\rm Dr})}$$
est 
un diagramme commutatif 
de $G$-repr\'esentations.
\end{prop}
\begin{proof}  Les descriptions de $\Omega^1(\Omega_{\rm Dr})$ et 
de $H^1_{\rm dR}(\Omega_{\rm Dr})$ en termes de repr\'esentations de Steinberg sont parfaitement standard, cf. 
 \cite{Morita} et la discussion suivant le th\'eor\`eme 5 dans \cite{SS}. Le reste se d\'eduit de la proposition 
 ci-dessus.
\end{proof}

\subsection{Cohomologie \'etale $p$-adique du demi-plan de Drinfeld}
Le r\'esultat suivant est d\^u \`a Drinfeld \cite{Drinfeld}. 
    
    \begin{theo}\label{omegaet}
      On a des isomorphismes de $G\times \mathcal{G}_F$-modules pour tout $k\geq 1$, 
       $$H^1_{\eet}(\Omega_{\rm Dr}, \Z/{p^k}(1))\simeq ({\rm St}^{\rm lisse}_{\Z/p^k})^\dual,$$
et donc aussi des isomorphismes de $G\times \mathcal{G}_F$-modules 
       $$H^1_{\eet}(\Omega_{\rm Dr}, \zp(1))\simeq ({\rm St}^{\rm cont}_{\zp})^\dual,
\quad H^1_{\eet}(\Omega_{\rm Dr}, \qp(1))\simeq ({\rm St}^{\rm cont}_{\Q_p})^\dual.$$
    
    \end{theo}
    
    \begin{proof}    
La troisi\`eme assertion se d\'eduit de la seconde en tensorisant
par $\Q_p$ et la seconde de la premi\`ere par passage \`a la limite:
$({\rm St}^{\rm lisse}_{\Z/p^k\Z})^\dual=({\rm St}^{\rm cont}_{\zp})^\dual/p^k$.
Il suffit donc de d\'emontrer la premi\`ere assertion.
On a   $$H^1_{\eet}(\Omega_{\rm Dr}, \Z/{p^k}(1))\simeq A^\dual /(A^\dual )^{p^k}.$$
En effet, la suite exacte de Kummer
     $ 0\to \Z/{p^k}(1)\to {\mathbb G}_m\stackrel{p^k}{\to} {\mathbb G}_m\to 0 $
fournit la suite exacte courte
     $$
     0\to A^\dual /(A^\dual )^{p^k}\to H^1_{\eet}(\Omega_{\rm Dr},\Z/{p^k}(1))\to H^1_{\eet}(\Omega_{\rm Dr},{\mathbb G}_m)_{p^k}\to 0
     $$
Il suffit donc de prouver que
$H^1_{\eet}(\Omega_{\rm Dr},{\mathbb G}_m)=0$, 
ce qui r\'esulte de la nullit\'e
de ${\rm Pic}(A_n)$ et $\R^1\varprojlim A^\dual _n$ (lemme~\ref{technique}).

Maintenant, $A^\dual /(A^\dual )^{p^k}\simeq (A^\dual /C^\dual )/(A^\dual /C^\dual )^{p^k}$. 
  On conclut en utilisant la prop.~\ref{steinb1} et en remarquant  que 
     $ ({\rm St}^{\rm lisse}_{\Z})^\dual /p^k \simeq ({\rm St}^{\rm lisse}_{\Z/p^k})^\dual $.
    \end{proof}

\Subsection{Cohomologie pro\'etale du demi-plan de Drinfeld}
\begin{theo}\label{omegaproet}
On a une suite exacte de $G\times \mathcal{G}_F$-modules 
$$0\to \mathcal{O}(\Omega_{\rm Dr})/C\to H^1_{\proet} (\Omega_{\rm Dr}, \qp(1))
\to ({\rm St}^{\rm lisse}_{\Q_p})^\dual\to 0,$$
   \end{theo}
   
   \begin{proof} Comme $\R^1\varprojlim H^0_{\proet}(U_{n},\qp(1))=0$,
   on obtient un isomorphisme
   $$H^1_{\proet} (\Omega_{\rm Dr}, \qp(1))\simeq \varprojlim H^1_{\proet} (U_{n}, \qp(1)).$$
  L'espace $U_{n}$ \'etant quasi-compact et s\'epar\'e, on a $$H^1_{\proet} (U_{n}, \qp(1))=H^1_{\eet}(U_{n}, \qp(1))=(\varprojlim H^1_{\eet}(U_{n}, \Z/{p^k}(1)))\otimes_{\zp} \qp={A_n^\dual }\widehat{\otimes} \qp,$$
   o\`u $\widehat{\otimes}$ d\'enote le produit tensoriel compl\'et\'e.
   
Tensoriser la suite exacte du cor.\,\ref{unites} par $\Q_p$ fournit 
une suite exacte 
$$0\to (A_n^{**}/(1+{\goth m}_C))\widehat{\otimes} \qp\to 
(A_n^\dual /C^\dual)\widehat{\otimes}\qp 
\to {\rm LC}( {\cal P}_n, \qp)^\dual \to \Q_p\to 0.$$
Le th\'eor\`eme s'en d\'eduit en
passant \`a la limite projective, en utilisant l'isomorphisme
$$({\rm St}^{\rm lisse}_{\Q_p})^\dual=\varprojlim_n{\rm Ker}\big({\rm LC}( {\cal P}_n, \qp)^\dual \to \Q_p\big)\,,$$
la nullit\'e de $\R^1\varprojlim\, A_n$ (qui d\'ecoule du fait que $(A_n)_{n}$ d\'efinit 
une alg\`ebre de Fr\'echet-Stein) et l'isomorphisme de prosyst\`emes
$$((A_n^{**}/(1+{\goth m}_C))\widehat{\otimes} \qp)_{n}\simeq (A_n/C)_{n}
$$
que nous allons \'etablir.
Pour cela,
      fixons un point $Q\in U_1$ et notons $A_{n,Q}^+$ (resp.~$A_{n,Q}^{\dual\dual}$)
l'ensemble des 
      $g\in A_n^+$ (resp.~$f\in A_{n}^{\dual\dual}$) telles que 
$g(Q)=0$ (resp.~$f(Q)=1$). 
On d\'eduit du lemme\,\ref{gab1} que, si $kv_p(\varpi)\geq 1$, alors:

\quad $\bullet$ $f\in A_{n,Q}^{\dual\dual}$ implique $f\in 1+p\varpi A^+_{n-k-1,Q}$
et donc $\log f\in A^+_{n-k-1,Q}$,

\quad $\bullet$ $g\in A_{n,Q}^+$ implique $g\in p\varpi A^+_{n-k-1,Q}$
et donc $\exp f\in A_{n-k-1,Q}^{\dual\dual}$.

D'o\`u des isomorphismes de pro-objets $(A_{n,Q}^{**})_{n}
\simeq (A_{n,Q}^+)_{n}$ et 
$$(({A_n^{**}/(1+{\goth m}_C))}\widehat{\otimes} \qp)_{n} \simeq 
      (A_{n,Q}^{\dual\dual}\widehat\otimes \qp)_{n}\simeq (A_{n,Q}^+\widehat\otimes \qp)_{n}\simeq (A_n/C)_{n},$$
comme annonc\'e.
\end{proof}
 \begin{coro}
La suite exacte du th\'eor\`eme~\ref{omegaproet} s'inscrit dans le
diagramme commutatif suivant de
$G\times \mathcal{G}_F $-repr\'esentations:
 $$
 \xymatrix@R=.6cm@C=.8cm{
 0\ar[r] & \mathcal{O}(\Omega_{\rm Dr}) /C\ar[r]^-{\exp}\ar@{=}[d] & H^1_{\proet}(\Omega_{\rm Dr},\qp(1))\ar[d]^{\dlog}\ar[r] & ({\rm St}^{\rm lisse}_{\Q_p})^\dual\ar[r]\ar[d]^{\iota} & 0\\
0\ar[r] & \mathcal{O}(\Omega_{\rm Dr} )/C\ar[r]^-d  & \Omega^1 (\Omega_{\rm Dr})\ar[r] & 
({\rm St}^{\rm lisse}_{C})^\dual\ar[r] & 0\\ } $$
 \end{coro}
 \begin{proof}
$\iota$ est l'inclusion \'evidente.
 Pour d\'efinir l'application $\dlog$, utilisons les isomorphismes
 \begin{align*}
 H^1_{\proet}(\Omega_{\rm Dr},\qp(1))=\varprojlim_n  A_n^\dual \widehat{\otimes}_{\Z_p}\qp,\quad
 \Omega^1(\Omega_{\rm Dr})=\varprojlim_n\Omega^1(U_{n})
 \end{align*}
et posons $\dlog=\varprojlim_n \dlog$. 
Le carr\'e de gauche commute de mani\`ere \'evidente; la commutativit\'e de celui de droite
r\'esulte de ce que la classe de $f=\frac{z-a}{z-b}$ dans $H^1_{\proet}$
a pour image $\frac{df}{f}$ dans $\Omega^1$ et ${\rm Dir}_a-{\rm Dir}_b$ dans
$({\rm St}^{\rm lisse}_{\Q_p})^\dual$ et dans $({\rm St}^{\rm lisse}_C)^\dual$ identifi\'e
\`a~$H^1_{\rm dR}$.
 \end{proof}     
\begin{rema}
Comme nous le verrons, le terme $({\rm St}^{\rm lisse}_{\Q_p})^\dual$ s'interpr\`ete
comme $(\bst^+\otimes H^1_{\rm HK}(\Omega_{\rm Dr}))^{N=0,\varphi=p}$ et $\iota$ comme
$\theta\otimes\iota_{\rm HK}$, o\`u 
$H^1_{\rm HK}$ est la cohomologie de Hyodo-Kato,
$\iota_{\rm HK}:C\widehat\otimes_{\breve\Q_p}H^1_{\rm HK}(\Omega_{\rm Dr})
\overset{\sim}\to H^1_{\rm dR}(\Omega_{\rm Dr})$ est l'isomorphisme de Hyodo-Kato
et $\theta:\bst^+\to C$ est l'application habituelle.
\end{rema}

\section{Cohomologie \'etale $p$-adique et correspondance de Langlands locale}\label{DEMI2}
Nous allons admettre le th.\,\ref{diagfond} (que nous d\'emontrerons dans le chap.~\ref{GAB4}, cf.~\S\,\ref{GAB8})
et en d\'eduire un certain nombre de cons\'equences, dont le th.\,\ref{intro1}.
\subsection{La cohomologie pro\'etale $p$-adique de ${\cal M}_\infty$}\label{Emer1}
Si $M$
est un $(\varphi,N,\G_F)$-module supercuspidal, de pente~$\frac{1}{2}$ et rang $2$ sur
$L\otimes\Q_p^{\rm nr}$,
posons\footnote{Si $M$ est de pente $\frac{1}{2}$, le caract\`ere central de ${\rm JL}(M)$
est unitaire et l'action de $W_F$ sur ${\rm Hom}({\rm JL}(M),X)$,
pour $X=L\otimes \O({\cal M}_\infty), L\otimes\Omega^1({\cal M}_\infty)$, s'\'etend en une action de $\G_F$.}:
\begin{align*}
\O[M]=&\ {\rm Hom}_{L[\check G]}
\big({\rm JL}(M), L\otimes_{\Q_p}\O({{\cal M}_\infty})\big)^{\G_F}\\
\Omega^1[M]=&\ {\rm Hom}_{L[\check G]}
\big({\rm JL}(M), L\otimes_{\Q_p}\Omega^1({{\cal M}_\infty})\big)^{\G_F}\\
H^1_{\proet}[M]=&\ {\rm Hom}_{L[\check G]}
\big({\rm JL}(M), H^1_{\proet}({\cal M}_\infty,L(1))\big)
\end{align*}
(Les modules $\O[M]$ et $\Omega^1[M]$ sont des $L\otimes_{\Q_p} F$-modules
tandis que $H^1_{\proet}[M]$ est un $L$-module.)

Il r\'esulte du th.\,\ref{diagfond} que l'on
a le r\'esultat suivant.
\begin{coro}\label{Diag10}
Si $M$ est un $L$-$(\varphi,N,\G_F)$-module supercuspidal, de pente~$\frac{1}{2}$ et rang $2$,
on a le diagramme commutatif suivant
de $(G\times \G_F)$-fr\'echets,
$$\xymatrix@R=.5cm@C=.5cm{
0\ar[r] & C\widehat\otimes_F\O[M]\ar[r]\ar@{=}[d]& H^1_{\proet}[M]\ar[r]\ar[d]&
(\bcris^+\otimes_{\Q_p^{\rm nr}} M)^{\varphi=p}\widehat\otimes_L {\rm LL}(M)^\dual\ar[r]\ar[d]&0\\
0\ar[r] &C\widehat\otimes_F\O[M]\ar[r] &C\widehat\otimes_F\Omega^1[M]\ar[r]&
(C\otimes_{\Q_p^{\rm nr}} M) \widehat\otimes_L {\rm LL}(M)^\dual\ar[r]&0
}$$
dans lequel les lignes sont exactes, et les fl\`eches verticales
sont injectives et ont une image ferm\'ee. 
\end{coro}
(La ligne du bas est constitu\'ee de $L\otimes_{\Q_p}F$-modules, tandis que les termes
de la ligne du haut ne sont que des $L$-modules.  Notons que $N=0$ sur $M$, ce qui permet de remplacer $\bst$ par $\bcris$.)

\subsubsection{Consid\'erations topologiques}\label{affin3.1}
Si $Y$ est un affino\"{\i}de  de dimension~$1$ sur $C$ et si
$M=\Z/p^n,\Z_p,\Q_p$, on dispose des groupes de cohomologie \'etale
$H^i_{\eet}(Y,M(1))$, pour $i=0,1$, reli\'es par:
$$H^i_{\eet}(Y,\Q_p(1))=\Q_p\otimes H^i_{\eet}(Y,\Z_p(1)),
\quad H^i_{\eet}(Y,\Z_p(1))=\varprojlim H^i_{\eet}(Y,\Z/p^n(1)).$$

La suite exacte $0\to\Z/p^n(1)\to{\bf G}_m\to{\bf G}_m\to 0$ induit
la suite exacte de Kummer:
$$0\to (\O(Y)^\dual/C^\dual)/(\O(Y)^\dual/C^\dual)^{p^n}\to H^1_{\eet}(Y,\Z/p^n(1))\to {\rm Pic}(Y)[p^n]\to 0$$
et, si $f\in\O(Y)^\dual$, on note $(f)_{{\eet},n}$ son image
dans $H^1_{\eet}(Y,\Z/p^n(1))$ induite par la fl\`eche ci-dessus.

En prenant une limite projective sur $n$, puis en inversant $p$,
on obtient les suites exactes:
\begin{align*}
0\to (\O(Y)^\dual/C^\dual)^\wedge\to H^1_{\eet}(Y,\Z_p(1))\to T_p({\rm Pic}(Y))\to 0,\\
0\to \Q_p\widehat\otimes (\O(Y)^\dual/C^\dual)^\wedge\to H^1_{\eet}(Y,\Q_p(1))\to V_p({\rm Pic}(Y))\to 0,
\end{align*}
o\`u $(\O(Y)^\dual/C^\dual)^\wedge$ d\'esigne le s\'epar\'e compl\'et\'e pour la topologie $p$-adique.
Si $f\in \O(Y)^\dual$, on note $$(f)_{\eet}\in H^1_{\eet}(Y,\Z_p(1))$$ son {\it symbole en cohomologie \'etale}:
c'est l'image de $f$ par la fl\`eche ci-dessus (en composant avec l'application naturelle
de $\O(Y)^\dual$ dans $(\O(Y)^\dual/C^\dual)^\wedge$).
\begin{rema}\label{torsion}
{\rm (i)} Il r\'esulte de la suite exacte ci-dessus que $H^1_{\eet}(Y,\Z_p(1))$ est sans torsion,
et donc que $H^1_{\eet}(Y,\Q_p(1))$ est un banach.

{\rm (ii)}  
Si $Y$ est une courbe Stein,
on peut \'ecrire $Y$ comme la r\'eunion croissante
d'affino\"{\i}des $Y_n$.
Alors $H^1_{\proet}(Y,\Q_p(1))=\varprojlim_n H^1_{\eet}(Y_n,\Q_p(1))$ (on a
$H^1_{\proet}(Y_n,\Q_p(1))=H^1_{\eet}(Y_n,\Q_p(1))$ puisque $Y_n$ est un affino\"{\i}de);
la limite ne d\'epend pas du choix des $Y_n$ car deux tels syst\`emes sont cofinaux.
Comme les $H^1_{\eet}(Y_n,\Q_p(1))$ sont des banachs,
$H^1_{\proet}(Y,\Q_p(1))$ est naturellement un fr\'echet.
\end{rema}

\subsubsection{Le $\G_F$-module $H^1_{\proet}({\cal M}_\infty)$}
Nous allons calculer la multiplicit\'e d'une repr\'esentation
$V$ de $\G_F$ dans $H^1_{\proet}({\cal M}_\infty)$ comme repr\'esentation de $G$.
Notons que l'on ne peut pas esp\'erer que le r\'esultat ait toujours un lien
avec la correspondance de Langlands locale $p$-adique car
le sous-espace $C\widehat\otimes_F\O({\cal M}_\infty)$ contient des repr\'esentations
de $\G_F$ de dimension arbitraire (il suffit qu'un des poids de $V$ soit $0$
pour que $V$ apparaisse dans $C$), mais la prop.\,\ref{LL} et la rem.\,\ref{hope}
montrent que, si $V$ a la bonne forme, cette multiplicit\'e est celle esp\'er\'ee.  

Soit $M$ un $L$-$(\varphi,N,\G_F)$-module supercuspidal de rang $2$ et de pente~$\frac{1}{2}$,
et soient
$$M_{\rm dR}=(C\otimes_{\Q_p^{\rm nr}} M)^{\G_{F}}
\quad{\rm et}\quad
X_{\rm st}^+(M)=(\bcris^+\otimes_{\Q_p^{\rm nr}}M)^{\varphi=p}.$$
Alors $M_{\rm dR}$ est un $L\otimes F$-module de rang~$2$ et
$X_{\rm st}^+(M)$ est un $L$-Espace de Banach~\cite{CB,Cdr} de $L$-Dimension $([L:\Q_p],2)$.

Si $V$ est une $L$-repr\'esentation de $\G_F$, posons
$$H_M(V)={\rm Hom}_{L[\G_{F}]}(V,X_{\rm st}^+(M)),\quad
H_C(V)={\rm Hom}_{L[\G_{F}]}(V,L\otimes_{\Q_p}C).$$
Alors $H_M(V)$ est un $L$-espace de dimension finie, tandis que
$H_C(V)=(C\otimes_{\Q_p}V^\dual)^{\G_F}$ est un $L\otimes F$-module de type fini.

\begin{rema}\label{ajout1} 
Comme $M$ est de pente~$>0$, on a $(\bcris^+\otimes_{\Q_p^{\rm nr}} M)^{\varphi=1}=0$,
et
le morphisme naturel $X_{\rm st}^+(M)\to M_{\rm dR}\otimes_F C$ (induit par $\theta: X_{\rm st}^+(M)\to 
M\otimes_{\Q_p^{\rm nr}} C\simeq 
M_{\rm dR}\otimes_F C$) est injectif;
il induit donc un morphisme injectif
$$\iota: H_M(V)\to M_{\rm dR}\otimes_{L\otimes F} H_C(V),$$
ce qui permet d'identifier 
$H_M(V)$ \`a un sous-espace de $M_{\rm dR}\otimes_{L\otimes F} H_C(V)$.
\end{rema}

\begin{prop}\label{ajout2}
On a le diagramme commutatif suivant de $G$-fr\'echets
{\small
$$\xymatrix@R=.3cm@C=.3cm{
0\ar[r] & H_C(V)\bigotimes\limits_{L\otimes F}\O[M]\ar[r]\ar@{=}[d]& 
{\rm Hom}_{L[\G_F]}(V,H^1_{\rm proet}[M])\ar[r]\ar[d]&
H_M(V)\bigotimes\limits_L {\rm LL}(M)^\dual\ar[r]\ar[d]&0\\
0\ar[r] &H_C(V)\bigotimes\limits_{L\otimes F}\O[M]\ar[r] &H_C(V)\bigotimes\limits_{L\otimes F}\Omega^1[M]\ar[r]&
H_C(V)\bigotimes\limits_{L\otimes F} M_{\rm dR}\bigotimes\limits_L {\rm LL}(M)^\dual\ar[r]&0
}$$ 
}
dans lequel les lignes sont exactes et les fl\`eches verticales sont injectives, d'image ferm\'ee.
\end{prop}
\begin{proof}
L'application naturelle $$H^1(\G_{F},C\widehat\otimes_F\O[M]\otimes_L V^\dual)\to
H^1(\G_{F},C\widehat\otimes_F\Omega^1[M]\otimes_L V^\dual)$$ est injective
(elle est obtenue en tensorisant $\O[M]\to \Omega^1[M]$ au dessus de $L\otimes F$ par
$H^1(\G_{F},C\otimes_{\Q_p} V^\dual)$, qui est un $L\otimes F$-module de rang fini).
Par commutativit\'e du diagramme du cor.~\ref{Diag10}, il en est de m\^eme de
$$H^1(\G_{F},C\widehat\otimes_F\O[M]\otimes_L V^\dual)\to
H^1(\G_{F}, H^1_{\proet}[M]\otimes_L V^\dual),$$
ce qui prouve que ${\rm Hom}_{L[\G_{F}]}(V,-)$ laisse exactes les lignes
du diagramme du cor.~\ref{Diag10}.  On en d\'eduit le r\'esultat.
\end{proof}

\subsubsection{Repr\'esentations supercuspidales de $\G_F$}\label{Emer2}
La prop.\,\ref{LL1} ci-dessous fournit des contraintes tr\`es fortes sur
$H_M(V)$. 

Si $M$ est un $(\varphi,N,\G_{F})$-module, et $k\in\Z$, on note
$M[k]$ le $(\varphi,N,\G_{F})$-module obtenu en multipliant par $p^k$
l'action de $\varphi$ sur $M$.

\begin{prop}\label{LL1}
Soit $W$ une $L$-repr\'esentation supercuspidale de $\G_{F}$, de dimension~$d$,
et soient $M=\dpst(W)$ et\footnote{
{\og $W$ supercuspidale\fg} $\Rightarrow$ 
{\og $N=0$ sur $M$\fg}, et on a donc
aussi $X_{\rm st}(M)=(\bcris\otimes M)^{\varphi=p}$.}
 $X_{\rm st}(M)=(\bst\otimes_{\Q_p^{\rm nr}} M)^{N=0,\varphi=p}$.
 Si
$V$ est une $L$-repr\'esentation de $\G_{F}$, alors:

{\rm (i)} ${\rm Hom}_{L[\G_{F}]}(V,X_{\rm st}(M))={\rm Hom}_{L[{\rm WD}_F]}(M[-1]^\dual,\dpst(V^\dual))$.

{\rm (ii)}
Si $V$ est de dimension~$d$, et si ${\rm Hom}_{L[\G_{F}]}(V,X_{\rm st}(M))\neq 0$,
alors $V$ est supercuspidale, $\dpst(V)\cong M[-1]$,
et ${\rm Hom}_{L[\G_{F}]}(V,X_{\rm st}(M))$ est de dimension~$1$ sur $L$.
\end{prop}
\begin{proof}  
Si $K$ est une extension galoisienne suffisamment grande de $F$, les isomorphismes $M=\Q_p^{\rm nr}\otimes_{K_0}M^{\G_{K}}$ (et son analogue pour $\dpst(V^\dual)$) 
et $\bst\otimes_{\Q_p} V^*=\bst\otimes_{\Q_p^{\rm nr}}\dpst(V^\dual)$ fournissent
\begin{align*}
{\rm Hom}&_{L[\G_{F}]}(V,X_{\rm st}(M))=\ 
(\bst\otimes_{\Q_p^{\rm nr}} M\otimes_L V^\dual)^{N=0,\varphi=p,\G_{F}}\\
= 
&\ 
(\bst\otimes_{K_0} M^{\G_{K}}\otimes_L V^*)^{N=0,\varphi=p,\G_{F}}
=((\bst\otimes_L V^*)^{\G_{K}}\otimes_{K_0} M^{\G_{K}})^{N=0,\varphi=p,\G_{F}}\\
&=  \big(\dpst(V^\dual)\otimes_{L\otimes\Q_p^{\rm nr}} M[-1]\big)^{N=0,\varphi=1,\G_{F}}=
{\rm Hom}_{L[{\rm WD}_{F}]}(M[-1]^\dual, \dpst(V^\dual)),
\end{align*}
ce qui prouve le (i).

Passons au (ii).
Comme $M$ est irr\'eductible, et $\dpst(V^\dual)$ de dimension~$\leq\dim M$,
cela implique que ${\rm Hom}_{L[{\rm WD}_{F}]}(M[-1]^\dual,\dpst(V^\dual))$ est de dimension $0$ ou $1$
suivant que $\dpst(V^\dual)$ est isomorphe ou pas \`a $M[-1]^\dual$, ce qui \'equivaut
\`a $\dpst(V)\cong M[-1]$ car l'isomorphisme $\dpst(V^\dual)\cong M[-1]^\dual$ implique
que $V$ est supercuspidale, et donc que $V$ aussi et que $\dpst(V^\dual)=\dpst(V)^\dual$.
\end{proof}

\subsubsection{La repr\'esentation supercuspidale $V_{M,{\cal L}}$ de $\G_F$}\label{Emer3}
Si ${\cal L}$ est une $L\otimes F$-droite du $L\otimes F$-module $M_{\rm dR}$ de rang~$2$,
soit $V_{M,{\cal L}}$ la $L$-repr\'esentation de $\G_{F}$:
$$V_{M,{\cal L}}={\rm Ker}\big[X_{\rm st}^+(M)\to \bdr^+\otimes_FM_{\rm dR}\overset{\theta}{\to}
C\otimes_{F} (M_{\rm dR}/{\cal L})\big].$$
Alors $V_{M,{\cal L}}$ est une repr\'esentation supercuspidale de
$\G_F$, de dimension~$2$, \`a poids $0$ et $1$,
et $\dpst(V_{M,{\cal L}})=M[-1]$.
\begin{rema}\label{Hum1}
Si $V$ est une $L$-repr\'esentation de dimension~$2$ telle que
${\rm Hom}_{L[\G_F]}(V,X_{\rm st}^+(M))\neq 0$, il r\'esulte de la prop.~\ref{LL1}
que $V$ est supercuspidale, que
${\rm Hom}_{L[\G_F]}(V,X_{\rm st}(M))$
est de dimension~$1$, et que $\dpst(V)\cong M[-1]$.
Il existe donc une filtration sur $M_{\rm dR}$ par des
$L\otimes F$-modules non n\'ecessairement libres, telle
que $V=X_{\rm st}(M)\cap{\rm Fil}^1(\bdr\otimes_FM_{\rm dR})$,
et comme ${\rm Hom}_{L[\G_F]}(V,X_{\rm st}(M))$
est de dimension~$1$ et que l'on dispose d'un plongement de
$V$ dans $X_{\rm st}^+(M)\subset X_{\rm st}(M)$, on obtient:
$$V={\rm Ker}\big[X_{\rm st}^+(M)\to (\bdr^+\otimes_FM_{\rm dR})/{\rm Fil}^1\big].$$
On peut d\'ecomposer $L\otimes F$ sous la forme $\prod_{\sigma:F\to L}L$,
ce qui d\'ecompose $M_{\rm dR}$ sous la forme $\oplus_{\sigma:F\to L}M_{{\rm dR},\sigma}$,
o\`u $M_{{\rm dR},\sigma}$ est un $L$-module de rang $2$ (vu comme $L\otimes F$-module).
Se donner une filtration par des $L\otimes F$-modules de $M_{\rm dR}$ revient \`a se donner
une filtration de $M_{{\rm dR},\sigma}$ par des $L$-espaces vectoriels pour tout $\sigma$,
d'o\`u des poids de Hodge-Tate $\kappa_{1,\sigma}\leq \kappa_{2,\sigma}$ pour chaque $\sigma$.

Les seules contraintes r\'esultant de l'admissibilit\'e de la filtration
et du fait que l'on peut prendre $\bcris^+$ au lieu de $\bcris$ pour retrouver $V$,
sont $\kappa_{i,\sigma}\geq 0$ pour tous $\sigma$ et $i$, et $\sum_{i,\sigma}\kappa_{i,\sigma}=[F:\Q_p]$.

$\bullet$ {\it Si $F=\Q_p$}, cela ne laisse pas de choix, et {\it $V=V_{M,{\cal L}}$ pour un 
unique~${\cal L}$}.

$\bullet$ {\it Si $F\neq \Q_p$}, les $V_{M,{\cal L}}$ correspondent \`a $\kappa_{1,\sigma}=0$ et
$\kappa_{2,\sigma}=1$ pour tout $\sigma$, mais {\it il y a d'autre possibilit\'es, y compris
\`a poids $0$ et $1$}.
\end{rema}

On tire de la prop.~\ref{LL1} et de la rem.~\ref{Hum1} le r\'esultat suivant.
\begin{lemm}\label{Fil1}
{\rm (i)}
${\rm Hom}_{L[\G_F]}(V_{M,{\cal L}},X_{\rm st}^+(M))$
est de dimension~$1$ sur $L$.

{\rm (ii)} Si $F=\Q_p$, et si
$V$ est une $L$-repr\'esentation de dimension~$2$ de $\G_F$, alors
${\rm Hom}_{L[\G_F]}(V,X_{\rm st}^+(M))\neq 0$
si et seulement si $V\cong V_{M,{\cal L}}$ pour un {\rm (unique)} ${\cal L}$.
\end{lemm}
\begin{rema}\label{Fil1.1}
Par construction, $C\otimes_{\Q_p} V_{M,{\cal L}}\cong (L\otimes_{\Q_p}C)\oplus
(L\otimes_{\Q_p}C)(1)$.  On d\'eduit des calculs de cohomologie continue
de Tate que, si $V=V_{M,{\cal L}},V_{M,{\cal L}}^\dual$, alors
$H^0(\G_F,C\otimes_{\Q_p} V)$ et $H^1(\G_F,C\otimes_{\Q_p} V)$ sont des
$(L\otimes F)$-modules de rang~$1$.
\end{rema}

\subsubsection{La repr\'esentation $W_{M,{\cal L}}$ de $G$}\label{Emer4}
Si on prend les points fixes par $\G_F$ de la ligne du bas du diagramme
du cor.~\ref{Diag10}, on obtient la suite exacte:
$$0\to\O[M]\to\Omega^1[M]\to M_{\rm dR}\otimes_L{\rm LL}(M)^\dual\to 0$$
de $L\otimes F$-repr\'esentations de $G$.
On d\'efinit la
repr\'esentation $W'_{M,{\cal L}}$ de $G$,
en prenant l'image inverse de ${\cal L}\otimes {\rm LL}(M)^\dual$ dans
$\Omega^1[M]$. On a donc une suite exacte:
$$0\to\O[M]\to W'_{M,{\cal L}}\to {\cal L}\otimes_L{\rm LL}(M)^\dual\to 0$$
de $(L\otimes F)$-repr\'esentations de $G$.

Soit $M_{\rm dR}^\dual$ le $L\otimes F$-dual de $M_{\rm dR}$;
c'est un $L\otimes F$-module libre de rang~$2$.
On dispose d'un accouplement $(L\otimes F)$-bilin\'eaire naturel 
$$M_{\rm dR}^\dual\otimes_{L\otimes F}M_{\rm dR}\to
L\otimes F,$$
et on note ${\cal L}^\perp\subset M_{\rm dR}^\dual$ l'orthogonal
de ${\cal L}$.  Alors ${\cal L}^\perp$ est un $(L\otimes F)$-module de rang~$1$,
et l'accouplement ci-dessus induit un isomorphisme
$$(M_{\rm dR}^\dual/{\cal L}^\perp)\otimes_{L\otimes F}{\cal L}\cong L\otimes F.$$
Au vu de cet isomorphisme, on pose 
$${\cal L}^{-1}=M_{\rm dR}^\dual/{\cal L}^\perp.$$
Comme $L$ s'injecte naturellement dans $L\otimes F$ (par $x\mapsto x\otimes 1$),
on fabrique une repr\'esentation $W_{M,{\cal L}}$, extension de
${\rm LL}(M)^\dual$ par ${\cal L}^{-1}\otimes_{L\otimes F}\O[M]\cong\O[M]$,
gr\^ace au diagramme suivant (la premi\`ere ligne est obtenue en tensorisant
par ${\cal L}^{-1}$ la suite exacte d\'efinissant $W'_{M,{\cal L}}$):
$$\xymatrix@R=.4cm@C=.5cm{
0\ar[r]&{\cal L}^{-1}\underset{L\otimes F}{\bigotimes}\O[M]\ar[r]&{\cal L}^{-1}\underset{L\otimes F}{\bigotimes}W'_{M,{\cal L}}
\ar[r]& (L\otimes F)\,\underset{L}{\bigotimes}\,{\rm LL}(M)^\dual\ar[r]&0\\
0\ar[r]&{\cal L}^{-1}\underset{L\otimes F}{\bigotimes}\O[M]\ar[r]\ar@{=}[u]&W_{M,{\cal L}}
\ar[r]\ar[u]& {\rm LL}(M)^\dual\ar[r]\ar[u]&0
}$$

\begin{prop}\label{dl2}
Si $F=\Q_p$, alors $$W_{M,{\cal L}}=(\Pi(V_{M,{\cal L}})^{\rm an})^\dual.$$
\end{prop}

\begin{proof}
 C'est le r\'esultat principal de \cite{DL}, cf. th. 1.4 de loc.cit. 
\end{proof}

\subsubsection{Cohomologie pro\'etale de ${\cal M}_\infty$ et correspondance de Langlands}
Le r\'esultat suivant d\'ecrit la multiplicit\'e de $V_{M,{\cal L}}$
dans $H^1_{\proet}[M]$, en tant que repr\'esentation de $G$.
Dans le cas $F=\Q_p$, combin\'e avec la prop.\,\ref{dl2}, il montre
que la cohomologie pro\'etale de ${\cal M}_\infty$ encode la correspondance
de Langlands locale $p$-adique (version localement analytique) coupl\'ee
avec la correspondance de Jacquet-Langlands classique pour les
repr\'esentations $V_{M,{\cal L}}$ (i.e.~pour les repr\'esentations supercuspidales
\`a poids $0$ et $1$, cf.~rem.\,\ref{Hum1}).

\begin{prop}\label{LL}
On a un isomorphisme
$$W_{M,{\cal L}}\cong {\rm Hom}_{L[\G_{F}]}(V_{M,{\cal L}},H^1_{\proet}[M]).$$
\end{prop}
\begin{proof}
Notons simplement $H_M$ et $H_C$ les modules $H_M(V_{M,{\cal L}})$ et $H_C(V_{M,{\cal L}})$.
D'apr\`es le lemme~\ref{Fil1}, $H_M$ est un $L$-module de rang~$1$.
L'injection $H_M\to M_{\rm dR}\otimes_{L\otimes F}H_C$ induit 
une application naturelle
$M_{\rm dR}^\dual\otimes_LH_M\to H_C$.
Le noyau de cette application contient ${\cal L}^\perp\otimes_L H_M$ car
l'image de $V_{M,{\cal L}}$ dans $C\otimes_F M_{\rm dR}$ par un \'el\'ement de $H_M$
est incluse dans ${\cal L}$ par d\'efinition de $V_{M,{\cal L}}$ (et le fait
que $H_M$ est de dimension~$1$).
Comme $M_{\rm dR}^\dual$ est un $(L\otimes F)$-module de rang~$2$ et
$H_C$
est un $(L\otimes F)$-module de rang~$1$ (rem.\,\ref{Fil1.1}),
on en d\'eduit que le noyau est exactement \'egal \`a ${\cal L}^\perp$,
et donc $H_C= {\cal L}^{-1}$.

On en tire, en utilisant la prop.\,\ref{ajout2}, le diagramme commutatif:
$$\xymatrix@R=.5cm@C=.4cm{
0\ar[r]&{\cal L}^{-1}\hskip-.1cm\underset{L\otimes F}{\bigotimes}\O[M]\ar[r]\ar@{=}[d]& 
 {\rm Hom}_{L[\G_{F}]}(V_{M,{\cal L}},H^1_{\proet}[M])\ar[d]\ar[r]&
{\rm LL}(M)^\dual\ar@{^{(}->}[d]\ar[r]&0\\
0\ar[r]&{\cal L}^{-1}\hskip-.1cm\underset{L\otimes F}{\bigotimes}\O[M]\ar[r]&{\cal L}^{-1}\hskip-.1cm\underset{L\otimes F}{\bigotimes}
W'_{M,{\cal L}}\ar[r]&{\cal L}^{-1}\hskip-.1cm\underset{L\otimes F}{\bigotimes}{\cal L}\underset{L}{\bigotimes}{\rm LL}(M)^\dual\ar[r]&0}$$
qui permet, en utilisant 
l'isomorphisme
${\cal L}^{-1}\otimes_{L\otimes F}{\cal L}\cong L\otimes F$ et
la d\'efinition de~$W_{M,{\cal L}}$,
de conclure.
\end{proof}

\begin{rema}\label{hope}
Posons
$$\Pi_{\rm geo}^{\rm an}(V_{M,{\cal L}}):=\big({\rm Hom}_{L[\G_F]}
(V_{M,{\cal L}},H^1_{\proet}[M])\big)^\dual.$$
Il r\'esulte des prop.\,\ref{LL} et~\ref{dl2} que, si $F=\Q_p$, alors
$\Pi_{\rm geo}^{\rm an}(V_{M,{\cal L}})$ est li\'e \`a la correspondance de Langlands locale
$p$-adique: 
$$\Pi_{\rm geo}^{\rm an}(V_{M,{\cal L}})=\Pi(V_{M,{\cal L}})^{\rm an}.$$
On peut esp\'erer que, si $F\neq \Q_p$, alors $\Pi_{\rm geo}^{\rm an}(V_{M,{\cal L}})$
a un lien avec l'espace de vecteurs $F$-analytiques de la
repr\'esentation $\Pi(V_{M,{\cal L}})$ associ\'ee \`a $V_{M,{\cal L}}$ par
l'hypoth\'etique correspondance de Langlands locale $p$-adique pour ${\bf GL}_2(F)$.
\end{rema}

\subsection{La cohomologie \'etale $p$-adique de ${\cal M}_\infty$}\label{Emer5}

On cherche \`a comprendre quelles repr\'esentations de $\G_{\Q_p}$
peuvent appara\^{\i}tre dans la cohomologie \'etale de la tour de
Drinfeld.  Comme on l'a vu, pour la cohomologie pro\'etale, il suffit
qu'un des poids de Hodge-Tate soit nul; comme nous allons le voir (th.\,\ref{eliminate}
et cor.\,\ref{emer4}),
les conditions sont nettement plus restrictives pour la cohomologie \'etale.
Cela r\'esulte en partie de ce que $\O({\cal M}_\infty)$ n'a pas d'\'el\'ements $G$-born\'es,
ce qui d\'ecoule, si $F=\Q_p$, des arguments de compl\'et\'es universels utilis\'es ci-dessous,
mais peut aussi se v\'erifier directement, pour $F$ g\'en\'eral (cf.~rem.\,\ref{extraordinary})

\subsubsection{La cohomologie \'etale comme sous-espace de la cohomologie pro\'etale}
Le r\'esultat suivant permet de d\'ecrire la cohomologie \'etale de
${\cal M}_\infty$ \`a l'int\'erieur de sa cohomologie pro\'etale.

Rappelons qu'un sous-ensemble $X$ d'un $\Q_p$-espace vectoriel localement convexe $E$
est {\it born\'e} si $p^nx_n\to 0$ pour toute suite $(x_n)_{n\in \N}$ d'\'el\'ements
de $X$.  Si $E$ est un banach d\'efini par une valuation $v$, cela \'equivaut
\`a l'existence de $N\in\Z$ tel que $v(x)\geq N$, quel que soit $x\in X$.
Si $E$ est un fr\'echet d\'efini par une famille $(v_k)_{k\in\N}$ de valuations, cela \'equivaut
\`a l'existence, pour tout $k\in\N$, de $N_k\in\Z$ tel que $v_k(x)\geq N_k$, quel que soit $x\in X$.

Si $G$ est un groupe agissant contin\^ument sur $E$, un \'el\'ement $x$ de $E$ est
dit {\it $G$-born\'e} si son orbite $\{g\cdot x,\ g\in G\}$ est born\'ee.
\begin{prop}\label{fonda1}
{\rm (i)} $H^1_{\eet}({\cal M}_n,L(1))$ est l'espace des vecteurs $G$-born\'es de
$H^1_{\proet}({\cal M}_n,L(1))$.

{\rm (ii)} Si $M$ est un $(\varphi,N,\G_{\Q_p})$-module supercuspidal,
libre de rang~$2$ sur $L\otimes_{\Q_p}\Q_p^{\rm nr}$, alors
$H^1_{\eet}[M]$ est l'espace des vecteurs $G$-born\'es de $H^1_{\proet}[M]$.
\end{prop}
\begin{proof}
Le (ii) est une cons\'equence imm\'ediate du (i) (prendre la $M$-partie permet de
travailler en niveau fini).  D\'emontrons donc le (i).
Soit $X_i$ une suite croissante d'affino\"{\i}des
dont la r\'eunion est ${\cal M}_n$.  Alors $H^1_{\proet}({\cal M}_n,L(1))$
est la limite projective des $H^1_{\eet}(X_i,L(1))$, et chacun des
$H^1_{\eet}(X_i,L(1))$ est un banach dont la boule unit\'e est
$H^1_{\eet}(X_i,\O_L(1))$ puisque $H^1_{\eet}(X_i,\O_L(1))$ est sans torsion.
Un sous-ensemble $A$ de $H^1_{\proet}({\cal M}_n,L(1))$ est donc born\'e si et
seulement si il existe une suite $k_i$ d'entiers tels que
${\rm Res}_{X_i}(A)\subset p^{-k_i}H^1_{\eet}(X_i,\O_L(1))$, pour tout $i$.
En particulier, $H^1_{\eet}({\cal M}_i,\O_L(1))=\varprojlim_i H^1_{\eet}(X_i,\O_L(1))$
est born\'e, et comme il est invariant par $G$, cela implique qu'il
est $G$-born\'e, et donc que $H^1_{\eet}({\cal M}_n,L(1))$ est inclus dans l'ensemble
des vecteurs $G$-born\'es.

Montrons maintenant que tout vecteur $G$-born\'e de $H^1_{\proet}({\cal M}_n,L(1))$
appartient \`a $H^1_{\eet}({\cal M}_n,L(1))$.
Pour cela,
choisissons un sous-groupe
cocompact $\Gamma$ de $G$ op\'erant sans point fixe sur l'arbre de ${\bf PGL}_2$.
Choisissons aussi des affino\"{\i}des $Y_1,\dots,Y_r$ de ${\cal M}_n$ avec les propri\'et\'es
suivantes:

$\bullet$ Les $\gamma\cdot Y_i$, pour $\gamma\in\Gamma$ et $1\leq i\leq r$,
forment un recouvrement de ${\cal M}_n$.

$\bullet$ Les intersections de trois $\gamma\cdot Y_i$, correspondant \`a des couples
$(\gamma,i)$ distincts deux \`a deux, sont vides.

\noindent (Pour construire de tels $Y_i$, on peut prendre un syst\`eme de repr\'esentants
des noeuds du squelette de ${\cal M}_n$ modulo l'action de $\Gamma$ -- comme $\Gamma$
est cocompact, ce syst\`eme est fini, notons $x_1,\dots,x_r$ ses \'el\'ements --
et, si $1\leq i\leq r$,
prendre un affino\"ide correspondant \`a une \'etoile assez grande du squelette, de sommet $x_i$,
mais ne contenant aucun autre noeud.)

On a alors des suites exactes:
\begin{align*}
&\prod_{(\gamma,i)\neq(\gamma',i')}\hskip-.5cm
H^0_{\eet}((\gamma\cdot Y_i)\cap(\gamma'\cdot Y_{i'}),\O_L(1)))
\to  H^1_{\eet}({\cal M}_n,\O_L(1))\to 
\prod_{(\gamma,i)}H^1_{\eet}(\gamma\cdot Y_i,\O_L(1))\\
&\prod_{(\gamma,i)\neq(\gamma',i')}\hskip-.5cm
H^0_{\eet}((\gamma\cdot Y_i)\cap(\gamma'\cdot Y_{i'}),L(1)))
\overset{\beta}{\to} H^1_{\proet}({\cal M}_n,L(1))\to 
\prod_{(\gamma,i)}H^1_{\eet}(\gamma\cdot Y_i,L(1))
\end{align*}
Les composantes connexes des
$(\gamma\cdot Y_i)\cap(\gamma'\cdot Y_{i'})$ ne forment qu'un nombre fini d'orbites
sous l'action de $\Gamma$.  L'image de $\beta$ dans
$H^1_{\proet}({\cal M}_n,L(1))$ est donc, topologiquement, engendr\'ee par les
translat\'es sous $\Gamma$ d'un nombre fini d'\'el\'ements $v_1,\dots,v_t$.
Soit alors $Y$ un affino\"{\i}de contenant $Y_1,\dots,Y_r$ et tel que
les images de $v_1,\dots,v_t$ dans $H^1_{\eet}(Y,L(1))$ soient non nulles
(il suffit de prendre pour $Y$ un des $X_i$ ci-dessus, pour $i\gg0$).
Alors
$H^1_{\proet}({\cal M}_n,L(1))\to
\prod_{\gamma}H^1_{\eet}(\gamma\cdot Y,L(1))$ est injective
et $H^1_{\eet}({\cal M}_n,\O_L(1))$ s'identifie \`a l'ensemble
des $v\in H^1_{\proet}({\cal M}_n,L(1))$ tels que ${\rm Res}_{\gamma\cdot Y}(v)\in
H^1_{\eet}(\gamma\cdot Y,\O_L(1))$, pour tout $\gamma\in\Gamma$, ce qui \'equivaut
\`a ${\rm Res}_{Y}(\gamma\cdot v)\in
H^1_{\eet}(Y,\O_L(1))$, pour tout $\gamma\in\Gamma$.

Or un \'el\'ement $G$-born\'e $v$ de $H^1_{\proet}({\cal M}_n,L(1))$ est a fortiori
$\Gamma$-born\'e et donc
v\'erifie, en particulier, que l'ensemble $\{{\rm Res}_{Y}(\gamma\cdot v),\ \gamma\in\Gamma\}$
est born\'e dans $H^1_{\eet}(Y,L(1))$, et donc est contenu dans
$p^{-N}H^1_{\eet}(Y,\O_L(1))$, pour $N$ assez grand.  D'apr\`es ce qui pr\'ec\`ede,
ceci implique $v\in p^{-N}H^1_{\eet}({\cal M}_n,\O_L(1))$, ce qui permet de conclure.
\end{proof}

\begin{center}
{\it Dans le reste de ce paragraphe, on suppose que $F=\Q_p$}.
\end{center}

\subsubsection{Sous-repr\'esentations galoisiennes de $H^1_{\eet}({\cal M}_\infty,\Q_p(1))$}
Soit $M$ un
$L$-$(\varphi, N, \mathcal{G}_{\Q_p})$-module, supercuspidal de pente $\frac{1}{2}$, libre de rang $2$ sur 
$L\otimes_{\Q_p} \Q_p^{\rm nr}$.
On dispose des $L$-modules $M_{\rm dR}$ et $X_{\rm st}^+(M)$ (\S\,\ref{Emer1}) et des $L$-repr\'esentations
$V_{M,{\cal L}}$ de $\G_F$ (\no\ref{Emer3}).
\begin{coro}\label{fonda2}
Si ${\cal L}$ est une $L$-droite de $M_{\rm dR}$, alors
$${\rm Hom}_{L[\G_{\Q_p}]}(V_{M,{\cal L}},H^1_{\eet}[M])\cong \Pi(V_{M,{\cal L}})^\dual.$$
En particulier, ${\rm Hom}_{L[\G_{\Q_p}]}(V_{M,{\cal L}},H^1_{\eet}[M])\neq 0$.
\end{coro}
\begin{proof}
Il r\'esulte de la prop.~\ref{fonda1}
que 
${\rm Hom}_{L[\G_{\Q_p}]}(V_{M,{\cal L}},H^1_{\eet}[M])$ est l'ensemble des vecteurs
$G$-born\'es de ${\rm Hom}_{L[\G_{\Q_p}]}(V_{M,{\cal L}},H^1_{\proet}[M])$.
Or les prop.~\ref{LL} et~\ref{dl2} permettent d'identifier ce dernier espace \`a
$(\Pi(V_{M,{\cal L}})^{\rm an})^\dual$ et le r\'esultat est une traduction, par dualit\'e,
de ce que $\Pi(V_{M,{\cal L}})$
est le compl\'et\'e universel de $\Pi(V_{M,{\cal L}})^{\rm an}$~\cite[th. VII.11]{CD}.
\end{proof}

\begin{rema}\label{hope2}
Ce r\'esultat sugg\`ere que, si $F\neq \Q_p$, la $G$-repr\'esentation
$$\Pi_{\rm geo}(V_{M,{\cal L}}):=
{\rm Hom}_{L[\G_{\Q_p}]}(V_{M,{\cal L}},H^1_{\eet}[M])^\dual$$
pourrait \^etre celle que l'on cherche en vue d'une correspondance de Langlands locale $p$-adique
pour ${\bf GL}_2(F)$.
\end{rema}

\begin{theo} \label{eliminate}
Soit $V$ une $L$-repr\'esentation de $\mathcal{G}_{\Q_p}$. Alors 
     
{\rm (i)} ${\rm Hom}_{L[\mathcal{G}_{\Q_p}]}(V, H^1_{\eet}[M])\ne 0$ 
si et seulement s'il existe\footnote{On se permet de remplacer, si besoin, $L$ par une extension finie.} $\cal L$ tel que $V_{M,{\cal L}}$ soit un quotient de $V$.
     
{\rm (ii)} Le $G$-banach $\Pi_M(V)={\rm Hom}_{L[\mathcal{G}_{\Q_p}]}(V, H^1_{\eet}[M])^\dual $ 
est admissible,
de longueur finie, et ses facteurs de Jordan-H\"older 
sont isomorphes \`a des $\Pi(V_{M, {\cal L}})$ pour certains~$\cal L$.        
\end{theo}
          
La preuve de ce r\'esultat est faite au \no\ref{emer5.2}.
On en tire, en utilisant le cor.~\ref{fonda2} et le fait que toute repr\'esentation
supercuspidale, de dimension~$2$, \`a poids $0$ et $1$, est de la forme $V_{M,{\cal L}}$
(cf.~rem.\,\ref{Hum1}),
 la cons\'equence suivante qui montre que:

$\bullet$ \`A part pour des caract\`eres lisses
provenant de $\Omega_{\rm Dr}\times\pi_0({\cal M}_\infty)$, {\it le socle de $H^1_{\eet}({\cal M}_\infty)$ ne contient que
des repr\'esentations supercuspidales de dimension~$2$, \`a poids~$0$ et~$1$}.

$\bullet$ {\it La cohomologie \'etale $p$-adique de la tour de Drinfeld 
encode la correspondance
de Langlands locale $p$-adique} pour les repr\'esentations supercuspidales de dimension~$2$, 
\`a poids~$0$ et~$1$.
\begin{coro}\label{emer4}
Soit $V$ une $L$-repr\'esentation absolument irr\'eductible de $\G_{\Q_p}$, de dimension~$\geq 2$.

{\rm (i)}
Si $V$ est supercuspidale, de dimension~$2$, \`a poids $0$ et $1$,
$${\rm Hom}_{L[{\rm W}_{\Q_p}]}(V,L\otimes_{\Q_p}H^1_{\eet}({\cal M}_\infty,\Q_p(1)))=
{\rm JL}(V)\otimes \Pi(V)^\dual.$$

{\rm (ii)} Dans le cas contraire, ${\rm Hom}_{L[{\rm W}_{\Q_p}]}(V,L\otimes_{\Q_p}H^1_{\eet}({\cal M}_\infty,\Q_p(1)))=0$.
\end{coro}

\subsubsection{D\'emonstration du th.~\ref{eliminate}}\label{emer5.2}
Passons \`a la preuve du th.~\ref{eliminate}.
On a d\'ej\`a prouv\'e 
que ${\rm Hom}_{L[\G_{\Q_p}]}(V_{M,{\cal L}},H^1_{\eet}[M])\neq 0$
(cf.~cor.~\ref{fonda2}).
Il s'ensuit que, si $V$ a un quotient isomorphe \`a $V_{M,{\cal L}}$, alors
${\rm Hom}_{L[\G_{\Q_p}]}(V,H^1_{\eet}[M])\neq 0$.
Il s'agit donc de prouver la r\'eciproque, et le (ii).

Soit donc $V$ une $L$-repr\'esentation de $\mathcal{G}_{\Q_p}$; notons simplement
$H_M$ et $H_C$ les modules $H_M(V)$ et $H_C(V)$ (\no\ref{Emer1}).
La prop.\,\ref{ajout2}
fournit une identification 
$${\rm Hom}_{L[\G_{\Q_p}]}(V, H^1_{\rm proet}[M])\simeq \{\omega\in H_C\otimes_{L} \Omega^1[M],\ 
\pi_{\rm dR}(\omega)\in H_M\otimes_L {\rm LL}(M)^\dual \}.$$    
    
Appliquer le foncteur ${\rm Hom}_{L[\G_{\Q_p}]} (V,-)$ \`a la suite exacte 
$$0\to V_{M,{\cal L}}\to X_{\rm st}^+(M)\to C\otimes_{\Q_p} (M_{\rm dR}/{\cal L})\to 0$$    
fournit le r\'esultat suivant.
\begin{lemm} \label{super} 
Si ${\cal L}$ est une $L$-droite de $M_{\rm dR}$, les conditions suivantes sont \'equivalentes:

{\rm (i)} ${\rm Hom}_{L[\G_{\Q_p}]} (V, V_{M, {\cal L}})=0$.

{\rm (ii)} $\iota:H_M\to M_{\rm dR}\otimes_L H_C$ 
induit une injection $H_M\to (M_{\rm dR}/{\cal L})\otimes_L H_C$.
\end{lemm}

Fixons un $\cal L_0$ tel que ${\rm Hom}_{L[\G_{\Q_p}]} (V, V_{M, {\cal L_0}})=0$. 
Soit $m_1$ une base de $\cal L_0$, et compl\'etons $m_1$ en une base 
$m_1, m_2$ de $M_{\rm dR}$. On peut \'ecrire $x\in H_M$ sous la forme 
$x=m_1\otimes a(x)+m_2\otimes b(x)$, avec $a(x),b(x)\in H_C$, et le lemme\,\ref{super}
implique que
$b: H_M\to H_C$ est injective. En notant $S={\rm Im}(b)\subset H_C$, 
on en d\'eduit l'existence d'une application lin\'eaire $A: S\to H_C$ telle que 
$$H_M=\{m_1\otimes A(x)+m_2\otimes x,\  x\in S\}.$$ 
        
Supposons maintenant que 
${\rm Hom}_{L[\G_{\Q_p}]} (V, V_{M, {\cal L}})=0$
pour tout $\cal L$ (en se permettant de remplacer $L$ par une extension finie). Nous allons montrer que $ {\rm Hom}_{L[\G_{\Q_p}]}(V, H^1_{\eet}[M])=0$. Par le lemme\,\ref{super}, l'application
$\iota: H_M\to M_{\rm dR}\otimes_L H_C$ induit une injection $ H_M\to (M_{\rm dR}/\cal L)\otimes_L H_C$ pour tout $\cal L$. Comme $$H_M=\{m_1\otimes A(x)+m_2\otimes x,\ x\in S\},$$ 
on en d\'eduit que $A-\lambda\cdot {\rm id}$ est injective pour tout $\lambda$,
et donc que le seul sous-espace de $S$ stable par $A$ est $0$.
     
Fixons $\lambda_1,\dots,\lambda_n\in L$ deux \`a deux distincts, avec $n=\dim H_C$. 
Gr\^ace au
lemme\,\ref{painful} ci-dessous, $A$ s'\'etend en un endomorphisme $u: H_C\to H_C$, diagonalisable,
de valeurs propres $\lambda_1,\dots,\lambda_n$.
 Quitte \`a changer de base dans $H_C$, on peut donc supposer que $u$ est repr\'esent\'e par la matrice diagonale de coefficients 
 $\lambda_1,...,\lambda_n$. Soit $\mathcal{L}_i$ la droite engendr\'ee par $m_2+\lambda_i m_1$ dans $M_{\rm dR}$. On obtient une inclusion $H_M\subset \oplus_{i=1}^n 
\mathcal{L}_i$ qui, combin\'ee avec 
l'identification (cf. ci-dessus)
$${\rm Hom}_{L[\G_{\Q_p}]}(V, H^1_{\rm proet}[M])\simeq \{\omega\in H_C\otimes_{L} \Omega^1[M],\ 
\pi_{\rm dR}(\omega)\in H_M\otimes_L {\rm LL}(M)^\dual \}\,,$$
induit un plongement (prop.\,\ref{dl2})
$$ {\rm Hom}_{L[\G_{\Q_p}]}(V, H^1_{\rm proet}[M])\subset \bigoplus_{i=1}^n (\Pi(V_{M, {\cal L_i}})^{\rm an})^\dual .$$
 En passant aux vecteurs $G$-born\'es comme dans la preuve du cor.\,\ref{fonda2},
on obtient un plongement 
 $$ {\rm Hom}_{L[\G_{\Q_p}]}(V, H^1_{\eet}[M])\subset \bigoplus_{i=1}^n \Pi(V_{M, {\cal L_i}})^\dual $$
Puisque les $\lambda_i$ sont arbitraires (deux \`a deux distincts) et les $ \Pi(V_{M, {\cal L_i}})^\dual $ sont des
$G$-modules topologiques simples, on en d\'eduit que  $ {\rm Hom}_{L[\G_{\Q_p}]}(V, H^1_{\eet}[M])=0$. 

Cela termine la preuve du (i) du th\'eor\`eme; passons au (ii).
On sait d\'ej\`a (cor.~\ref{fonda2}), que $\Pi_M(V_{M,{\cal L}})=\Pi(V_{M,{\cal L}})$.
Si $V$ n'a pas de quotient isomorphe \`a un $V_{M,{\cal L}}$, alors
$\Pi_M(V)=0$ d'apr\`es le (i).  Dans le cas contraire, on a
une suite exacte
$$0\to V'\to V\to V_{M, {\cal L}}\to 0$$
pour un certain ${\cal L}$, et on d\'emontre le r\'esultat, par r\'ecurrence
sur la dimension de $V$, en utilisant la suite exacte
$$0\to \Pi_M(V_{M, {\cal L}})^\dual \to \Pi_M(V)^\dual \to \Pi_M(V')^\dual $$
et les propri\'et\'es~\cite{STBanach} de la cat\'egorie des repr\'esentations unitaires admissibles
de $G$.
       
\smallskip
Pour conclure, il ne reste plus qu'\`a prouver le r\'esultat d'alg\`ebre lin\'eaire suivant.
\begin{lemm}\label{painful}
Soit $A:S\to H_C$ une application lin\'eaire telle que le seul sous-espace de $S$ stable par $A$ soit $0$,
et soient
$\lambda_1,\dots,\lambda_n\in L$, avec $n=\dim_L H_C$, deux \`a deux distincts.  Alors
$A$ admet un prolongement \`a $H_C$ de valeurs propres $\lambda_1,\dots,\lambda_n$.
\end{lemm}
\begin{proof}
La d\'emonstration se fait par r\'ecurrence sur $n$, le cas $n=1$ \'etant trivial.

Si $v\in S$ est non nul, on note $k(v)$ le plus grand entier tel que $v,Av,\dots,A^{k(v)}\in S$.
Si $\sum_{i\leq k(v)+1}\lambda_iA^iv=0$, on a $\lambda_{k(v)+1}=0$ puisque 
$A^{k(v)+1}v\notin S$, alors que $A^iv\in S$ si $i\leq k(v)$; on en d\'eduit que
$\lambda_i=0$ pour tout $i$ puisque $A$ ne laisse pas invariant le sous-espace
engendr\'e par $v,Av,\dots,A^{k(v)}v$.  Autrement dit,
$v,Av,\dots,A^{k(v)+1}v$ sont lin\'eairement ind\'ependants et, en particulier, $k(v)\leq n-2$.

Si $k(v)=n-2$, alors les $e_i=A^{i-1}v$, pour $1\leq i\leq n$, forment une base de $H_C$
dans laquelle la matrice de $A$ est une matrice \`a $n$ lignes et $n-1$ colonnes
avec des $1$ en dessous de la diagonale et des $0$ partout ailleurs.
Si $\sigma_1,\dots,\sigma_n$ sont les fonctions sym\'etriques de $\lambda_1,\dots,\lambda_n$,
poser $Ae_n=\sigma_ne_1+\sigma_{n-1}e_2+\cdots+\sigma_1e_n$ fournit une extension de $A$ ayant
les propri\'et\'es voulues.

Si $k(v)<n-2$, choisissons un suppl\'ementaire $S'$ de $S_v=Lv\oplus\cdots\oplus LA^{k(v)}v$
dans $S$.  Alors $AS'\cap (S_v+AS_v)=0$: en effet, si $Av'=\sum_{i=0}^{n+1}\lambda_iA^iv$,
alors $\lambda_0=0$, sinon $v\in AS$, ce qui est contraire \`a la maximalit\'e de $k(v)$,
et donc $v'\in S_v$ puisque $A$ est injective.
On peut donc trouver un suppl\'ementaire $H_C'$ de $S_v+AS_v$ dans $H_C$, qui contient $AS'$,
et appliquer l'hypoth\`ese de r\'ecurrence \`a $A:S_v\to (S_v+AS_v)$ et $A:S'\to H_C'$,
en partitionnant $\lambda_1,\dots,\lambda_n$ en deux ensembles de cardinaux
$\dim_L (S_v+AS_v)$ et $\dim_L H_C'$.

Ceci permet de conclure.
\end{proof}

\section{M\'ethodes perfecto\"{\i}des}\label{DEMI3}
     \subsection{Quelques faisceaux pro-\'etales} 
                        
   Consid\'erons un espace adique lisse $X/C$ (dans les applications $X$ sera une vari\'et\'e rigide 
   lisse, de dimension $1$). 
    On dispose \cite{RAV} du site pro-\'etale
   $X_{\proet}$ de $X$ et d'une projection 
   $\nu: X_{\proet}\to X_{\eet}$ vers le site \'etale $X_{\eet}$ de $X$, ainsi que 
   des faisceaux suivants sur $X_{\proet}$:
   
       $\bullet$ les faisceaux 
   $\Z_p=\varprojlim_{n} \mathbf{Z}/p^n$ et $\Z_p(1)=\varprojlim_{n} \mu_{p^n}$. Si $\mathcal{F}$ est un faisceau de $\Z_p$-modules sur 
   $X_{\proet}$, on note $\mathcal{F}(1)=\mathcal{F}\otimes_{\Z_p}Ê\Z_p(1)$.

     $\bullet$ le faisceau $\mathcal{O}_{X}^+$ (pullback 
    du faisceau $\mathcal{O}_{X}^+$ sur $X_{\eet}$), et ses compl\'et\'es  $\widehat{\mathcal{O}}_X^+=\varprojlim_{n} \mathcal{O}_X^+/p^n$ et $\widehat{\mathcal{O}}_X=\widehat{\mathcal{O}}_X^+\otimes_{\Z_p} \Q_p$.

      Nous allons utiliser syst\'ematiquement dans la suite le r\'esultat suivant de Scholze \cite[lemma 3.24]{survey}.

\begin{prop}\label{Scholze} Le morphisme naturel  $\mathcal{O}_{X_{\eet}}\to \nu_* \widehat{\mathcal{O}}_X$
est un isomorphisme, et on dispose d'un isomorphisme canonique
     $\mathcal{O}_{X_{\eet}}$-lin\'eaire $\Omega^1_{X_{\eet}}\simeq R^1\nu_*\widehat{\mathcal{O}}_X(1)$.
\end{prop}

Une cons\'equence tr\`es utile pour la suite est la suivante:

\begin{lemm} \label{five term}
On dispose d'une suite exacte canonique 
$$0\to H^1_{\eet}(X, \mathcal{O}_{X})\to H^1_{\proet}(X, \widehat{\mathcal{O}}_X)\to \Omega^1(X)(-1)\to H^2_{\eet}(X, \mathcal{O}_{X}).$$
Si $X$ est affino\"{\i}de ou Stein, alors on a un isomorphisme canonique 
$$H^1_{\proet}(X, \widehat{\mathcal{O}}_X)\simeq \Omega^1(X)(-1).$$
\end{lemm}

\begin{proof} La premi\`ere partie d\'ecoule de la suite spectrale 
$$H^i_{\eet}(X, R^j\nu_* \widehat{\mathcal{O}}_X)\Rightarrow  H^{i+j}_{\proet}(X, \widehat{\mathcal{O}}_X),$$
combin\'ee avec la proposition \ref{Scholze}, la seconde est une cons\'equence du th\'eor\`eme de Kiehl (dans le cas Stein) et Tate (si $X$ est affino\"{\i}de). 
\end{proof}

\subsection{Le diagramme (tautologique) fondamental} 
Soit $X$ une vari\'et\'e rigide analytique lisse sur $C$.
La construction des anneaux de Fontaine $\ainf$, $\bcris^+$, $\bdr^+$, etc. se faisceautise
et donne naissance \`a des faisceaux ${\mathbb A}_{\rm inf}$, ${\mathbb B}_{\rm cris}^+$,
${\mathbb B}_{\rm dR}^+$, etc. sur le site pro\'etale de $X$.

On a un diagramme de faisceaux pour la topologie pro\'etale, \`a lignes exactes: 
$$\xymatrix@R=.5cm{0\ar[r]&\Q_p(1)\ar[r]\ar[d]&({\mathbb B}_{\rm cris}^+)^{\varphi=p}
\ar[r]\ar[d]&\widehat\O\ar@{=}[d]\ar[r]&0\\
0\ar[r]&\widehat\O(1)\ar[r]&{\mathbb B}_{\rm dR}^+/t^2\ar[r]&\widehat\O\ar[r]&0}$$
Par ailleurs, $\Omega^1(X)(-1)$ est naturellement un quotient de $H^1_{\proet}(X,\widehat\O)$.
D\'efinissons le groupe $\tHK(X)$ par:
\begin{align*}
\widetilde{\rm HK}(X)
={\rm Ker}\big[H^1_{\proet}(X,({\mathbb B}_{\rm cris}^+)^{\varphi=p})\to
H^1_{\proet}(X,\widehat\O)\to\Omega^1(X)(-1)\big].
\end{align*}

\begin{theo}\label{main}
Soit $X/C$ une courbe lisse, Stein, avec un nombre fini de composantes connexes. Alors
on dispose d'un diagramme commutatif canonique,
\`a lignes exactes 
$$\xymatrix@R=.5cm@C=.6cm{
 0\ar[r]&C\otimes\Z[\pi_0(X)]\ar[r]\ar@{=}[d]&\mathcal{O}(X)\ar[r]^-{\rm exp}\ar@{=}[d]& H^1(X, \Q_p(1))\ar[r]\ar[d]^-{\rm dlog}&
\tHK(X)\ar[r]\ar[d]^-{\rm \iota_{\rm can}}&0\\
0\ar[r]&C\otimes\Z[\pi_0(X)]\ar[r]&\mathcal{O}(X)\ar[r]^-{d} & \Omega^1(X) \ar[r]^-{\pi_{\rm dR}}&
H^1_{\rm dR}(X)\ar[r]&0
}$$           
De plus:

$\bullet$ toutes les fl\`eches 
sont d'image ferm\'ee,

$\bullet$ ${\rm Ker}\,\iota_{\rm can}={\rm Ker}\,\dlog$ s'identifie naturellement \`a un
sous espace ferm\'e de $H^1_{\rm dR}(X)(1)$.
         \end{theo}
\begin{proof}         
On peut supposer que $X$ est connexe.
Si $X$ est une courbe Stein,
$H^0_{\proet}(X,\widehat \O)=\O(X)$ et 
$H^1_{\proet}(X,\widehat \O)=\Omega^1(X)(-1)$ (lemme \ref{five term}).
En passant \`a la cohomologie, et en posant
\begin{align*}
\widetilde{\rm DR}(X)={\rm Ker}\big[
H^1_{\proet}(X,{\mathbb B}_{\rm dR}^+/t^2)\to
H^1_{\proet}(X,\widehat\O)\big],
\end{align*}
on obtient un diagramme
$$\xymatrix@R=.6cm@C=.5cm{0\ar[r]&
\Q_p(1)\ar[r]\ar[d]&H^0_{\proet}(X,({\mathbb B}_{\rm cris}^+)^{\varphi=p})\ar[r]\ar[d]
&\O(X)\ar[r]\ar@{=}[d] &H^1_{\proet}(X,\Q_p(1))\ar[r]\ar[d]& \widetilde{\rm HK}(X)\ar[r]\ar[d]&0\\
0\ar[r]&\O(X)(1)\ar[r]& H^0_{\proet}(X,{\mathbb B}_{\rm dR}^+/t^2)
\ar[r]&\O(X)\ar[r]&\Omega^1(X)\ar[r]&\widetilde{\rm DR}(X)\ar[r]&0}$$
Dans ce diagramme:

$\bullet$ $\O(X)\to H^1_{\proet}(X,\Q_p(1))$ est l'exponentielle: plus pr\'ecis\'ement,
si $g\in\O(X)$, et si $U$ est un ouvert affino\"{\i}de de $X$, alors
$p^{-n}\otimes\exp(p^n g)$ d\'efinit, si $n\gg 0$, un \'el\'ement de $\Q_p\otimes\O(U)^\dual$
qui ne d\'epend pas de $n$ et dont le symbole
en cohomologie (pro)\'etale (donn\'e par l'application de Kummer) ne d\'epend pas non plus
de $n$; les classes ainsi obtenues se recollent et l'application ci-dessus s'obtient
par limite inverse sur un recouvrement croissant de $X$ par des affino\"{\i}des.

$\bullet$ D'apr\`es \cite[lemma 3.24]{survey}, $H^1_{\proet}(X,\Q_p(1))\to \Omega^1(X)$ 
est compatible avec les symboles: si $f\in\O(X)^\dual$ elle envoie le symbole $(f)_{\eet}$
de $f$ en cohomologie pro\'etale (donn\'ee par l'application de Kummer)
sur $\frac{df}{f}$.

Il en r\'esulte que $\O(X)\to\Omega^1(X)$ (en bas) est juste $d$ et donc que
$\widetilde{\rm DR}(X)=H^1_{\rm dR}(X)$.

On en d\'eduit que le noyau de $\exp$ est juste les constantes,
ce qui nous donne le diagramme commutatif annonc\'e, et permet aussi de prouver que
\begin{equation}\label{H0}
H^0_{\proet}(X,({\mathbb B}_{\rm cris}^+)^{\varphi=p})=(\bcris^+)^{\varphi=p}\otimes\Z[\pi_0(X)].
\end{equation}

Le fait que $d$ soit d'image ferm\'ee est classique, cf. \cite[cor. 3.2]{GKDR}.
Montrons que
 ${\rm exp}$ est d'image ferm\'ee. Si
$f_n\in \mathcal{O}(X)$ et $c\in H^1_{\proet}(X,\Q_p(1))$ sont tels que $\lim_{n\to \infty} {\rm exp}(f_n)=c$, alors 
$\lim_{n\to\infty} {\rm dlog}({\rm exp}(f_n))={\rm dlog}(c)$ dans $\Omega^1(X)$, et donc
${\rm lim}_{n\to\infty} df_n={\rm dlog}(c)$, ce qui montre que  
$f_n$ converge dans $\mathcal{O}(X)/C$. Ainsi $c={\rm exp}(\lim_{n\to \infty} f_n)$,
d'o\`u le r\'esultat.

Le fait que $\dlog$ et $\iota_{\rm can}$ soient d'image ferm\'ee est prouv\'e
au \no\ref{GRAB2} et l'\'enonc\'e concernant le noyau de $\iota_{\rm can}$ au \no\ref{GRAB3}.
\end{proof}
   
\subsubsection{Exemples concrets: disques et couronnes ouverts} \label{GRAB1}
       
Soit $D=\{|z|<1\}$ le disque unit\'e ouvert sur $C$. Nous allons calculer
$H^1_{\proet}(D, \Q_p(1))$ et d\'ecrire les applications ${\rm dlog}$ et 
${\rm exp}$ introduites ci-dessus.
         
Soit $B_r$ la boule ferm\'ee de rayon $r<1$. Alors 
$$H^1_{\proet}(D, \Q_p(1))=\varprojlim_{r\to 1} H^1_{\proet}(B_{r}, \Q_p(1))=
\varprojlim_{r\to 1} H^1_{\eet}(B_{r}, \Q_p(1)).$$
La suite de Kummer fournit un isomorphisme 
$$H^1_{\eet}(B_{r}, \Q_p(1))\simeq {(\mathcal{O}(B_r)^\dual /C^\dual )}\widehat\otimes_{\Z_p}\Q_p.$$
La d\'ecomposition $\mathcal{O}(B_r)^\dual =C^\dual \O(B_r)^{\dual\dual}_1$, o\`u $\O(B_r)^{\dual\dual}_1=\{f\in \O(B_r)^{\dual\dual},\ f(1)=1\}$,
fournit une application $\log: \mathcal{O}(B_r)^\dual /C^\dual\to \mathcal{O}(B_r)/C$ (la s\'erie
d\'efinissant le logarithme converge si $f\in \O(B_r)^{\dual\dual}_1$).
Cette application {\it ne se prolonge
 pas au compl\'et\'e $p$-adique}, mais elle se prolonge pour tout $r'<r$ en une application 
 $\log: {(\mathcal{O}(B_r)^\dual /C^\dual )}\widehat\otimes_{\Z_p}\Q_p\to \mathcal{O}(B_{r'})/C$ (si 
 $f\in \O(B_r)^{\dual\dual}_1$, alors $v_{B_{r'}}(f-1)\geq c$ pour une constante $c>0$, qui d\'epend de 
 $r$ et $r'$). En passant \`a la limite projective, on obtient une application 
 $$\log: H^1_{\proet}(D, \Q_p(1))\to \varprojlim_{r'\to 1} \mathcal{O}(B_{r'})/C=\mathcal{O}(D)/C.$$
En utilisant la compatibilit\'e de $H^1_{\proet}(X, \Q_p(1))\to \Omega^1(X)$ avec les symboles
et la d\'efinition de l'application $\exp$ ci-dessus (preuve du th.\,\ref{main}),
on en d\'eduit le r\'esultat suivant.
 
 \begin{prop}
Les applications ${\rm exp}$ et ${\rm log}$ induisent un isomorphisme 
$$\mathcal{O}(D)/C\simeq H^1_{\proet}(D, \Q_p(1)), $$
l'application ${\rm dlog}$ s'identifiant \`a $d\circ \log$, o\`u 
$d: \mathcal{O}(D)/C\to \Omega^1(D)$ est la d\'erivation. En particulier, 
${\rm dlog}: H^1_{proet}(D, \Q_p(1))\to  \Omega^1(D)$ est un hom\'eomorphisme sur son image, qui est ferm\'ee. 
\end{prop}
 
On laisse au lecteur le soin de d\'emontrer (avec les m\^emes arguments):
   
\begin{prop}\label{couronne}
Soit $\mathcal{C}$ une couronne ouverte (des deux c\^ot\'es) sur $C$. 
On a une suite exacte 
$$0\to \mathcal{O}(\mathcal{C})/C\to H^1_{\proet}(\mathcal{C}, \Q_p(1))\to \Q_p\to 0,$$
et l'application ${\rm dlog}:  H^1_{\proet}(\mathcal{C}, \Q_p(1))\to \Omega^1(\mathcal{C})$ est un 
hom\'eomorphisme sur son image, qui est ferm\'ee. 
\end{prop}

\Subsubsection{L'image de $\dlog$}\label{GRAB2}
\begin{conv}
 On \'ecrit simplement 
$H^1(Z)$ au lieu de $H^1_{\proet}(Z, \Q_p(1))$ dans la suite de ce paragraphe.
\end{conv}
Que l'image de $\iota_{\rm can}$ soit ferm\'ee \'equivaut \`a ce que celle de ${\rm dlog}$ le soit;
nous allons montrer que celle de $\dlog$ l'est.
Cela repose sur le r\'esultat suivant.

 \begin{prop}\label{EST} Soit $U\Subset V$ une inclusion stricte d'affinoides lisses de dimension $1$ sur 
$C$. 

{\rm a)} Si $c_n\in H^1(V)$ sont tels que
$\lim_{n\to \infty} {\rm dlog}(c_n)=0$ dans $\Omega^1(V)$, alors il existe $d_n\in H^1(U)^{{\rm dlog}=0}$ tels que 
$\lim_{n\to \infty} (c_n|_{U}+d_n)=0$ dans $H^1(U)$.

{\rm b)} L'image de $H^1(V)^{\rm dlog=0}$ dans $H^1(U)^{\rm dlog=0}$ est de dimension finie sur $\Q_p$.
     \end{prop}

      \begin{proof}   a) On peut supposer que $U$ est connexe. 
       D'apr\`es \cite{vdp} on peut trouver une courbe propre et lisse, connexe $Z/C$, telle que 
       $D:=Z\moins U$ soit une r\'eunion disjointe finie 
   de disques ouverts $D_1,...,D_s$. Fixons des param\`etres 
$z_i$ sur $D_i$ pour identifier chaque $D_i$ au disque $\{|z_i|<1\}$. Soit $r<1$ et consid\'erons 
l'affino\"{\i}de $U(r)$, compl\'ementaire de la r\'eunion des disques ouverts $D_i(r)\subset D_i$ de rayon 
$r$. Les $U(r)$ (pour $r\to 1$) forment une base de voisinages stricts de $U$, donc on peut supposer que 
$V=U(r)$ pour un $r<1$. Notons  
 $\mathcal{C}(r)=D\cap U(r)$, une r\'eunion disjointe de couronnes $\mathcal{C}_i(r)$ d\'efinies par $r\leq |z_i|<1$, et soit $\mathcal{C}(r)^{-}$ la r\'eunion des couronnes ouvertes associ\'ee.
   
     Le recouvrement (admissible) de $Z$ par $D$ et $U(r)$ induit      
     une 
           suite de Mayer-Vietoris s'ins\'erant dans un diagramme commutatif 
            $$\xymatrix@R=.6cm{\ H^1(Z)\ar [r]^-{\alpha} \ar [d]^{\rm dlog}&
H^1(U(r))\oplus H^1(D)\ar[d]^{{\rm dlog}}\ar [r]^-{\beta}&
H^1(\mathcal{C}(r))\ar[d]^{\rm dlog}\ar[r]&H^2(Z)\\
\Omega^1(Z)\ar[r]&\Omega^1(U(r))\oplus \Omega^1(D)\ar[r]&\Omega^1(\mathcal{C}(r))
}$$

   Soient maintenant $c_n\in H^1(U(r))$ tels que $\lim_{n\to \infty} {\rm dlog}(c_n)=0$. 
   Alors ${\rm dlog}(\beta(c_n, 0))$ tend vers $0$ dans $\Omega^1(\mathcal{C}(r))$. 
   La proposition \ref{couronne} montre alors que la restriction de $\beta(c_n,0)$ \`a 
   la couronne ouverte $\mathcal{C}(r)^{-}$ tend vers $0$. Fixons $r'\in (r,1)$.
   Nous avons un diagramme analogue \`a celui ci-dessus avec 
   $r$ remplac\'e par $r'$, et on note $\alpha', \beta'$ les applications correspondantes, et $c_n'$ la restriction de $c_n$ \`a $U(r')$. On vient de montrer que $\beta'(c_n', 0)$ tend vers $0$ dans $H^1(\mathcal{C}(r'))$. Mais $\beta'$ est un hom\'eomorphisme sur son image (lemme \ref{gens de l'est 1} ci-dessous), d'o\`u l'existence d'une suite $x_n\in H^1(Z)$ telle que 
   $\alpha'(x_n)+(c_n',0)\to 0$ dans 
    $H^1(U(r'))\oplus H^1(D)$. Le diagramme ci-dessus montre alors que ${\rm dlog}(x_n)\to 0$. En appliquant encore le lemme \ref{gens de l'est 1} on obtient l'existence de $y_n\in H^1(Z)^{\rm dlog=0}$ tels que $x_n-y_n\to 0$ dans 
    $H^1(Z)$. Il suffit de poser $d_n=y_n|_{U}$, alors $d_n\in H^1(U)^{\rm dlog=0}$ et par construction 
    $d_n+c_n|_U\to 0$. 
    
    b) Nous gardons les notations introduites ci-dessus. Si $c\in H^1(U(r))^{\rm dlog=0}$, alors la restriction de $\beta(c,0)$ \`a 
    $\mathcal{C}(r)^{-}$ est nulle (prop. \ref{couronne}). Le diagramme ci-dessus montre que la restriction de 
    $c$ \`a $U(r')$ provient d'une classe de $H^1(Z)$. On en d\'eduit que l'image de $H^1(V)^{\rm dlog=0}$ dans $H^1(U)^{\rm dlog=0}$ est un quotient de $H^1(Z)$, donc de dimension finie. 
\end{proof}

\begin{lemm} \label{gens de l'est 1}
Soit $f: U\to V$ une application continue entre des $L$-fr\'echets. Si le conoyau est
de dimension finie, alors ${\rm Im}(f)$ est ferm\'ee et $f$ induit un hom\'eomorphisme 
$U/{\rm Ker}(f)\overset{\sim}\to {\rm Im}(f)$.
\end{lemm}
\begin{proof} Cons\'equence directe du th\'eor\`eme de l'image ouverte.
\end{proof}

Revenons \`a la preuve du fait 
que ${\rm dlog}$ est d'image ferm\'ee. L'argument est familier, mais pour le confort du lecteur on donne les d\'etails. Soit $(U_k)_{k\geq 1}$ un recouvrement croissant de type Stein de $X$.
      On notera $x^{(k)}$ la restriction \`a $U_k$ d'une classe $x$ d\'efinie sur un ouvert contenant $U_k$.     
      Soient 
     $c_n\in H^1(X)$ et $\omega\in \Omega^1(X)$ tels que $\lim_{n\to\infty} {\rm dlog}(c_n)=\omega$.
     
      Fixons $k\geq 1$. On a 
    $\lim_{n\to\infty}{\rm dlog}(c_n^{(k+1)})=\omega^{(k+1)}$, donc $\lim_{n\to\infty} {\rm dlog}(c_{n+1}^{(k+1)}-c_n^{(k+1)})=0$. La proposition \ref{EST} fournit une suite $(d_n^{(k)})$ d'\'el\'ements de $H^1(U_k)^{{\rm dlog}=0}$ telle que 
    $\lim_{n\to\infty} c_{n+1}^{(k)}-c_n^{(k)}-d_n^{(k)}=0$. En posant $e_n^{(k)}=d_1^{(k)}+...+d_{n-1}^{(k)}$, il s'ensuit que 
    $s_k:=\lim_{n\to\infty} (c_n^{(k)}-e_n^{(k)})$ existe dans $H^1(U_k)$, et puisque ${\rm dlog}(e_n^{(k)})=0$ on a  
    $${\rm dlog}(s_k)=\lim_{n\to\infty} {\rm dlog}(c_n^{(k)})=\omega^{(k)}.$$
    
      On a donc obtenu une suite de classes $s_k\in H^1(U_k)$ telles que ${\rm dlog}(s_k)=\omega^{(k)}$ pour tout 
      $k$. Nous avons $s_{k+1}^{(k)}-s_k\in H^1(U_k)^{\rm dlog=0}$ pour tout $k$. Comme 
      $R^1\varprojlim H^1(U_k)^{\rm dlog=0}=0$ (d'apr\`es le b) de la proposition \ref{EST}), il existe 
      $\alpha_k\in H^1(U_k)^{\rm dlog=0}$ tels que $u:=(s_1-\alpha_1, s_2-\alpha_2,...)\in \varprojlim_{k} H^1(U_k)=H^1(X)$. 
      On obtient donc une section globale $u\in H^1(X)$ telle que $u^{(k)}=s_k-\alpha_k$ pour tout $k$, et donc 
      ${\rm dlog}(u^{(k)})={\rm dlog}(s_k)=\omega^{(k)}$ pour tout $k$. On en d\'eduit que ${\rm dlog}(u)=\omega$, ce qui montre que ${\rm dlog}$ est d'image ferm\'ee.

\subsubsection{Le noyau de $\iota_{\rm can}$}\label{GRAB3}
      Nous aurons aussi besoin de contr\^oler le noyau de $\iota_{\rm can}$ ou, ce qui revient au m\^eme gr\^ace au diagramme du 
th.\,\ref{main}, le noyau de ${\rm dlog}$. Soit $X/C$ un espace adique lisse, affino\"{\i}de ou Stein. 
Le morphisme $\Q_p(1)\to\widehat\O(1)$ utilis\'e pour d\'efinir $\dlog$
se factorise \`a travers $\Q_p(1)\to(\Bdr^+/t^2)(1)\to \widehat\O(1)$.
D'o\`u une application naturelle (cf.\,preuve du th.\,\ref{main} (d\'ebut)
pour l'isomorphisme $\widetilde{\rm DR}\cong H^1_{\rm dR}(X)$)
$$\lambda_X: H^1(X)^{\rm dlog=0}\to \widetilde{\rm DR}(1)\cong H^1_{\rm dR}(X)(1).$$
\begin{prop}\label{noyau plonge} 
Si $X/C$ est un espace Stein, lisse de dimension $1$, alors
$\lambda_X$ est injective, d'image ferm\'ee. 
\end{prop}
      
\begin{proof} Si 
$X=\cup_{k} U_k$ est un recouvrement de type Stein de $X$, alors 
$\lambda_X$ est la limite inverse des $\lambda_{U_k}$. 
En raisonnant comme dans la preuve du fait que l'image de $\iota_{\rm can}$ est ferm\'ee,
on ram\`ene l'injectivit\'e de $\lambda_X$ \`a l'\'enonc\'e local ci-dessous 
(le fait que $\lambda_X$ est d'image ferm\'ee d\'ecoule facilement du fait que l'image de 
$H^1(U_{k+1})^{\rm dlog=0}$ dans $H^1(U_k)^{\rm dlog=0}$ est de dimension finie, cf. prop. \ref{EST}). 
\end{proof}

         \begin{lemm}
          Soit $U\Subset V$ une inclusion stricte d'affinoides lisses sur $C$, de dimension $1$. 
         Si $c\in H^1(V)^{\rm dlog=0}$ est tel que $\lambda_V(c)=0$, alors $c|_U=0$ dans $H^1(U)$. 
          \end{lemm}
         
         \begin{proof} Reprenons les notations du paragraphe 1 de la preuve de la proposition \ref{EST}: 
on a donc $U=Z\moins D$, avec $Z$ une courbe propre et lisse, connexe. On peut          
         supposer que $V=U(r)$ pour un $r<1$. Fixons $r'\in (r,1)$. On va montrer que 
         $c|_{U(r')}=0$ dans $H^1(U(r'))$, ce qui permettra de conclure. 
         Comme nous l'avons d\'ej\`a remarqu\'e (cf. la preuve du point b) de la proposition \ref{EST}), la restriction $c|_{U(r')}$ de $c$ \`a $U(r')$ se rel\`eve en une classe $s\in H^1(Z)$, nulle en restriction \`a $D$. Consid\'erons le diagramme obtenu \`a partir de la suite de Mayer-Vietoris pour le recouvrement de 
         $Z$ par $U(r')$ et $D$, o\`u l'on a pos\'e 
$\mathcal{F}= \mathbb{B}_{\rm dR}^+/t^2(1)$. 
        $$\xymatrix@C=.4cm@R=.5cm{\ H^0_{\proet}(\mathcal{C}(r'))\ar [r]\ar [d]&
H^1(Z)\ar[d]\ar [r]&
H^1(U(r'))\oplus H^1(D)\ar[d]\\
H^0_{\proet}(\mathcal{C}(r'), \mathcal{F})\ar[r] \ar[d]& H^1_{\proet}(Z, \mathcal{F}) \ar[r] \ar[d]&
H^1_{\proet}(U(r'),\mathcal{F})\oplus H^1_{\proet}(D, \mathcal{F}) \ar[d]\\
 H^0_{\proet}(\mathcal{C}(r'), \widehat{\mathcal{O}}(1)) \ar[r] & H^1_{\proet}(Z,  \widehat{\mathcal{O}}(1))\ar[r]& H^1_{\proet}( U(r'),  \widehat{\mathcal{O}}(1))\oplus H^1_{\proet}(D, \widehat{\mathcal{O}}(1))
}$$
 Comme $\lambda_{U(r')}(c)=0$, l'image $s'$ de $s$ 
dans $H^1_{\proet}(Z, \mathcal{F})$ vient d'une section $\alpha\in H^0_{\proet}(\mathcal{C}(r'), \mathcal{F})$. 
La compos\'ee $$H^0_{\proet}(\mathcal{C}(r'), \mathcal{F})\to  H^0_{\proet}(\mathcal{C}(r'), \widehat{\mathcal{O}}(1)) \to 
H^1_{\proet}(Z, \widehat{\mathcal{O}}(1))$$ est nulle, puisque l'image de $H^0_{\proet}(\mathcal{C}(r'), \mathcal{F})\to  H^0_{\proet}(\mathcal{C}(r'), \widehat{\mathcal{O}}(1)) $ s'identifie \`a 
${\rm Ker}(\mathcal{O}(\mathcal{C}(r'))(1)\to \Omega^1(\mathcal{C}(r'))(1))$ (utiliser la suite exacte 
$0\to \widehat{\mathcal{O}}(2)\to \mathcal{F}\to \widehat{\mathcal{O}}(1)\to 0$), et ce dernier espace est contenu dans  
${\rm Ker}(H^0_{\proet}(\mathcal{C}(r'), \widehat{\mathcal{O}}(1)) \to 
H^1_{\proet}(Z, \widehat{\mathcal{O}}(1)))$ (car l'image de 
$\mathcal{O}(D)(1)$ est contenue dans le dernier noyau). Donc $s$ a une image nulle dans 
$H^1_{\proet}(Z, \widehat{\mathcal{O}}(1))$ et puisque 
$H^1(Z)\to H^1_{\proet}(Z, \widehat{\mathcal{O}}(1))$ est injective \cite{RAV}, on obtient
$s=0$. Cela permet de conclure.                            
         \end{proof}
          
\Subsection{Courbes propres}\label{GRAB4}
\begin{prop} \label{Shimura}     
Si $X=X_K\otimes_{K} C$, o\`u $K$ est une extension finie de $\Q_p$ et $X_K$ est une courbe propre et lisse sur 
 $K$, alors on a des isomorphismes naturels 
$$ \tHK(X)\simeq (\bst^+\otimes_{\Q_p^{\rm nr}} D_{\rm pst}(H^1_{\eet}(X, \Q_p)))^{\varphi=p, N=0}=
(\bst^+\otimes_{\Q_p^{\rm nr}} H^1_{\rm HK}(X))^{\varphi=p, N=0}.$$
\end{prop}
     
\begin{proof} 
Soient $V=H^1_{\eet}(X, \Q_p)$, $M=D_{\rm pst}(V)$ et $M_{\rm dR}=D_{\rm dR}(V)$. 
Alors $V$ est une repr\'esentation potentiellement semi-stable
\`a poids de Hodge-Tate 
$0,-1$ et ${\rm Fil}^1(M_{\rm dR})\cong \Omega^1(X_K)$ et $M\cong H^1_{\rm HK}(X)$
d'apr\`es le th\'eor\`eme de comparaison semi-stable~\cite{Ts}.

On a un diagramme commutatif
   $$\xymatrix@R=.6cm@C=.5cm{
 0 \ar[r]& H^1_{\eet}(X, \Q_p)(1)\ar[r]\ar[d]^{\wr}& (\bcris^+)^{\varphi=p}\otimes H^1_{\eet}(X, \Q_p)
\ar[r]\ar[d]& C \otimes_{\Q_p}H^1_{\eet}(X, \Q_p) \ar[d]^{\alpha_X}_{\wr} \ar[r]& 0\\
0\ar[r] & H^1_{\proet}(X, \Q_p(1)) \ar[r]&
H^1_{\proet}(X,  (\mathbb{B}_{\rm cris}^+)^{\varphi=p})\ar[r]& H^1_{\proet}(X, \widehat{\mathcal{O}})
}$$  
dans lequel les lignes sont exactes et $\alpha_X$ est un isomorphisme, cf. \cite{RAV}.
D'o\`u un isomorphisme naturel:
$$(\bcris^+)^{\varphi=p}\otimes_{\Q_p}H^1_{\eet}(X, \Q_p)\overset{\sim}{\to} H^1_{\proet}(X, (\mathbb{B}_{\rm cris}^+)^{\varphi=p}).$$
        L'isomorphisme
     $\bst\otimes_{\Q_p} V\simeq \bst\otimes_{\Q_p^{\rm nr}} M$
nous donne
\begin{align*}
      (\bcris^+)^{\varphi=p}\otimes_{\Q_p} &V
\simeq {\rm Fil}^0 (\bst\otimes_{\Q_p^{\rm nr}} M)^{\varphi=p, N=0}\\
     &= (\bst\otimes_{\Q_p^{\rm nr}} M)^{\varphi=p, N=0}\cap 
\big(\bdr^+\otimes_K M_{\rm dR}+\tfrac{1}{t} \bdr^+\otimes_K \Omega^1(X_K)\big).
\end{align*}
On en d\'eduit que le noyau du morphisme\footnote{Induit par l'inclusion 
$(\bcris^+)^{\varphi=p}\otimes_{\Q_p} V\subset \bdr^+\otimes_K
M_{\rm dR}+\frac{1}{t} \bdr^+\otimes_K \Omega^1(X_K)$ et l'application $x\mapsto \theta(tx)(-1)$.} 
$(\bcris^+)^{\varphi=p}\otimes_{\Q_p} V\to \Omega^1(X_K)(-1)$ est 
$$(\bst\otimes_{\Q_p^{\rm nr}} M)^{\varphi=p, N=0}\cap 
\big(\bdr^+\otimes_K M_{\rm dR}\big)=(\bst^+\otimes_{\Q_p^{\rm nr}} M)^{\varphi=p, N=0}.$$
Mais, d'apr\`es ce qui pr\'ec\`ede, ce noyau s'identifie \`a 
    $\tHK(X)$, d'o\`u le r\'esultat.       
         \end{proof}

\section{Cohomologie de de Rham \`a support compact de ${\cal M}_\infty$}\label{GAB1.1}
\Subsection{Cohomologie des tours de Drinfeld et de Lubin-Tate} \label{GAB1}
    Soit $({\rm LT}_j)_{j\geq 0}$ la tour de Lubin-Tate.
On note $\widehat{\rm LT}_\infty$ le compl\'et\'e de la limite projective ${\rm LT}_\infty$ de la tour des
$({\rm LT}_j)_{j\geq 0}$, 
et $\widehat{\cal M}_{\infty}$ celui de ${\cal M}_\infty$.
Les tours compl\'et\'ees $\widehat{\rm LT}_\infty$ et $\widehat{\cal M}_{\infty}$ sont des espaces perfecto\"{\i}des~\cite[th. 6.5.4]{SW},
munis d'actions de $G\times \check{G}$ (voir \cite[chap. 3]{Dat} pour les d\'etails concernant ces actions),
et isomorphes en tant que $(G\times\check G)$-espaces perfecto\"{\i}des~\cite[th. 7.2.3]{SW}
(cet isomorphisme pr\'ecise l'isomorphisme de Faltings-Fargues \cite{Faltings2tours, Fargues} et nous utiliserons pleinement ce raffinement). 

Fixons une uniformisante $\varpi$ de $F$ et notons
simplement
$\widehat{\rm LT}_{\infty}^\varpi$ et $\widehat{\cal M}_{\infty}^\varpi$
(resp.~${\rm LT}_j^\varpi$ et ${\cal M}_j^\varpi$) les quotients de
$\widehat{\rm LT}_{\infty}$ et $\widehat{\cal M}_{\infty}$
(resp.~${\rm LT}_j$ et ${\cal M}_j$) par
       l'action de $\varpi$ vu comme \'el\'ement du centre de $G$ (ou de $\check G$,
cela revient au m\^eme).

Enfin, posons 
$$G_j=\begin{cases}{\bf GL}_2(\O_F)&{\text{si $j=0$,}}\\
1+\varpi^j{\bf M}_2(\O_F)&{\text{ si $j\geq 1$,}}\end{cases}
\quad
\check G_n=\begin{cases}\O_D^\dual&{\text{si $n=0$,}}\\
1+\varpi_D^n\O_D&{\text{ si $n\geq 1$.}}\end{cases}$$ 

Le but de ce chapitre est d'\'etablir le r\'esultat suivant:
    
\begin{theo} \label{comparer}
Si $j,n\in\N$,
on a des isomorphismes naturels:
$$H^1_{\rm dR, c}(\mv_n)^{G_j}\cong H^1_c(\widehat{\cal M}^\varpi_\infty,\O)^{G_j\times\check G_n}
\cong H^1_c(\widehat{\rm LT}^\varpi_\infty,\O)^{G_j\times\check G_n}\cong
H^1_{\rm dR, c}({\rm LT}_j^\varpi)^{\check G_n},$$
et donc
un isomorphisme naturel 
$$H^1_{\rm dR, c}({\rm LT}_\infty^\varpi)\simeq H^1_{\rm dR, c}({\cal M}_\infty^\varpi)$$
qui munit les deux membres d'une structure de
$C[\check G\times G]$-module lisse, admissible d\'ej\`a en tant que $G$-module.
\end{theo}

\Subsection{Quelques rappels sur la tour de Lubin-Tate} 
 Nous faisons des rappels tr\`es succincts sur la tour de Lubin-Tate, en renvoyant le lecteur \`a \cite{Drinfeld, Dat, Weinstein} pour plus de d\'etails. 
    La th\'eorie de Lubin-Tate fournit un sch\'ema formel 
    ${\rm Spf}(A_0)$
    sur $\mathcal{O}_{\breve{F}}$, classifiant les d\'eformations par quasi-isog\'enies de
  l'unique $\mathcal{O}_F$-module formel de dimension $1$ et de hauteur $2$ (relativement \`a $F$) sur 
    $\overline{\mathbf{F}}_p$. 
    En ajoutant des structures de niveau \`a ces d\'eformations, Drinfeld \cite{Drinfeld} a construit 
        une tour de sch\'emas formels $({\rm Spf}(A_n))_{n\geq 1}$ sur 
         $\mathcal{O}_{\breve{F}}$. La fibre g\'en\'erique adique $\breve{{\rm LT}}_n$ de ${\rm Spf}(A_n)$ est un rev\^etement fini \'etale galoisien,
     de groupe de Galois ${\rm GL}_2(\mathcal{O}_F/\varpi^n)$ de $\breve{{\rm LT}}_0$, ce dernier \'etant une r\'eunion disjointe d\'enombrable de copies du disque unit\'e ouvert sur $\breve{F}$. On note ${\rm LT}_n=\breve{{\rm LT}}_n\times_{\breve{F}} C$.
     On peut d\'efinir (voir \cite{Dat}) une action de $G\times \check{G}$ sur la tour $({\rm LT}_n)_{n\geq 0}$ (contrairement \`a la tour de Drinfeld, on n'a pas d'action de $G\times \check{G}$ sur chaque \'etage de la tour). Soit $\widehat{\rm LT}_\infty$ le compl\'et\'e de la limite projective de la tour des
$({\rm LT}_j)_{j\geq 0}$. 
     
\begin{prop}\label{wein1}
{\rm (i)} Si $U$ est un domaine rationnel de ${\rm LT}_j$, 
l'image inverse de $U$ dans $\widehat{\rm LT}_\infty$ est un affino\"{\i}de perfecto\"{\i}de.

{\rm (ii)} Si $U$ est un domaine rationnel de ${\cal M}_j$, 
l'image inverse de $U$ dans $\widehat{\cal M}_\infty$ est un affino\"{\i}de perfecto\"{\i}de.
\end{prop}
\begin{proof}
La structure de niveau universelle sur $A_1$ donne naissance \`a deux \'el\'ements
$X_1,Y_1$ de $A_1$. Weinstein \cite[lemme. 2.10.1]{Weinstein} a montr\'e que les sous-espaces 
de $\widehat{\rm LT}_\infty$
d\'efinis par les
in\'egalit\'es $|X_1|^n\leq |\varpi|$ et $|Y_1|^n\leq |\varpi|$ sont des affino\"ides perfecto\"{\i}des.
Il s'ensuit que $\widehat{\rm LT}_\infty$ est une r\'eunion croissante d'affino\"ides perfecto\"{\i}des
et donc que tout domaine rationnel de $\widehat{\rm LT}_\infty$ est affino\"{\i}de 
perfecto\"{\i}de~\cite[th. 6.3]{ScholzeIHES}.  On en d\'eduit le (i); le (ii) s'ensuit
gr\^ace \`a l'isomorphisme $\widehat{\cal M}_\infty\cong \widehat{\rm LT}_\infty$ d'espaces
perfecto\"{\i}des.
\end{proof}

\begin{rema}\label{wein2}
Si $X$ est une composante connexe de ${\cal M}_0$ ou ${\rm LT}_0$, on dispose d'une famille
de domaines rationnels $U_\lambda$, $\lambda\in\Q_+$, strictement croissante, telle
que 
$U_\lambda\moins \ocirc U_{\lambda'}$
est rationnel si $\lambda'<\lambda$, $X=\cup_{\lambda\in\Q_+}U_\lambda$ et $U_{\lambda}=\cap_{\lambda'>\lambda}U_{\lambda'}$.  
En particulier, $X$ et $X\moins U_\lambda$
sont des r\'eunions croissantes strictes de domaines rationnels.

On en d\'eduit, en utilisant la prop.~\ref{wein1}, que si $j\geq 0$, alors
$\widehat X=\widehat{\rm LT}_\infty^\varpi\cong
\widehat{\cal M}_\infty^\varpi$ admet un recouvrement croissant
par des affino\"ides perfecto\"{\i}des $\widehat U_n$ 
images inverses d'affino\"{\i}des $U_n$ de ${\rm LT}_j^\varpi$ (resp.~${\cal M}_j^\varpi$)
tels que $\widehat X\moins \widehat U_n$ admette aussi
un recouvrement croissant
par des affino\"ides perfecto\"{\i}des 
images inverses d'affino\"{\i}des de ${\rm LT}_j^\varpi$ (resp.~${\cal M}_j^\varpi$).
\end{rema}

\Subsection{Cohomologie de $\mathcal{O}(\widehat{\cal M}^\varpi_{\infty})$ et cohomologie
de de Rham de $\mv_\infty$}\label{GAB2}
\subsubsection{Cohomologie galoisienne de $\mathcal{O}(\widehat{\cal M}^\varpi_{\infty})$}
Si $K$ est un groupe profini, on note simplement $H^i(K,-)$ les groupes de cohomologie continue de $K$.
   Nous aurons besoin du r\'esultat technique suivant.
      
\begin{lemm}\label{scholzestyle}
Soit $X$ un affino\"{\i}de ou une courbe Stein, et soit $\widehat X$ un rev\^etement
pro\'etale perfecto\"{\i}de de $X$, galoisien de groupe de Galois $K$. Si $X$ est une courbe Stein, on suppose de plus que 
l'on peut \'ecrire $X$ comme une r\'eunion croissante stricte
d'affino\"ides $X_n$ dont les images inverses dans $X$ sont des affino\"ides perfecto\"{\i}des.
Alors on dispose d'isomorphismes canoniques
$$H^0(K, \O(\widehat X))=\O(X), \quad H^1(K, \O(\widehat X))\simeq \Omega^1(X).$$
\end{lemm}
\begin{proof} 
Commen\c cons par traiter le cas o\`u $X$ est un affino\"{\i}de.
Par hypoth\`ese 
$\widehat X\to X$ est un recouvrement dans $X_{\proet}$, de groupe de Galois $K$. 
Comme 
$\widehat X\times_{X} \widehat X=K\times \widehat X$, le complexe de \v{C}ech 
du faisceau $\widehat{\mathcal{O}}$ attach\'e \`a ce recouvrement 
est $(\mathcal{C}^0(K^j, \O(\widehat X)))_{j}$ (cela utilise \cite[cor. 6.6]{RAV}). 
Puisque 
tous les objets que l'on consid\`ere sont quasi-compacts et s\'epar\'es et puisque le faisceau 
$\widehat{\mathcal{O}}$ n'a pas de cohomologie en degr\'e $>0$ sur les affino\"{\i}des 
perfecto\"{\i}des, on en d\'eduit un isomorphisme 
$H^j(K, \O(\widehat X))=H^j_{\proet}(X, \widehat{\mathcal{O}})$. 
Le lemme \ref{five term} permet de conclure.

Supposons maintenant que $X$ est Stein et que l'on peut \'ecrire $X$ comme une r\'eunion croissante stricte
d'affino\"ides $X_n$, dont l'image inverse $\widehat X_n$ de $X_n$ dans
$\widehat X$ est affino\"{\i}de perfecto\"{\i}de. Comme $\widehat X_n\to X_n$ est
galoisien de groupe de Galois $K$,
pour d\'eduire le r\'esultat du cas perfecto\"{\i}de, il suffit donc de v\'erifier
que $H^i(K,\varprojlim_n\O(\widehat X_n))=\varprojlim_nH^i(K,\O(\widehat X_n))$.
C'est imm\'ediat si $i=0$ et, si $i=1$, cela suit de ce que $X$ est suppos\'e Stein
et donc $\widehat X$ est aussi Stein
(de type g\'en\'eralis\'e, cf. \cite[2.6, lemma. 2.6.3]{KL2}),
et donc 
$${\rm R}^1\varprojlim \O(X_n)=0 \quad{\rm et}\quad  {\rm R}^1\varprojlim \O(\widehat X_n)=0,$$
et de la suite exacte (o\`u $B_n=\O(\widehat X_n)$)
$${\rm R}^1\varprojlim H^0(K,B_n)\to H^1(K,\varprojlim_nB_n)\to
\varprojlim_nH^1(K,B_n)\to H^0(K,{\rm R}^1\varprojlim B_n).\qedhere$$
\end{proof}

\begin{prop} \label{coh infinie} 
On dispose d'isomorphismes canoniques 
\begin{align*} H^0(G_j, \mathcal{O}(\widehat{\rm LT}^\varpi_{\infty}))=\mathcal{O}({\rm LT}^\varpi_j)
\quad&{\rm et}\quad
H^1(G_j, \mathcal{O}(\widehat{\rm LT}^\varpi_{\infty}))=\Omega^1({\rm LT}^\varpi_j), \\
H^0(\check{G}_j,  \mathcal{O}(\widehat{\cal M}^\varpi_{\infty}))=\mathcal{O}({\cal M}^\varpi_j)
\quad&{\rm et}\quad
H^1(\check{G}_j, \mathcal{O}(\widehat{\cal M}^\varpi_{\infty}))=\Omega^1({\cal M}^\varpi_j). 
\end{align*}
      \end{prop}
\begin{proof} 
C'est une cons\'equence directe du lemme~\ref{scholzestyle}
utilis\'e pour $(X,\widehat X,K)=({\rm LT}^\varpi_j, \widehat{\rm LT}^\varpi_{\infty},G_j)$
ou $({\cal M}^\varpi_j,\widehat{\cal M}^\varpi_{\infty}, \check{G}_j)$, 
les hypoth\`eses \'etant satisfaites gr\^ace \`a la rem.~\ref{wein2}.
\end{proof}

\subsubsection{Cohomologie \`a support compact}\label{GRAB5}
Si $X$ est un des espaces ${\rm LT}^\varpi_j$, ${\cal M}^\varpi_{j}$ on note 
$$\mathcal{O}(\partial X)=\varinjlim_{U} \mathcal{O}(X\moins U),$$
la limite inductive \'etant prise sur les affino\"{\i}des $U$ d'un recouvrement Stein de 
$X$. 
On d\'efinit par le m\^eme proc\'ed\'e l'espace $\Omega^1(\partial X)$.  
On d\'efinit alors l'espace $H^1_c(X, \mathcal{O})$ comme le quotient
$\O(\partial X)/\O(X)$.

De m\^eme, si $\widehat X=\widehat{\rm LT}_\infty^\varpi\cong \widehat{\cal M}_\infty^\varpi$,
on pose $\O(\partial \widehat X)=\varinjlim_{U} \mathcal{O}(\widehat X\moins U),$
la limite inductive \'etant prise sur les domaines rationnels de $\widehat X$
(ce sont des affino\"{\i}des perfecto\"{\i}des).  Comme $\widehat X$ admet des recouvrements
croissants par les images inverses d'affino\"{\i}des de ${\rm TL}_j^\varpi$
(resp.~$\mv_j$), on peut ne consid\'erer que des $U$ provenant de ${\rm TL}_j^\varpi$
(resp.~$\mv_j$) pour calculer $\O(\partial \widehat X)$ et on peut
m\^eme supposer (rem.~\ref{wein2}) que $\widehat X\moins U$ admet des recouvrements
croissants par les images inverses d'affino\"{\i}des de ${\rm TL}_j^\varpi$
(resp.~$\mv_j$).
On d\'efinit
$H^1_c(\widehat X, \mathcal{O})$ par la suite exacte
$$0\to \mathcal{O}(\widehat X)\to \mathcal{O}(\partial \widehat X)\to H^1_c(\widehat X, \mathcal{O})\to 0.$$
\begin{rema}
Les d\'efinitions ci-dessus sont parfaitement ad hoc, mais donnent le m\^eme
r\'esultat que la d\'efinition naturelle ci-dessous ou que
la d\'efinition habituelle~\cite{Chiarellotto,vdpdual} dans le cas de $X$.
Si $Z=X,\widehat X$, on d\'efinit le pro-objet $\partial Z$
comme la limite projective des $Z\moins U$,
o\`u $U$ varie dans les domaines rationnels.  Si ${\cal F}$ est
un faisceau sur $Z$, on d\'efinit $\rg(\partial Z,{\cal F})$
comme la limite inductive $\varinjlim \rg(Z\moins U,{\cal F})$
et $\rg_c(Z,{\cal F})$ comme le c\^one $[\rg(Z,{\cal F})\to \rg(\partial Z,{\cal F})]$.
On a alors une suite exacte longue
$$0\to H^0_c(Z,{\cal F})\to H^0(Z,{\cal F})\to H^0(\partial Z,{\cal F})\to 
H^1_c(Z,{\cal F})\to H^1(Z,{\cal F})$$
Dans le cas qui nous int\'eresse, \`a savoir ${\cal F}=\O$, on a $H^0_c(Z,{\cal F})=0$
et $H^1(Z,{\cal F})=0$ (cf.~\cite[2.6, lemma~2.6.3]{KL2} pour $Z=\widehat X$),
ce qui justifie les d\'efinitions ci-dessus.
\end{rema}

\begin{coro}\label{comparer1}
On a des isomorphismes naturels:
$$H^1_c({\rm LT}_j^\varpi,\O)=H^1_c(\widehat{\rm LT}_\infty^\varpi,\O)^{G_j}
\quad{\rm et}\quad
H^1_c(\mv_n,\O)=H^1_c(\widehat{\cal M}_\infty^\varpi,\O)^{\check G_n}.$$
\end{coro}
\begin{proof}
Soit $(X,\widehat X,K)=({\rm LT}_j^\varpi,\widehat{\rm LT}_\infty^\varpi,G_j)$ ou
$(\mv_n,\widehat{\cal M}_\infty^\varpi,\check G_n)$.
Comme $\Omega^1(X)\to\Omega^1(\partial X)$ est injective, on d\'eduit de la prop.~\ref{coh infinie}
(et de son analogue pour $\widehat X$ priv\'e d'un affino\"{\i}de perfecto\"{\i}de 
du type de la rem.~\ref{wein2}
auquel on peut appliquer le lemme~\ref{scholzestyle})
une suite exacte $0\to \O(X)\to \O(\partial X)\to H^1_c(\widehat X,\O)^K\to 0$ qui,
combin\'ee avec la suite exacte
$0\to \O(X)\to \O(\partial X)\to H^1_c(X,\O)\to 0$, permet de conclure. 
\end{proof}

Si $X$ est un des espaces ${\rm LT}^\varpi_j$, ${\cal M}^\varpi_{j}$, la dualit\'e de Serre \cite{Bayer, Chiarellotto, vdpdual} fournit des isomorphismes naturels 
   $$H^i_{\rm dR}(X)\simeq (H^{2-i}_{\rm dR, c} (X))^\dual , \quad H^i_{\rm dR, c}(X)\simeq (H^{2-i}_{\rm dR} (X))^\dual ,
\quad \Omega^1(X)^\dual \simeq H^1_{c}(X, \mathcal{O}_X).$$ 
En dualisant la suite exacte $$0\to \mathcal{O}(X)/H^0_{\rm dR}(X)\to \Omega^1(X)\to H^1_{\rm dR}(X)\to 0,$$
on obtient donc la
suite exacte d'espaces de type compact
$$0 \to H^1_{\rm dR, c}(X)\to \Omega^1(X)^\dual\simeq H^1_c(X, \mathcal{O})\to 
 \mathcal{O}(X)^\dual\to H^0_{\rm dR}(X)^\dual\to 0.$$
Compte-tenu du cor.\,\ref{comparer1}, le th.\,\ref{comparer} est une cons\'equence de
la proposition suivante.
\begin{prop}\label{nulle}
{\rm (i)} On a $[\mathcal{O}({\cal M}^\varpi_n)^\dual ]^{G_j}=0$ et 
$[\mathcal{O}({\rm LT}^\varpi_j)^\dual ]^{\check{G}_n}=0$.
     
{\rm (ii)} On dispose d'isomorphismes canoniques 
$$H^1_{\rm dR, c}({\cal M}^\varpi_n)^{G_j}=H^1_c({\cal M}^\varpi_n, \mathcal{O})^{G_j}
\quad{\rm et}\quad 
H^1_{\rm dR, c}({\rm LT}^\varpi_j)^{\check{G}_n}=H^1_c({\rm LT}^\varpi_j, \mathcal{O})^{\check{G}_n}.$$
\end{prop}
    
\begin{proof} Le (ii) est une cons\'equence directe du (i) et de suite exacte ci-dessus. 
Pour montrer le (ii), on montre d'abord que l'action de $G_j$ (resp. $\check{G}_n$) sur $\mathcal{O}({\cal M}^\varpi_n)^\dual $ (resp. $\mathcal{O}({\rm LT}^\varpi_j)^\dual $) peut se d\'eriver (et l'op\'erateur ${\rm Lie}(G_j)\to {\rm End}(\mathcal{O}({\cal M}^\varpi_n)^\dual )$ ainsi obtenu est $\mathcal{O}_F$-lin\'eaire), en raisonnant comme dans la preuve du th\'eor\`eme 3.2 de \cite{DL}. Il suffit donc de montrer que $[\mathcal{O}({\cal M}^\varpi_n)^\dual ]^{{\rm Lie}(G)}=0$ et 
$[\mathcal{O}({\rm LT}^\varpi_j)^\dual ]^{{\rm Lie}(\check{G})}=0$.
    
Montrons-le d'abord du c\^ot\'e Drinfeld. Soit $z$ la "variable" sur $\mathbf{P}^1$ et soit 
$a^+$ (resp.~$u^+$) l'action infinit\'esimale de 
$\matrice{\mathcal{O}_F^\dual}{0}{0}{1}$ 
(resp.~de $\matrice{1}{\mathcal{O}_F}{0}{1}$). 
Un calcul direct en niveau $0$ combin\'e avec le fait que ${\cal M}^\varpi_n\to {\cal M}^\varpi_0$ est \'etale montre que:
$$u^+(f)=-f'\quad \text{et}\quad a^+(f)=-zf'=u^+(zf)+f$$ pour tout 
$f\in \mathcal{O}({\cal M}^\varpi_n)$. Si $\lambda \in \mathcal{O}({\cal M}^\varpi_n)^\dual $ est tu\'ee par ${\rm Lie}(G)$, on obtient 
$$\lambda (f)=\lambda (a^+(f)-u^+(zf))=-(a^+\lambda )(f)+(u^+\lambda )(z f)=0$$
pour tout $f\in \mathcal{O}({\cal M}^\varpi_n)$, donc $\lambda =0$.
    
L'argument est similaire, mais un peu plus subtil du c\^ot\'e Lubin-Tate. Soit 
$\pi_{\rm GH}={\rm LT}^\varpi_0\to \mathbf{P}^1$ l'application des p\'eriodes de Gross-Hopkins:
on a $\pi_{\rm GH}(z)=A(z)/B(z)$ o\`u $A,B$ sont des fonctions analytiques sur ${\rm LT}^\varpi_0$
(qui, rappelons-le, est juste une r\'eunion disjointe de disques unit\'es ouverts),
sans z\'eros communs. 
Notons $$g=AB'-A'B, \, f_1=\frac{A^2}{g},\, f_2=\frac{B^2}{g}, \, f_3=\frac{AB}{g}.$$
Gross et Hopkins (voir la toute derni\`ere page de \cite{GH}) montrent que 
$g$ est inversible dans $\mathcal{O}({\rm LT}^\varpi_0)$ et que les op\'erateurs 
$\partial_i=f_i\frac{d}{dz}$ donnent l'action infinit\'esimale de 
$\check{G}/Z(\check{G})$ sur\footnote{Cela est vrai \`a priori pour 
$j=0$, mais reste vrai pour tout $j$ car ${\rm LT}^\varpi_j\to {\rm LT}^\varpi_0$ est \'etale.} $\mathcal{O}({\rm LT}^\varpi_j)$. 
Supposons que $\lambda \in \mathcal{O}({\rm LT}^\varpi_j)^\dual $ est tu\'ee par 
${\rm Lie}(\check{G})$, donc $\lambda (f_i\cdot f')=0$ pour 
$1\leq i\leq 3$ et $f\in \mathcal{O}({\rm LT}^\varpi_j)$. En particulier 
$\lambda (f_1\cdot (f_2f)')=0$ et $\lambda (f_2\cdot (f_1f)')=0$ pour tout $f$, d'o\`u 
$\lambda ((f_1f_2'-f_2f_1')f)=0$. Un calcul direct montre que $f_1f_2'-f_2f_1'=2f_3$. Donc 
$\lambda (f_3f)=0$ pour tout $f$. On obtient de la m\^eme mani\`ere 
$\lambda (f_1f)=\lambda (f_2f)=0$ pour tout $f$, donc $\lambda $ 
est nulle sur les fonctions dans l'id\'eal engendr\'e par
$f_1,f_2,f_3$. Cet id\'eal est dense car les $f_i$ n'ont pas de z\'ero commun. 
On a donc $\lambda =0$, ce qui permet de conclure.
\end{proof}

\section{Applications de la compatibilit\'e local-global}
\label{GAB4}
Dans ce chapitre, nous utilisons des m\'ethodes globales pour prouver les th.\,\ref{gabriel} et~\ref{diagfond}.
Nous allons commencer (\S\,\ref{GAB42}, prop.\,\ref{weak}) par identifier la multiplicit\'e de ${\rm JL}(M)^\dual\otimes {\rm LL}(M)^\dual$
dans $H^1_{\rm HK}({\cal M}_\infty)$ en tant que repr\'esentation de ${\rm WD}_F$
avant de d\'eterminer (\S\,\ref{Theo1.1}) celle qui nous int\'eresse pour le th.\,\ref{gabriel}, \`a savoir celle de
${\rm JL}(M)^\dual$ comme repr\'esentation de $G\times {\rm WD}_F$.
Enfin, dans le \S\,\ref{GAB8}, nous prouvons le th.\,\ref{diagfond}.

\subsection{Notations et remarques pr\'eliminaires}\label{notasup}
Il n'y a pas que les $M$ supercuspidaux qui contribuent \`a la cohomologie
de ${\cal M}_\infty$: tout $M$ ind\'ecomposable de rang $2$ (i.e.~pas somme directe
de deux objets de rang~$1$) contribue mais, si $M$ n'est pas
supercuspidal, sa contribution est \'eparpill\'ee fa\c{c}on puzzle entre le $H^0$ et le $H^1$
(pour rassembler les morceaux, il faudrait passer \`a la cat\'egorie d\'eriv\'ee~\cite{Dat}).
Pour tenir compte de ce ph\'enom\`ene,
on d\'efinit des repr\'esentations ${\rm WD}^i(M)$, ${\rm LL}^i(M)$ et ${\rm JL}^i(M)$, pour $i=0,1$.

Notons qu'un $(\varphi,N,\G_F)$-module $M$ de rang $2$ qui est ind\'ecomposable est soit
supercuspidal soit un tordu ${\rm Sp}\otimes\eta$, o\`u $\eta$ est un caract\`ere lisse de
$F^\dual$, du module ${\rm Sp}$ d\'efini par 
$${\rm Sp}=\Q_p^{\rm nr}e_1\oplus \Q_p^{\rm nr}e_2,\quad
\varphi(e_1)=e_1,\ \varphi(e_2)=pe_2,\hskip.2cm Ne_1=0,\ Ne_2=e_1.$$
Dans ce cas, 
$M$ est dit {\it sp\'ecial}.

$\bullet$ Si $M$ est supercuspidal, on pose 
\begin{align*}
{\rm WD}^0(M)=0\quad{\rm et}\quad&  {\rm WD}^1(M)={\rm WD}(M),\\
{\rm LL}^0(M)=0\quad{\rm et}\quad&  {\rm LL}^1(M)={\rm LL}(M),\\
{\rm JL}^0(M)=0\quad{\rm et}\quad&  {\rm JL}^1(M)={\rm JL}(M).
\end{align*} 

$\bullet$ Si $M={\rm Sp}\otimes\eta$ est sp\'ecial,
on pose 
\begin{align*}
{\rm WD}^0(M)=L(N_{F/\Q_p}\eta)\quad{\rm et}\quad& {\rm WD}^1(M)=L(\eta),\\
{\rm LL}^0(M)=\eta\circ\nu_G \quad{\rm et}\quad&{\rm LL}^1(M)={\rm St}^{\rm lisse}\otimes(\eta\circ\nu_G),\\
{\rm JL}^0(M)={\rm JL}^1&(M)= \eta\circ\nu_{\check G}.
\end{align*}

Nous allons travailler avec le quotient $\mv_\infty$ 
de ${\cal M}_\infty$ par le sous-groupe $\varpi^\Z$
du centre de $\check G$, ce qui fournit un objet plus maniable (en particulier,
$\mv_n$ est d\'efini sur $F$ au lieu de $\breve F$)
et est inoffensif pour les probl\`emes
qui nous int\'eressent: cela restreint l'ensemble de $M$ contribuant \`a la cohomologie, mais
les autres se r\'ecup\`erent en tordant par un caract\`ere non ramifi\'e.
Plus pr\'ecis\'ement,
si $\alpha\in L^\dual$, notons ${\rm nr}_\alpha$ le caract\`ere de $F^\dual$, trivial sur
$\O_F^\dual$ et envoyant $\varpi$ sur $\alpha$ et,
si $H=G,\check G,W_F$, notons ${\rm nr}_{\alpha,H}$ le caract\`ere ${\rm nr}_\alpha\circ\nu_H$ de $H$.
Si $M$ est un $L$-$(\varphi,N,\G_F)$-module de rang~$2$, ind\'ecomposable, alors $\varpi$
(vu comme \'el\'ement du centre de $\check G$) agit par un scalaire $\lambda\in L$ sur ${\rm JL}^i(M)$.
Si $\alpha^{-2}=\lambda$, alors $\varpi$ agit trivialement
sur ${\rm JL}^i(M)\otimes{\rm nr}_{\alpha,\check G}$ et, si $H^\bullet$ est une cohomologie raisonnable, on a
$${\rm Hom}_{\check G}\big({\rm JL}^i(M), H^i({\cal M}_\infty)\big)=
{\rm Hom}_{\check G}\big({\rm JL}^i(M)\otimes{\rm nr}_{\alpha,\check G}, H^i({\cal M}_\infty^\varpi)\big)
\otimes\big({\rm nr}_{\alpha,G}^{-1}
\otimes{\rm nr}_{\alpha,W_F}^{-1}\big),$$
en tant que repr\'esentations de $G\times W_F$.  Pour prouver les th.\,\ref{gabriel} et~\ref{diagfond},
on peut donc supposer que $\varpi$ agit trivialement sur ${\rm JL}^i(M)$ et remplacer ${\cal M}_\infty$ par $\mv_\infty$.

On dit $M$ est {\it $\varpi$-compatible} 
si $\varpi$ (vu comme \'el\'ement du centre de $\check G$)
agit trivialement sur ${\rm JL}^i(M)$.  
Cela implique que $M$ est de pente $\frac{1}{2}$ et,
si $F=\Q_p$, 
que $\varpi$ (vu comme \'el\'ement du centre de $G$) agit trivialement
sur $\Pi(V_{M,{\cal L}})$ pour toute droite ${\cal L}$ de $M_{\rm dR}$.
On note $\fn^{\varpi}$ l'ensemble des $(\varphi,N,\G_F)$-modules de rang~$2$, ind\'ecomposables, et
$\varpi$-compatibles. 

La th\'eorie de Lubin-Tate non ab\'elienne \cite{Carayol, Carayol2, Faltings} 
(voir aussi \cite{Weinstein} pour un \'enonc\'e compact) fournit le r\'esultat suivant
(dans lequel les coefficients de ${\rm JL}^i(M)$, ${\rm WD}^i(M)$ et ${\rm LL}^i(M)$ ont \'et\'e
\'etendus \`a $\Qbar_\ell$) dont nous allons prouver l'analogue pour la cohomologie de de Rham
\`a support compact (th.\,\ref{DAT11}).
\begin{prop}\label{DAT10}
Si $\ell\neq p$ et si
$i=0,1$, alors\footnote{${\rm LL}^i(M)^\vee$ est la contragr\'ediente de ${\rm LL}^i(M)$,
i.e. l'ensemble des vecteurs lisses de ${\rm LL}^i(M)^\dual$.}
$$\Qbar_\ell\otimes_{\Q_{\ell}}H^i_{\eet,c}(\mv_\infty,\Q_\ell)=\bigoplus_{M\in\fn^\varpi}
{\rm JL}^i(M)\otimes{\rm WD}^i(M)\otimes{\rm LL}^i(M)^\vee,$$
en tant que repr\'esentations de $\check G\times W_F\times G$.
\end{prop}

\subsection{Multiplicit\'e de ${\rm JL}(M)^\dual\otimes {\rm LL}(M)^\dual$}\label{GAB42}
     Soit $M\in\fn^\varpi$, supercuspidal.
            Soit 
      $M_{\rm dR}=(M\otimes_{\Q_p^{\rm nr}} \Qbar_p)^{\mathcal{G}_F}$, 
un $L\otimes_{\Q_p} F$-module libre de rang $2$, et soit $\breve{M}=M\otimes_{\Q_p^{\rm nr}} \breve{\Q}_p$.

          Pour all\'eger les notations, nous allons \'ecrire 
     $$\quad X[M]:={\rm Hom}_{L[\check G]}({\rm JL}(M), X\otimes_{\Q_p} L)$$
     si $X$ est un $\check{G}$-module.
Notons que le foncteur $X\mapsto X[M]$ est exact sur la cat\'egorie
des $\check G$-modules lisses avec action triviale de $\varpi$ (car $\check G/\varpi^\Z$ est profini).

Nous fixons $n$ assez grand pour que 
      ${\rm JL}(M)$ soit triviale sur $1+\varpi_D^n \mathcal{O}_D$, ce qui est possible car 
      ${\rm JL}(M)$ est lisse et de dimension finie sur $L$. 
         
    \subsubsection{Courbes de Shimura et compatibilit\'e local-global}\label{LGC}

On choisit:

\quad $\bullet$ un corps totalement r\'eel $E$ ayant une place $\pp$ au-dessus de $p$ telle que
     $E_{\pp}=F$.

\quad $\bullet$ une place infinie $\infty_0$ de $E$,

\quad $\bullet$ une alg\`ebre de quaternions $\check B$ d\'eploy\'ee en $\infty_0$, 
compacte (modulo le centre) aux autres places infinies de $E$ et ramifi\'ee en $\pp$.
  
On note:

\quad $\bullet$ $\mathbf{A}$ les ad\`eles de $E$ et $\mathbf{A}_f$ (resp. $\mathbf{A}_f^{\pp}$) 
les ad\`eles finies, (resp.~finies hors de $\pp$).

\quad $\bullet$ $B$ l'alg\`ebre de quaternions ayant m\^emes invariants que $\check B$
en dehors de $\infty_0$ et~$\pp$, compacte modulo le centre en $\infty_0$ (et donc en toutes les places infinies de $E$) et
d\'eploy\'ee en $\pp$.

\quad $\bullet$ ${\mathbb G}$ et $\check{\mathbb G}$ les groupes alg\'ebriques
associ\'es \`a $B^\dual$ et $\check B^\dual$: si 
$R$ est une $E$-alg\`ebre, alors ${\mathbb G}(R)=(B\otimes_{E} R)^\dual $ et 
$\check{\mathbb G}(R)=(\check B\otimes_{E} R)^\dual $.

\quad $\bullet$ $\Gamma$ et $\check\Gamma$ les groupes ${\mathbb G}(E)=B^\dual$ et $\check{\mathbb G}(E)=\check B^\dual$.

On fixe des isomorphismes 
$${\mathbb G}(E_\pp)\simeq G,\quad \check{\mathbb G}(E_\pp)\simeq \check G,
\quad \check{\mathbb G}(\mathbf{A}_f^{\pp})\cong {\mathbb G}(\mathbf{A}_f^{\pp}).$$

Soit ${\rm SD}_2(\check{\mathbb G})$ l'ensemble des repr\'esentations automorphes 
$\pi=\pi_{\infty}\otimes \pi_f$ de $\check{\mathbb G}(\mathbf{A})$ 
telles que ${\rm Hom}_{\check B_{\infty}^\dual }(\sigma_2, \pi_{\infty})\ne 0$, o\`u
$\sigma_2$ est la repr\'esentation de $\check{\mathbb G}(\mathbf{R}\otimes_\Q E)=\prod_{v\mid\infty}
\check{\mathbb G}(E_v)$ triviale aux places $v\ne \infty_0$ et 
s\'erie discr\`ete holomorphe de poids $2$ et de caract\`ere central trivial en la place $\infty_0$.
Si $n\geq 1$, soit ${\rm SD}_{2,n}\subset {\rm SD}_2$ l'ensemble des $\pi$ telles
que $\check G_n$ agisse trivialement sur $\pi_\pp$.

On fixe:

\quad $\bullet$ $M\in \Phi N^{\varpi}$ supercuspidal,

\quad $\bullet$ 
$\check\Pi\in {\rm SD}_{2,n}$ d\'efinie
\footnote{Pour pouvoir globaliser~\cite{Clozel} il faut ajuster le caract\`ere central, et donc on peut
\^etre amen\'e \`a tordre tout par un caract\`ere (et donc \`a changer $\varpi$); cela peut
aussi demander de remplacer $L$ par une extension finie.} 
 sur $L$ et telle que $\check\Pi_{\pp}={\rm JL}(M)$.

On note:

\quad $\bullet$ $\Pi$ -- la repr\'esentation automorphe de ${\mathbb G}(\A)$
qui correspond \`a $\check\Pi$ par la correspondance de Jacquet-Langlands globale;
on a donc $\Pi_f^\pp=\check\Pi_f^\pp$ et $\Pi_\pp={\rm LL}(M)$.

Si $n\geq 1$ et si $U$ est un sous-groupe ouvert compact suffisamment petit de $\check{\mathbb G}(\mathbf{A}_{f}^\pp)$,
notons ${\rm Sh}_n(U)_E$ la courbe de Shimura, d\'efinie sur $E$, 
dont les $\mathbf{C}$-points\footnote{O\`u $\mathbf{C}\moins \mathbf{R}$ est le
quotient habituel de ${\bf GL}_2(E_{\infty_0})\simeq {\rm GL}_2(\mathbf{R})$.}
sont donn\'es par 
$${\rm Sh}_n(U)_E(\mathbf{C})=\check\Gamma\backslash [(\mathbf{C}\moins \mathbf{R})\times 
(\check{\mathbb G}(\mathbf{A}_f)/(U\times\check G_n))].$$
Si $K$ est un corps contenant $E$, on note
${\rm Sh}_n(U)_K$ la courbe sur $K$ obtenue
par extension des scalaires et simplement 
${\rm Sh}_n(U)$ si $K=C$.

Si $H$ est une des cohomologies $H^1_{\rm dR}(-)$, $H^1_{\rm HK}(-)$ ou $H^1_{\eet}(-,L)$, 
si $n\geq 1$, et si $K$ est une extension de $E$ (contenant $F$ si $H=H^1_{\rm HK}$), soit\footnote{Si $K=C$,
on le fait dispara\^{\i}tre de la notation.}
      $$H({\rm Sh}_{n,K})=\varinjlim_{U} H({\rm Sh}_{n}(U)_K),$$
      la limite inductive \'etant prise sur les sous-groupes ouverts 
compacts $U$ de $\check{\mathbb G}(\mathbf{A}_f^\pp)$.
Le r\'esultat suivant est une cons\'equence des th\'eor\`emes de compatibilit\'e local-global de Carayol \cite{Carayol} et Saito \cite{Saito}.    
      
\begin{prop}\label{CSpst}
On a des isomorphismes
\begin{align*}
&{\rm Hom}_{\check{\mathbb G}(\mathbf{A}_{f}^{\pp})}( \Pi_f^{\pp}, L\otimes_{\Q_p}H^1_{\rm HK}({\rm Sh}_{n}))
= \ {\rm JL}(M)\otimes_L M\\
&{\rm Hom}_{\check{\mathbb G}(\mathbf{A}_{f}^{\pp})}( \Pi_f^{\pp}, L\otimes_{\Q_p}H^1_{\rm dR}({\rm Sh}_{n,K}))
= \ {\rm JL}(M)\otimes_L (K\otimes_F M_{\rm dR}),\quad{\text{si $F\subset K$}}.
\end{align*}
\end{prop}
\begin{proof} 
L'espace $H^1_{\eet}({\rm Sh}_{n,\Qbar},L)$ 
est muni d'actions qui commutent de $\G_{E}:={\rm Gal}(\bar{\mathbf{Q}}/E)$ et 
$\check{\mathbb G}(\mathbf{A}_f)$.
Les th\'eor\`emes de Carayol et Saito fournissent un isomorphisme 
de $\G_{E}\times \check{\mathbb G}(\mathbf{A}_f)$-modules\footnote{Rappelons que l'on se permet de voir $\pi_f$ comme une repr\'esentation sur 
  $\overline{\Q}_p$ en utilisant l'isomorphisme fix\'e entre $\overline{\Q}_p$ et $\mathbf{C}$.} 
$$H^1_{\eet}({\rm Sh}_{n,\Qbar},L)\otimes_{L} \Qbar_p 
\simeq \bigoplus_{\pi \in {\rm SD}_{2,n}} (\pi_f\otimes_{\Qbar_p} \rho_{\pi}(-1))$$
o\`u la repr\'esentation $\rho_{\pi}: \G_{E}\to {\rm GL}_2(\Qbar_p)$ 
est telle que, pour toute place finie 
$v$ de $E$, l'on ait un isomorphisme\footnote{
$\rho_{\pi,v}$ est la restriction de $\rho_\pi$ \`a $\G_{E_v}:={\rm Gal}(\Qbar_p/E_v)$
vu comme sous-groupe de $\G_E$ via le choix d'un plongement de $\overline E$ dans $\overline E_v$,
et ${\rm JL} ({\rm WD}(\rho_{\pi,v}))$ est la repr\'esentation
de $\check{\mathbb G}(E_v)$ qui lui est associ\'ee par les recettes habituelles:
${\rm WD}(\rho_{\pi,v})$ est la repr\'esentation de Weil-Deligne obtenue via 
le th\'eor\'eme de monodromie $\ell$ ou $p$-adique et la recette de Fontaine si $\ell=p$,
et ${\rm JL} ({\rm WD}(\rho_{\pi,v}))$ est la repr\'esentation attach\'ee \`a
${\rm WD}(\rho_{\pi,v})$ par la correspondance de Langlands locale, 
combin\'ee \'eventuellement avec celle de Jacquet-Langlands si
      $\check B$ est ramifi\'ee en $v$.} 
$$ \pi_v\simeq {\rm JL} ({\rm WD}(\rho_{\pi,v})).$$
Il suffit alors de restreindre cette repr\'esentation \`a $\G_F=\G_{E_\pp}$,
d'appliquer les foncteurs $\dpst$ et $\ddr$ \`a l'isomorphisme ci-dessus,
d'utiliser les th\'eor\`emes de comparaison $p$-adiques
et
l'isomorphisme $D_{\rm pst}(\rho_{\check\Pi, \pp}(-1))=M$
(qui d\'ecoule de l'hypoth\`ese $\Pi_{\pp}\simeq {\rm JL}(M)$)
et d'appliquer ensuite le foncteur ${\rm Hom}_{\check{\mathbb G}(\mathbf{A}_{f}^{\pp})}( \Pi_f^{\pp}, -)$
en utilisant le fait 
que $\Pi_f^\pp$ d\'etermine $\Pi_\pp$ (th\'eor\`eme de multiplicit\'e~$1$ fort).
\end{proof}

En appliquant ${\rm Hom}_{\check G}({\rm JL}(M),L\otimes -)$, on d\'eduit de la prop.~\ref{CSpst} 
un diagramme commutatif 
\begin{equation}\label{diagr1}
       \xymatrix@R=.6cm{{\rm Hom}_{\check{\mathbb G}(\mathbf{A}_f^{\pp})} 
(\Pi_f^{\pp}, H^1_{\rm HK}({\rm Sh}_n)[M])\ar[r]^-{{\simeq}}\ar[d]^{\iota_{\rm HK}}&
\breve{M} \ar[d]\\
{\rm Hom}_{\check{\mathbb G}(\mathbf{A}_f^{\pp})} 
(\Pi_f^{\pp}, H^1_{\rm dR}({\rm Sh}_n)[M])\ar[r]^-{{\simeq}} & M_{\rm dR}\otimes_F C&} 
\end{equation}

    \subsubsection{Uniformisation $p$-adique} \label{UNIF}
        
     Gardons les notations introduites dans le pragraphe pr\'ec\'edent. 
    Nous allons donner, en utilisant l'uniformisation
    $p$-adique des courbes de Shimura, une autre description des espaces 
${\rm Hom}_{\check{\mathbb G}(\mathbf{A}_f^{\pp})} (\Pi_f^{\pp}, H^1_{\rm HK}({\rm Sh}_n)[M])$ et 
    ${\rm Hom}_{\check{\mathbb G}(\mathbf{A}_f^{\pp})} (\Pi_f^{\pp}, H^1_{\rm dR}({\rm Sh}_n)[M])$ 
apparus dans le diagramme $(\ref{diagr1})$.

Si $U$ est un sous-groupe ouvert de $\check{\mathbb G}(\mathbf{A}_{f}^{\pp})$, on peut aussi
voir $U$ comme un sous-groupe ouvert de ${\mathbb G}(\mathbf{A}_{f}^{\pp})$.
On note $S^\pp(U)$ le quotient 
$$S^\pp(U)={\mathbb G}(\mathbf{A}_{f}^{\pp})/U.$$
C'est un ensemble discret muni d'une action de $\Gamma$.
Le th\'eor\`eme d'uniformisation suivant est d\^u \`a \v{C}erednik-Drinfeld et Boutot-Zink \cite{BZ} (dans ce contexte).

\begin{prop}\label{CD}
 Il existe une famille d'isomorphismes d'espaces rigides analytiques 
   $${\rm Sh}_n(U)^{\rm an}
\simeq \Gamma\backslash [ {\cal M}_n\times S^\pp(U)],$$
   compatibles avec la variation de $U$ et $n$.
   \end{prop}
   
          \begin{rema}\label{Shimura4}
{\rm (i)}
$S^\pp(U)$ ne comporte
qu'un nombre fini d'orbites $\Gamma x_1,\dots,\Gamma x_r$ sous l'action de $\Gamma$.
Si
on note 
      $\Gamma_i$ le stabilisateur de $x_i$ dans $\Gamma$, 
alors les $\Gamma_i$ sont des sous-groupes discrets et co-compacts de ${\mathbb G}(F)=G$, et 
   $${\rm Sh}_n(U)^{\rm an}=
\Gamma\backslash [ {\cal M}_n\times S^\pp(U)]=\coprod_{i=1}^r \Gamma_i\backslash {\cal M}_n.$$

{\rm (ii)}
Si $H^\bullet$ est une cohomologie raisonnable, on a\footnote{${\cal C}(S^\pp(U),X)$
d\'esigne l'espace des fonctions continues de $S^\pp(U)$ dans $X$, mais comme $S^\pp(U)$ est discret,
c'est l'espace de toutes les fonctions de $S^\pp(U)$ dans $X$.}
\begin{align*}
&H^q( {\cal M}_n\times S^\pp(U))={\cal C}(S^\pp(U), H^q( {\cal M}_n)),\\
&H^p(\Gamma,H^q( {\cal M}_n\times S^\pp(U)))=\oplus_{i=1}^rH^p(\Gamma_i,H^q( {\cal M}_n)),
\end{align*}
et une suite spectrale
$$H^p(\Gamma,H^q( {\cal M}_n\times S^\pp(U)))\Longrightarrow 
H^{p+q}(\Gamma\backslash( {\cal M}_n\times S^\pp(U)))=
H^{p+q}({\rm Sh}_n(U))$$
somme directe des suites analogues pour les rev\^etements ${\cal M}_n\to \Gamma_i\backslash{\cal M}_n$. 
 En bas degr\'es, cette suite spectrale
fournit la suite exacte
$$H^1(\Gamma,H^0( {\cal M}_n\times S^\pp(U)))\to H^1({\rm Sh}_n(U))\to
H^0(\Gamma,H^1( {\cal M}_n\times S^\pp(U)))\to H^2(\Gamma,H^0( {\cal M}_n\times S^\pp(U))).$$

{\rm (iii)}
Si $H^0(X)=\Z[\pi_0(X)]\otimes H^0({\rm point})$ pour $X$ ayant un nombre fini de composantes connexes (comme pour $H^\bullet_{\rm dR}$ ou $H^\bullet_{\rm HK}$),
alors $\check G$ op\`ere sur $H^0( {\cal M}_n\times S^\pp(U))$ par la norme r\'eduite, et donc
$H^0( {\cal M}_n\times S^\pp(U))[M]=0$.  La suite ci-dessus fournit donc un isomorphisme:
$$H^1({\rm Sh}_n(U) )[M]\simeq
H^0(\Gamma,H^1( {\cal M}_n\times S^\pp(U))[M]).$$
       \end{rema}

\subsubsection{Analyse fonctionnelle et r\'eciprocit\'e de Frobenius}\label{AFRF}

Soit ${\cal D}(G)$ (resp.~${\cal D}_{\rm alg}(G)$)
l'alg\`ebre des distributions (resp.~distributions alg\'ebriques) \`a support
compact sur $G$: c'est le dual de l'espace ${\rm LA}(G,L)$ (resp.~${\rm LC}(G,L)$)
des fonctions
localement analytiques (resp.~localement constantes) sur $G$.
Alors ${\cal D}_{\rm alg}(G)$ est le quotient de ${\cal D}(G)$
par l'id\'eal engendr\'e par l'alg\`ebre de Lie ${\goth g}$ de $G$;
un ${\cal D}_{\rm alg}(G)$ est donc la m\^eme chose qu'un
${\cal D}(G)$-module tu\'e par ${\goth g}$.
Les masses de Dirac sont denses dans ${\cal D}(G)$ et~${\cal D}_{\rm alg}(G)$.

Soit ${\cal F}$ un ${\cal D}(G)$-fr\'echet.
Si $U$ est un sous-groupe ouvert de ${\mathbb G}(\A_f^\pp)$,
on munit
l'espace ${\cal C}(S^\pp(U),{\cal F})$ 
d'une action de $\Gamma$
par:
$\gamma\cdot\phi(x)=\gamma(\phi(\gamma^{-1}x))$.

Comme $S^\pp(U)$ est discret, un \'el\'ement
de ${\rm LA}(G\times S^\pp(U))^\dual$ (resp.~${\rm LC}(G\times S^\pp(U))^\dual$
est une somme finie $\sum_i \mu_i\otimes\delta_{x_i}$, o\`u les
$\mu_i$ sont des \'el\'ements de
${\cal D}(G)$ (resp.~${\cal D}_{\rm alg}(G)$), $x_i\in S^\pp(U)$ et $\delta_{x_i}$
est la masse de Dirac en $x_i$.
On munit
${\rm Hom}({\rm LA}(G\times S^\pp(U))^\dual,{\cal F})$
d'actions de $\Gamma$ et $G$ commutant entre elles, en posant
$$(h*\lambda)(\delta_g\otimes\delta_x)=h(\lambda(\delta_{gh}\otimes\delta_x)),
\quad (\gamma\cdot\lambda)((\delta_g\otimes\delta_x)=
\lambda(\delta_{\gamma^{-1}g}\otimes\delta_{\gamma^{-1}x}).$$

On note $S(U)$ la vari\'et\'e $p$-adique (molle) 
$$S(U)=\Gamma\backslash(G\times S^\pp(U))=\Gamma\backslash{\mathbb G}(\A_f)/U.$$
C'est une vari\'et\'e compacte: $S(U)=\coprod_i\Gamma_i\backslash G$ et les
$\Gamma_i$ sont cocompacts dans $G$.

\begin{lemm}\label{gab13}
Si ${\cal F}$ est un ${\cal D}_{\rm alg}(G)$-fr\'echet, on a un isomorphisme\footnote{Tous
les ${\rm Hom}$ sont des ${\rm Hom}$ d'espaces vectoriels topologiques, i.e. sont
constitu\'es d'applications lin\'eaires continues.}
\begin{align*}
H^0(\Gamma,{\cal C}(S^\pp(U),{\cal F}))=&\ 
H^0(G,{\rm Hom}({\rm LC}(S(U))^{\dual},{\cal F})).
\end{align*}
\end{lemm}
\begin{proof}
Un \'el\'ement de ${\rm Hom}({\rm LA}(S(U))^{\dual},{\cal F})$
est la m\^eme chose qu'un \'el\'ement de
${\rm Hom}({\rm LA}(G\times S^\pp(U))^\dual,{\cal F})$ invariant par~$\Gamma$.

Soit $\iota:{\cal D}(G)\to{\cal D}(G)$ d\'efinie par $\iota(\delta_g)=\delta_{g^{-1}}$.
Si $\phi\in {\cal C}(S^\pp(U),{\cal F})^\Gamma$,
posons $\lambda_\phi(\mu\otimes \delta_x)=\iota(\mu)\cdot\phi(x)$.
Alors $\lambda_\phi(\gamma\mu\otimes \delta_{\gamma a})=\iota(\mu)\gamma^{-1}\phi(\gamma\cdot a)=
\iota(\mu)\cdot\phi(x)=\lambda_\phi(\mu\otimes \delta_x)$, ce qui prouve
que $\lambda_\phi$ se factorise \`a travers 
${\rm Hom}({\rm LA}(S(U))^{\dual},{\cal F})$.
Comme ${\goth g}$ tue ${\cal F}$, elle tue aussi 
${\rm Hom}({\rm LA}(G\times S^\pp(U))^\dual,{\cal F})$,
et donc $\lambda_\phi\in {\rm Hom}({\rm LC}(S(U))^{\dual},{\cal F})$.
Enfin, $(h\cdot\lambda_\pi)(\delta_g\otimes\delta_x)=h(\lambda_\phi(\delta_{gh}\otimes\delta_x))=
h((gh)^{-1}\cdot\phi_\lambda(x))=g^{-1}\cdot\phi_\lambda(x)$, ce qui prouve que $\lambda_\phi$
est invariant par $G$.

R\'eciproquement, si $\lambda\in 
H^0(G,{\rm Hom}({\rm LC}(S(U))^{\dual},{\cal F}))$,
on d\'efinit $\phi_\lambda$ par $\phi_\lambda(x)=\lambda({\delta}_1\otimes\delta_x)$.
Alors 
$$\gamma(\phi_\lambda(\gamma^{-1}x))=\gamma(\lambda(\delta_1\otimes\delta_{\gamma^{-1}x}))=
\lambda(\delta_{\gamma^{-1}}\otimes\delta_{\gamma^{-1}x})=\lambda(\delta_1\otimes\delta_x)=\phi_\lambda(x)$$
(on a utilis\'e l'invariance de $\lambda$ par $G$ puis par $\Gamma$); il s'ensuit que $\phi_\lambda$
est invariante par~$\Gamma$.

Enfin,
il est clair que $\phi_{\lambda_\phi}=\phi$
et que $\lambda_{\phi_\lambda}$ co\"{\i}ncide avec $\lambda$ sur les
masses de Dirac, et donc partout puisque les masses de Dirac
engendrent un sous-espace dense. Il s'ensuit que $\lambda\mapsto\phi_\lambda$ et $\phi\mapsto\lambda_\phi$
sont des isomorphismes inverses l'un de l'autre.
\end{proof}

\begin{prop}\label{weak}
Il existe un isomorphisme de $(\varphi, N, \mathcal{G}_F)$-modules 
$$\breve{M}\simeq {\rm Hom}_G({\rm LL}(M)^\dual, H^1_{\rm HK}({\cal M}^\varpi_n)[M])$$
s'ins\'erant dans un diagramme commutatif de $\G_F$-modules 
$$\xymatrix@R=.6cm{\breve{M}\ar[r]^-{{\sim}}\ar[d]&
{\rm Hom}_G({\rm LL}(M)^\dual , H^1_{\rm HK}({\cal M}^\varpi_n)[M])\ar[d]^{\iota_{\rm HK}}\\
M_{\rm dR}\otimes_{F} C\ar[r]^-{{\sim}}&  {\rm Hom}_G({\rm LL}(M)^\dual , H^1_{\rm dR}({\cal M}^\varpi_n)[M]) &
}$$
   \end{prop}
\begin{proof}
Le lemme~\ref{gab13}, utilis\'e pour ${\cal F}=H^1_{\rm dR}(\mv_n)$ et combin\'e avec la rem.~\ref{Shimura4}, fournit
   un diagramme commutatif 
     $$\xymatrix@R=.6cm{ H^1_{\rm HK}({\rm Sh}_n(U))[M]\ar[r]^-{\sim}\ar[d]^{\iota_{\rm HK}}&
{\rm Hom}_G^{\rm cont}({\rm LC}(S(U))^\dual , H^1_{\rm HK}({\cal M}^\varpi_n)[M])  \ar[d]^{\iota_{\rm HK}}\\
H^1_{\rm dR}({\rm Sh}_n(U))[M]\ar[r]^-{\sim} & 
{\rm Hom}_G^{\rm cont}({\rm LC}(S(U))^\dual , H^1_{\rm dR}({\cal M}^\varpi_n)[M]) &
} $$

En passant \`a la limite sur $U$, on en d\'eduit un diagramme
     $$\xymatrix{ H^1_{\rm HK}({\rm Sh}_n)[M]\ar[r]^-{\sim}\ar[d]^{\iota_{\rm HK}}&
{\rm Hom}_G^{\rm cont}({\rm LC}(\Gamma\backslash{\mathbb G}(\A_f))^\dual , H^1_{\rm HK}({\cal M}^\varpi_n)[M])  \ar[d]^{\iota_{\rm HK}}\\
H^1_{\rm dR}({\rm Sh}_n)[M]\ar[r]^-{\sim} & 
{\rm Hom}_G^{\rm cont}({\rm LC}(\Gamma\backslash{\mathbb G}(\A_f))^\dual , H^1_{\rm dR}({\cal M}^\varpi_n)[M]) &
} $$

Enfin, la d\'ecomposition de ${\rm LC}(\Gamma\backslash{\mathbb G}(\A_f))$
en somme de repr\'esentations automorphes pour ${\mathbb G}(\A)$
et le th\'eor\`eme de multiplicit\'e~$1$ fort
fournissent un diagramme commutatif 
     $$\xymatrix@R=.6cm{{\rm Hom}_{\check{\mathbb G}(\mathbf{A}_f^{\pp})} 
(\Pi_f^{\pp}, H^1_{\rm HK}({\rm Sh}_n)[M])\ar[r]^-{\sim}\ar[d]^{\iota_{\rm HK}}&
{\rm Hom}_G({\rm LL}(M)^\dual , H^1_{\rm HK}({\cal M}^\varpi_n)[M]) \ar[d]^{\iota_{\rm HK}}\\
{\rm Hom}_{\check{\mathbb G}(\mathbf{A}_f^{\pp})} (\Pi_f^{\pp},  
H^1_{\rm dR}({\rm Sh}_n)[M])\ar[r]^-{{\sim}} &  
{\rm Hom}_G({\rm LL}(M)^\dual , H^1_{\rm dR}({\cal M}^\varpi_n)[M]).} $$
Pour conclure,
il n'y a plus qu'\`a comparer ce diagramme avec
le diag.~(\ref{diagr1}).
\end{proof}

\Subsection{Cohomologies de de Rham et de Hyodo-Kato de $\mv_\infty$}\label{Theo1.1}
\subsubsection{Cohomologie \`a support compact}

Soit $H^1_{\rm dR, c}({\cal M}^\varpi_n)$ la cohomologie de de Rham 
\`a support compact de ${\cal M}^\varpi_n$. 
\begin{theo}\label{DAT11}
On a une d\'ecomposition\footnote{Apr\`es extension des scalaires \`a $C$ pour
${\rm JL}^1(M)$, ${\rm WD}^1(M)$ et ${\rm LL}^1(M)^\vee$.}
$$H^1_{\rm dR,c}(\mv_\infty)=\bigoplus_{M\in\fn^\varpi}
{\rm JL}^1(M)\otimes_C{\rm WD}^1(M)\otimes_C{\rm LL}^1(M)^\vee,$$
en tant que repr\'esentation de $\check G\times G$.
\end{theo}
\begin{proof}
On a une d\'ecomposition
$$H^1_{\rm dR, c}({\cal M}^\varpi_n)=
\bigoplus_{M\in \fn_{n}^\varpi} 
{\rm JL}^1(M)\otimes_C {\rm Hom}_{C[\check G]}({\rm JL}^1(M),H^1_{\rm dR, c}({\cal M}^\varpi_n)).$$
Notons que $H^1_{\rm dR, c}({\cal M}^\varpi_n)$ est le 
$C$-dual de $H^1_{\rm dR}({\cal M}^\varpi_n)$ (cf.~\cite[th.\,4.11]{GK1}), 
et donc que $H^1_{\rm dR, c}({\cal M}^\varpi_n)[M]$ est le
$C$-dual de $H^1_{\rm dR}({\cal M}^\varpi_n)[\check M]$, o\`u $\check M=M^\dual[1]$,
et que ce dernier 
a pour quotient $\check M_{\rm dR}\otimes_{L} {\rm LL}(\check M)^\dual
={\rm WD}^1(\check M)\otimes{\rm LL}^1(\check M)^\dual$
 si $M$ est supercuspidal
(prop.~\ref{weak}).
Le calcul de la cohomologie de de Rham
         de $\Omega_{\rm Dr}\times\pi_0(\mv_n)$ montre que cela reste vrai
si $M$ est sp\'ecial.
En prenant une limite inductive sur $n$, on en d\'eduit
une surjection
$$
H^1_{\rm dR, c}(\mv_\infty)\to \bigoplus_{M\in\fn^\varpi} 
{\rm JL}^1(M)\otimes_C {\rm WD}^1(M)\otimes_C {\rm LL}^1(M)^\vee.$$
On cherche \`a prouver que cette surjection de $G$-repr\'esentations 
lisses est en fait un isomorphisme.
D'apr\`es le th.\,\ref{comparer}, le membre de gauche est aussi isomorphe
\`a $H^1_{\rm dR, c}({\rm LT}^\varpi_\infty)$, et donc ses $G_j$-invariants
sont $H^1_{\rm dR, c}({\rm LT}^\varpi_j)$.
Or ${\rm LT}^\varpi_j$ est 
une courbe analytique ouverte au sens
de \cite{Wewers} et donc, si $\ell\neq p$, on a \cite{Wewers}
$$\dim_C H^1_{\rm dR, c}({\rm LT}^\varpi_j)=\dim_{\Q_\ell}
H^1_{\eet, c}({\rm LT}^\varpi_j, \Q_\ell),$$
et les deux membres sont finis.
La th\'eorie de Lubin-Tate non ab\'elienne (cf.~prop.\,\ref{DAT10}) permet donc de montrer
que les deux membres ont les m\^emes invariants sous l'action de $G_j$ (le foncteur des $G_j$-invariants
est exact sur les repr\'esentations lisses car $G_j$ est ouvert compact),
et donc que la surjection ci-dessus est un isomorphisme.
\end{proof}

\begin{rema}
Soit $\check G'\subset\check G$ le noyau de la norme r\'eduite.
Les r\'esultats du chap.\,\ref{DEMI} montrent que 
la cohomologie
de $\Omega_{\rm Dr}\times\pi_0(\mv_\infty)=\check G'\backslash \mv_\infty$
est enti\`erement d\'ecrite
par la contribution des $M$ sp\'eciaux et que, r\'eciproquement, les $M$ sp\'eciaux
ne contribuent qu'\`a la cohomologie de $\Omega_{\rm Dr}\times\pi_0(\mv_\infty)$.
\end{rema}

   \subsubsection{Preuve du th.\,\ref{gabriel}}\label{GAB7}
Soit $M\in\fn^\varpi$, supercuspidal. Le r\'esultat suivant implique
le th.\,\ref{gabriel}.
\begin{theo} \label{theo 1.1} 
Il existe un isomorphisme de $G$-fr\'echets
$${\rm Hom}_{\check{G}}({\rm JL(M)}, L\otimes_{\Q_p} H^1_{\rm HK}({\cal M}^\varpi_{\infty}))\simeq \breve{M}\widehat{\otimes}_{L} {\rm LL}(M)^\dual ,$$
compatible avec les actions de $\varphi$ et $\mathcal{G}_F$ et s'ins\'erant dans un diagramme commutatif de 
$G$-modules 
$$\xymatrix{{\rm Hom}_{\check{G}}({\rm JL(M)}, L\otimes_{\Q_p} H^1_{\rm HK}({\cal M}^\varpi_{\infty}))\ar[r]^-{{\simeq}}\ar[d]^{\iota_{\rm HK}}&
\breve{M}\widehat{\otimes}_{L} {\rm LL}(M)^\dual \ar[d]\\
{\rm Hom}_{\check{G}}\big({\rm JL}(M), L\otimes_{\Q_p} H^1_{\rm dR}({\cal M}^\varpi_\infty)\big)\ar[r]^-{{\simeq}}&C\widehat\otimes_F M_{\rm dR}\widehat{\otimes}_L {\rm LL}(M)^\dual }$$
la fl\`eche \`a gauche \'etant induite par l'isomorphisme de 
Hyodo-Kato et celle \`a droite par l'identification 
$M_{\rm dR}\otimes_{F}  C=M\otimes_{\Q_p^{\rm nr}} C=\breve{M}\otimes_{\breve{\Q}_p} C$.        
\end{theo}
\begin{proof}
L'isomorphisme de la ligne du bas (pour la cohomologie de de Rham)
se d\'eduit, par $C$-dualit\'e, du th.\,\ref{DAT11} en appliquant
le foncteur $Z\mapsto Z[M]$.  Le reste de l'\'enonc\'e
se d\'eduit alors de la prop.\,\ref{weak}.
\end{proof}
        
\subsection{Le diagramme fondamental}\label{GAB8}

Soient $M\in\fn^\varpi$, supercuspidal, et $n\geq 1$ tel que $M\in\fn^\varpi_n$.
Posons
$$M_{\rm dR}=(M\otimes_{\Q_p^{\rm nr}} \Qbar_p)^{\mathcal{G}_F}, \quad X_{\rm st}^+(M)=(B_{\rm st}^+\otimes_{\Q_p^{\rm nr}} M)^{\varphi=p,N=0}.$$
Le r\'esultat suivant implique le th.\,\ref{diagfond} de l'introduction.
\begin{theo}\label{MAIN}
Il existe un diagramme commutatif de $G$-fr\'echets 
$$\xymatrix@R=.6cm@C=.5cm{0 \ar[r]&\ \mathcal{O}({\cal M}^\varpi_n)[M]\ar[r]^-{{\rm exp}}\ar@{=}[d]&
H^1_{\proet}({\cal M}^\varpi_n, L(1))[M]\ar[d]^{{\rm dlog}}\ar[r]&
X_{\rm st}^+(M)\widehat{\otimes}_L{\rm LL}(M)^\dual \ar[d]^{\theta}\ar[r]&0\\
0\ar[r]& \mathcal{O}({\cal M}^\varpi_n)[M] \ar[r]^-d&\Omega^1({\cal M}^\varpi_n)[M]\ar[r]
&(C \otimes_F M_{\rm dR})\widehat{\otimes}_L {\rm LL}(M)^\dual \ar[r]&0
}$$
  \end{theo}
\begin{proof}
On obtient le diagramme voulu avec $\tHK(\mv_n)[M]$ 
au lieu de $X_{\rm st}^+(M)\widehat{\otimes}_L{\rm LL}(M)^\dual $
en utilisant:

$\bullet$  
le diagramme du th.\,\ref{main} pour $X=\mv_n$, auquel on applique $Z\mapsto Z[M]$,

$\bullet$ la nullit\'e de
$\Z[\pi_0(\mv_n)][M]$ ($\check G$ agit \`a travers la norme r\'eduite 
sur~$\pi_0(\mv_n)$),

$\bullet$ 
l'isomorphisme $H^1_{\rm dR}(\mv_n)[M]\cong (C \otimes_F M_{\rm dR})\widehat{\otimes}_L {\rm LL}(M)^\dual $
du th.\,\ref{theo 1.1}.  

\noindent Il suffit donc 
de prouver que
$\tHK(\mv_n)[M]\cong X_{\rm st}^+(M)\widehat{\otimes}_L{\rm LL}(M)^\dual $, ce qui est fait 
au~\no\ref{GAB10}.
\end{proof}

La preuve
s'appuie sur le r\'esultat suivant qui d\'ecrit les fonctions
\`a valeurs dans $\tHK(\mv_n)$ en termes de la cohomologie de Hyodo-Kato de
courbes de Shimura.
        
\begin{prop}\label{cruciale}
Si $U$ est un sous-groupe ouvert compact suffisamment petit de ${\mathbb G}(\A_f^\pp)$, alors
$${\cal C}(S^\pp(U),\tHK(\mv_n))^\Gamma[M]=
(\bst^+\otimes_{\breve\Q_p}H^1_{\rm HK}({\rm Sh}_n(U)))^{N=0,\varphi=p}[M].$$
\end{prop}
\begin{proof}
Soient ${\mathbb U}=(\Bcris^+)^{\varphi=p}$ et ${\rm U}=(\bcris^+)^{\varphi=p}$
et \'ecrivons $H^i(-)$ au lieu de $H^i_{\proet}(-)$.
Le membre de droite est aussi (cf.~prop.\,\ref{Shimura}) \'egal \`a 
$$\tHK({\rm Sh}_n(U))[M]=
{\rm Ker}\big[H^1({\rm Sh}_n(U),{\mathbb U})[M]\to \Omega^1({\rm Sh}_n(U)\big][M].$$
On a ${\cal C}(S^\pp(U),\tHK(\mv))=\tHK(\mv_n\times S^\pp(U))$, et on d\'eduit de
la d\'efinition de $\tHK$ que
$${\cal C}(S^\pp(U),\tHK(\mv_n))^\Gamma={\rm Ker}\big[H^1(\mv_n\times S^\pp(U),{\mathbb U})^\Gamma
\to \Omega^1(\mv_n\times S^\pp(U))^\Gamma\big].$$
Comme $\mv_n$ est Stein, et comme ${\rm Sh}_n(U)=\Gamma\backslash(\mv_n\times S^\pp(U))$, on a
$$\Omega^1(\mv_n\times S^\pp(U))^\Gamma=\Omega^1({\rm Sh}_n(U)).$$
Pour conclure, il suffit donc de v\'erifier que
$$H^1(\mv_n\times S^\pp(U),{\mathbb U})^\Gamma[M]=H^1({\rm Sh}_n(U), {\mathbb U})[M]$$ et, pour cela,
il suffit, compte-tenu de la suite spectrale habituelle, de v\'erifier que
$H^0(\mv_n\times S^\pp(U),{\mathbb U})[M]=0$ (car alors
$H^i(\Gamma, H^0(\mv_n\times S^\pp(U),{\mathbb U}))[M]=0$),
ou encore que $H^0(\mv_n,{\mathbb U})[M]=0$.

Or $H^0(\mv_n,{\mathbb U})=\Z[\pi_0(\mv_n)]\otimes {\rm U}$ (cf.~formule (\ref{H0})).
Le r\'esultat est donc une cons\'equence du fait que $\check G$ agit \`a travers la norme
r\'eduite sur $\pi_0(\mv_n)$, et donc que $\Z[\pi_0(\mv_n)][M]=0$.
\end{proof}

\subsubsection{$G$-modules de Fr\'echet isotypiques}\label{GAB9}
      
Soit $\pi$ une repr\'esentation lisse irr\'eductible de $G$ (dans les applications on aura 
$\pi={\rm LL}(M)$). 
     
\begin{defi}
Un $L$-fr\'echet $\mathcal{F}$ muni d'une action continue de $G$ est dit
{\it $\pi$-isotypique}
 s'il existe un $L$-banach $X$ et un isomorphisme de $L[G]$-modules topologiques (l'action de 
$G$ sur $X$ \'etant triviale)
$\mathcal{F}\simeq \pi^\dual \widehat{\otimes}_L X$.      
\end{defi}

\begin{prop}\label{isot}
  Soit $\mathcal{F}$ un $L$-fr\'echet $\pi$-isotypique. Tout sous-fr\'echet $G$-stable de $\mathcal{F}$ est $\pi$-isotypique. 
   \end{prop}

\begin{proof} 
\'Ecrivons $\mathcal{F}=\pi^\dual \widehat{\otimes}_L X$ pour un $L$-Banach $X$ avec action triviale de $G$. 
 Soit 
  $\mathcal{F'}$ un sous-fr\'echet $G$-stable de $\mathcal{F}$. 
On peut remplacer $X$ par l'adh\'erence de $Z:=\sum_{f\in \mathcal{F'}} f(\pi)$ dans $X$, ce qui permet de supposer que $Z$ est dense dans $X$. On va montrer que $\mathcal{F'}$ est dense dans $\mathcal{F}$, et donc $\mathcal{F}'=\mathcal{F}$
est $\pi$-isotypique. Supposons que ce n'est pas le cas, il existe 
donc $\mathcal{L}\in \mathcal{F}^\dual $ s'annulant sur  
$\mathcal{F'}$. 
Comme $\pi=\varinjlim_{n} V_n$ (r\'eunion croissante, avec 
$V_n$ de dimension finie), on a un isomorphisme d'espaces vectoriels 
$$\mathcal{F}^\dual =(\varprojlim_{n} (V_n^\dual \otimes_L X))^\dual =\varinjlim_{n} (V_n\otimes_L X^\dual )=\pi\otimes_L X^\dual .$$
On peut donc \'ecrire
 $\mathcal{L}=\sum_{i=1}^n v_i\otimes \lambda_i$, o\`u les $v_i\in\pi$ forment
une famille libre, et 
 $\lambda_i\in X^\dual $. Fixons $f\in \mathcal{F'}$. Par hypoth\`ese,
 $\sum_{i=1}^n \lambda_i(f(\mu\cdot v_i))=0$
 pour tout $\mu \in L[G]$. 
Le th\'eor\`eme de densit\'e de Jacobson \cite{Jacobson} assure,
pour tous $x_1,...,x_n\in \pi$, 
l'existence de $\mu\in L[G]$ tel que 
 $\mu\cdot v_i=x_i$ pour $1\leq i\leq n$. On en d\'eduit que $\sum_{i=1}^n \lambda_i(f(x_i))=0$
 pour tous $x_1,...,x_n\in \pi$. Donc $\lambda_i$ est nulle sur $f(\pi)$ pour tout $i$ et tout 
 $f\in \mathcal{F'}$. Puisque $\sum_{f\in \mathcal{F'}} f(\pi)$ est dense dans $X$, on obtient 
 $\lambda_i=0$ pour tout $i$, en contradiction avec 
l'hypoth\`ese $\mathcal{L}\ne 0$. Cela permet de conclure.
 \end{proof}
 
 \begin{coro}\label{isotype} L'image et le noyau de $\iota_{\rm can}: \tHK(\mv_n)[M]\to
 H^1_{\rm dR}({\cal M}^\varpi_n)[M]$ sont $\pi:={\rm LL}(M)$-isotypiques.
 \end{coro}
\begin{proof}
$\mv_n$ est d\'efini sur $F$.
Il s'ensuit que $H^1_{\rm dR}(\mv_n)=C\widehat\otimes H^1_{\rm dR}(\mv_{n,F})$,
et nous laissons au lecteur le soin de v\'erifier
que la ligne du bas des diagrammes de la prop.\,\ref{weak} et du th.\,\ref{theo 1.1}
s'obtiennent
par extension des scalaires \`a $C$ de la ligne analogue sur $F$.
Il s'ensuit que $H^1_{\rm dR}({\cal M}^\varpi_n)[M]$ est $\pi$-isotypique.
On conclut en combinant la proposition pr\'ec\'edente avec le th.\,\ref{main}
et la prop.\,\ref{noyau plonge}.
\end{proof}
              
   \subsubsection{Preuve du th\'eor\`eme \ref{MAIN}} \label{GAB10}
Nous aurons besoin de la variante suivante de la loi de r\'eciprocit\'e de Frobenius.
 \begin{prop}\label{Frobenius reciprocity}
     Soit $\pi$ une repr\'esentation lisse de $G$, irr\'eductible, 
et soit $\mathcal{F}=\pi^\dual \widehat{\otimes}_L X$ un $L$-fr\'echet $\pi$-isotypique. 
Alors, si $U$ est un sous-groupe ouvert compact de~${\mathbb G}(\A_f^\pp)$,
on a un isomorphisme:
$${\cal C}(S^\pp(U),{\cal F})=H^0(G,{\rm Hom}(\pi,{\rm LC}(S(U)))\otimes X.$$
         \end{prop}
\begin{proof}
D'apr\`es le lemme~\ref{gab13}, on a 
$${\cal C}(S^\pp(U),{\cal F})=H^0(G,{\rm Hom}({\rm LC}(S(U))^\dual, \pi^\dual \widehat{\otimes}_L X)).$$
Comme $G$ op\`ere trivialement sur $X$, on peut sortir $X$ de la parenth\`ese,
utiliser l'isomorphisme
${\rm Hom}({\rm LC}(S(U))^\dual, \pi^\dual )={\rm Hom}(\pi,{\rm LC}(S(U)))$
et la finitude de ${\rm Hom}(\pi,{\rm LC}(S(U)))$ qui r\'esulte du fait
que $S(U)$ est compacte et donc ${\cal C}(S(U))$ et, par suite, ${\rm LA}(S(U))$
est une repr\'esentation admissible de $G$.
\end{proof}

D'apr\`es le cor.~\ref{isotype}, on peut \'ecrire
${\rm Ker}(\iota_{\rm can})$ sous la forme $\pi^\dual \widehat{\otimes}_L X$, o\`u $X$
est un $L$-banach avec action triviale de $G$.
Maintenant, il r\'esulte de la prop.~\ref{CSpst} que
$H^1_{\rm HK}({\rm Sh}_n(U))$ est une somme de copies de $M$, et comme $M$ est de pente~$\frac{1}{2}$,
le noyau $t(\bst^+\otimes M)^{N=0,\varphi=1}$ de
$\theta:(\bst^+\otimes M)^{N=0,\varphi=1}\to C\otimes M_{\rm dR}$ est r\'eduit \`a $0$.
On en d\'eduit, en utilisant la prop.~\ref{cruciale}, que
$${\rm Ker}\big({\cal C}(S^\pp(U),\tHK(\mv_n))[M]\to 
{\cal C}(S^\pp(U),H^1_{\rm dR}(\mv_n))[M]\big)^\Gamma=0,$$
pour tout $U$.  En prenant $U$ assez petit, cela prouve que $X=0$ et donc
que
$\tHK(\mv_n)[M]$ s'identifie (via $\iota_{\rm can}$) \`a un sous $L$-fr\'echet de 
$H^1_{\rm dR}({\cal M}^\varpi_n)[M]$. Ce dernier \'etant $\pi$-isotypique (th.\,\ref{theo 1.1}) 
il s'ensuit que $\tHK(\mv_n)[M]$ est $\pi$-isotypique (cor.\,\ref{isotype}). 

\'Ecrivons donc 
$$\tHK(\mv_n)[M]=\pi^\dual \widehat{\otimes}_L Z,$$ 
o\`u $Z$ est un $L$-banach avec action triviale de~$G$. 
En combinant les prop.~\ref{Frobenius reciprocity} et~\ref{cruciale}, on obtient 
un isomorphisme 
$${\rm Hom}_G(\pi, {\rm LC}(U^{\pp}))\otimes_L Z\simeq 
(\bst^+\otimes_{\breve\Q_p}H^1_{\rm HK}({\rm Sh}_n(U)[M]))^{\varphi=p, N=0}.$$
En passant \`a la limite sur $U$, 
puis en appliquant ${\rm Hom}_{{\mathbb G}(\mathbf{A}_f^{\pp})}(\Pi_f^{\pp}, -)$ 
et en utilisant la prop.~\ref{CSpst}, on obtient un isomorphisme 
$$Z\simeq X_{\rm st}^+(M):=(M\otimes_{\Q_p^{\rm nr}} B_{\rm st}^+)^{\varphi=p, N=0}$$
qui permet de conclure.

\begin{appendix}
 \section{Mod\`eles semi-stables \'equivariants des rev\^etements du demi-plan de Drinfeld}\label{GRAB6}
\subsection{Mod\`eles semi-stables des courbes alg\'ebriques}\label{GRAB7}

Ce qui suit est un r\'esum\'e de la th\'eorie classique des mod\`eles
semi-stables des courbes~\cite{QL}; $K$ est un corps de caract\'eristique~$0$
complet pour $v_p$ suppos\'ee discr\`ete.

\subsubsection{Mod\`eles minimaux}\label{GRAB8}
Soit $X$ une courbe projective lisse, g\'eom\'etriquement irr\'eductible, d\'efinie sur $K$.
Un {\it mod\`ele minimal} pour $X$ est un mod\`ele $\sx$ propre, r\'egulier, de $X$, i.e.~un
sch\'ema r\'egulier $\sx$, propre sur $\so_K$, dont la fibre g\'en\'erique est $X$,
et ne contenant pas de courbe exceptionnelle du premier type\footnote{Un sous-sch\'ema ferm\'e
$E$ d'un sch\'ema noeth\'erien $X$ est {\it une courbe exceptionnelle du premier type}
si $E$ est un diviseur de Cartier effectif sur $X$, il existe un corps $k$ tel que
$E\simeq \piqp_k$, et le pull-back du faisceau normal $\sn_{E/X}$ est $\so_{\piqp_k}(-1)$.}.
Un mod\`ele propre r\'egulier existe d'apr\`es le th\'eor\`eme de Lipman de r\'esolution
des singularit\'es sur les surfaces.  De plus, si $\sx$ est un tel mod\`ele, on a une suite
de morphismes
$$
\sx=\sx_m\to \sx_{m-1}\to \ldots\to \sx_1\to \sx_0
$$
de mod\`eles propres r\'eguliers de $X$ telle que chaque morphisme soit obtenu par
contraction d'une courbe exceptionnelle du premier type et que $\sx_0$ soit un mod\`ele
minimal.  En particulier, il existe des mod\`eles minimaux.

\subsubsection{Mod\`eles semi-stables}\label{GRAB9}
  Il existe une extension finie $K'$ de $K$ telle que tout mod\`ele minimal de $X_{K'}$ soit
\'etale semi-stable, i.e.~semi-stable localement pour la topologie \'etale\footnote{
Les composantes irr\'eductibles de la fibre sp\'eciale peuvent avoir de l'autointersection.}.
Un tel mod\`ele est dit {\it minimal \'etale semi-stable}.  Plus pr\'ecis\'ement, il
existe une extension finie $K'$ de $K$ telle que, pour toute extension finie $L$ de $K'$,
il existe des mod\`eles minimaux sur $L$ et, de plus, tout mod\`ele minimal de $X_L$ est
\'etale semi-stable.  En particulier, on peut imposer \`a $L$ d'\^etre une extension
galoisienne de~$K$.
  
Si $X$ est de genre~$\geq 1$, un mod\`ele minimal est unique (\`a isomorphisme unique pr\`es).
Plus g\'en\'eralement, si $\sx$ est un mod\`ele minimal de $X$ et $\sy$
est un mod\`ele r\'egulier propre, alors il existe un unique morphisme
de mod\`eles $\sy\to \sx$; ce morphisme est une suite de contractions de courbes exceptionnelles 
du premier type.
 Si $\sx$, $\sx'$ sont des mod\`eles minimaux \'etale semi-stables de $X$ et $X'$ respectivement,
alors tout morphisme $X'\to X$ s'\'etend, de mani\`ere unique, en un morphisme
$\sx'\to \sx$. En particulier, si $X_{K'}$ a un mod\`ele minimal \'etale semi-stable $\sx$
sur une extension galoisienne $K'$ de $K$, l'action naturelle du groupe
$\Aut_K(X)\times\Gal(K^{\prime}/K)$ s'\'etend \`a~$\sx$. 

On peut aussi \'eclater tous les points d'auto-intersection des composantes de la fibre sp\'eciale de $\sx$
pour obtenir un mod\`ele semi-stable de $X_{K'}$; l'action de $\Aut_K(X)\times\Gal(K^{\prime}/K)$
continue \`a s'\'etendre.  Le r\'esultat est {\it le mod\`ele semi-stable minimal de~$X_{K'}$}.

\subsubsection{Mod\`eles stables}\label{GRAB10}
Un {\it mod\`ele stable} $\sx$ 
d'une courbe projective lisse, g\'eom\'etriquement irr\'eductible, d\'efinie sur $K$,
est un sch\'ema propre et plat sur $\so_K$ tel que $\sx_K\simeq X$ et dont la fibre sp\'eciale
g\'eom\'etrique est une {\it courbe stable}, i.e.~une courbe r\'eduite, connexe, avec uniquement
des singularit\'es nodales,
et dont toutes les composantes irr\'eductibles
de genre 0
rencontrent les autres composantes en au moins 3 points.

Si $X$ est de genre~$\geq 2$, et si $\sx$ est le mod\`ele semi-stable minimal de $X_{K'}$
sur une extension galoisienne finie $K'$ de $K$, alors le mod\`ele stable $\sy$
de $X_{K'}$ est obtenu en contractant les courbes exceptionnelles de la fibre sp\'eciale d'auto-intersection $-2$.
Le mod\`ele stable $\sy$ n'est pas forc\'ement r\'egulier; il est stable par changement de corps de base
et l'action de $\Aut_K(X)\times\Gal(K^{\prime}/K)$ s'\'etend \`a $\sy$.
Plus g\'en\'eralement, si $\sy$, $\sy'$ sont les mod\`eles stables de $X$ et $X'$ respectivement,
alors tout morphisme $X'\to X$ s'\'etend, de mani\`ere unique, en un morphisme
$\sy'\to \sy$.

     \subsection{Mod\`eles semi-stables des courbes analytiques}\label{GRAB11}
Passons aux mod\`eles semi-stables des courbes de Berkovich
\cite{Duc,BPR,Tem}.
     
\subsubsection{Squelette analytique}\label{GRAB12}
Soit $X$ une courbe $K$-analytique.  Rappelons que les points de $X$
se r\'epartissent en $4$ types~\cite[3.3.2]{Duc} suivant la forme de leur corps r\'esiduel compl\'et\'e.
Si $Y\subset X$, on note $Y_{[i]}$ l'ensemble de ses points de type $i\in\{1,2,3,4\}$.
L'ensemble des points de $X$ ayant un voisinage qui est un disque ouvert virtuel~\cite[3.6.34]{Duc}
est un ouvert de $X$ \cite[5.1.8]{Duc}. Le compl\'ementaire $S^{\an}(X)$ de cet ensemble
est le {\it squelette analytique} de $X$; c'est un sous-graphe ferm\'e de~$X$.  Il est
localement fini et inclus dans tout sous-graphe analytiquement admissible\footnote{
Un sous-graphe ferm\'e de $X$ est {\it analytiquement admissible} si toutes les composantes
connexes de $X\moins\Gamma$ sont des disques ouverts virtuels relativement compacts
dans $X$.} de $X$ \cite[5.1.8]{Duc}.  Si $X$ est lisse et si $S^{\an}(X)$ a une intersection
non vide avec toutes les composantes connexes de $X$, alors il est analytiquement admissible et
compos\'e de points de
type $2$ ou $3$ \cite[5.1.10]{Duc}.

Soient $X$ une $K$-courbe analytique lisse et $\Gamma$ un sous-graphe analytiquement admissible
et localement fini de $X$ trac\'e sur $X_{[2,3]}$.  Un point $x$ de $\Gamma$ est {\it un noeud} \cite[5.1.12]{Duc}
s'il satisfait au moins une des conditions suivantes:
\begin{enumerate}
\item $x$ est de genre~$\geq 1$,
\item $x$ est un sommet topologique de $\Gamma$, i.e., la valence de $(\Gamma, x)$ n'est pas $2$,
\item $x$ est un point du bord $\partial^{\an} X$,
\item ${\mathfrak s}(x)$ est\footnote{${\mathfrak s}(\jcdot )$
est le  ``corps des constantes'' \cite[4.5.11]{Duc}.  Si  $K$ est
 alg\'ebriquement clos, aucun point ne satisfait cette condition; il en est de m\^eme si on remplace
$K$ par une extension finie assez grande.} strictement inclus dans ${\mathfrak s}(\beta)$,
o\`u $\beta$ est une branche de $\Gamma$ partant de~$x$.
\end{enumerate}
L'ensemble $\Sigma$ des noeuds de $\Gamma$ est un sous-ensemble ferm\'e et discret de $\Gamma$.

On note $\Sigma^{\rm an}(X)$ l'ensemble des noeuds de $S^{\rm an}(X)$.
Si $X$ est l'analytification d'une courbe projective lisse de genre~$\geq 2$,
alors $\Sigma^{\rm an}(X)$ n'est pas vide
\cite[5.4.16]{Duc}.
 
\subsubsection{Triangulations}\label{GRAB13}
Si $X$ est une $K$-courbe analytique lisse, une
{\em triangulation}\footnote{Ou  {\em semistable vertex set}~\cite{BPR}.}  \cite[5.1.13]{Duc}
de $X$ est la donn\'ee d'un sous-ensemble ferm\'e discret $S$ de $X$, contenu dans
$X_{[2,3]}$, tel que toutes les composantes connexes
de $X\moins S$ soient des disques ou couronnes ouverts virtuels relativement compacts
dans $X$. 
Le squelette $\Gamma$ d'une triangulation $S$ \cite[5.1.14]{Duc} 
est un graphe admissible et localement fini, trac\'e sur
$X_{[2,3]}$, dont les noeuds appartiennent \`a $S$.

Si $X$ est g\'eom\'etriquement connexe et compacte, et si $\Sigma^{\rm an}(X)\neq\emptyset$,
alors 
tous les sommets de $S^{\an}(X)$ sont des noeuds, et
$\Sigma^{\rm an}(X)$ 
est une triangulation de $X$ contenue dans toutes les triangulations de $X$ \cite[5.4.12]{Duc}. 
   
Soit $\sx$ une courbe formelle, plate et normale sur $\so_K$.
Soit $S(\sx)$ l'ensemble des pr\'eimages (par la sp\'ecialisation) dans $\sx_K$ des points g\'en\'eriques
des composantes irr\'eductibles de la fibre sp\'eciale de $\sx$.
Tout
 automorphisme de la paire  $ (\sx_K,S(\sx))$ s'\'etend (de mani\`ere unique) en un automorphisme 
de $\sx$  \cite[6.3.23]{Duc}. Si $\sx$ est semi-stable au sens 
large\footnote{I.e., \'etale localement, de la forme $\Spf \so_K\{X, Y\} /(XY-a) $, $a\in \so_K\moins\{0\}$.}, 
alors $S(\sx)$ est une triangulation: cela suit de ce que la pr\'eimage (par l'application de sp\'ecialisation)
d'un point
non singulier est un disque ouvert et celle d'un point double est une couronne ouverte.
Si $\sx$ est le mod\`ele stable de l'analytification d'une courbe projective lisse $X$
de genre~$\geq 2$, alors
la triangulation $S(\sx)$ est \'egale \`a $\Sigma^{\rm an}(X)$ 
et le squelette de $S(\sx)$
 est \'egal \`a $S^{\an}(X^{\an})$ \cite[Th.\,4.22]{BPR}.

\subsection{Mod\`eles semi-stables \'equivariants de ${\cal M}_n$}\label{GRAB15}
Nous allons utiliser la th\'eorie des mod\`eles semi-stables des courbes analytiques
pour construire des mod\`eles semi-stables \'equivariants de $\mv_n$
{\it sur une extension finie $L$ de $F$}.  Comme $\mv_n$ n'est pas compacte, l'existence d'un tel
mod\`ele ne r\'esulte pas directement du th\'eor\`eme de r\'eduction semi-stable pour les courbes,
mais on peut utiliser l'action de $G$ pour se ramener au cas compact.

Soit $\Gamma\subset G/\varpi^\Z$ un sous-groupe de congruence suffisamment petit pour
que $\Gamma$ op\`ere librement et discr\`etement
sur l'arbre de Bruhat-Tits.
Soit  $ X:=\Gamma\backslash\Omega_{{\rm Dr},F}$. 
Notons $\Omega_{{\rm Dr},F}^+$ le mod\`ele semi-stable standard de $\Omega_{\rm Dr}$ \cite{BC};
l'action de $G$ sur $\Omega_{{\rm Dr},F}$ s'\'etend \`a $\Omega_{{\rm Dr},F}^+$.
Soit $X^{+} :=\Gamma\backslash\Omega_{{\rm Dr},F}^{+}$ -- 
c'est le mod\`ele semi-stable minimal de $X$ \cite[3.1]{Mus},
et c'est aussi son mod\`ele stable.

Soit  $X_n:=\Gamma\backslash\mv_{n,F}$.
Choisissons une extension galoisienne finie $L$ de $F$ telle que
$X_{n,L}$ ait un mod\`ele stable $X_{n,L}^+$. 
Notons $X_{n,L}^{\circ}$ le mod\`ele semi-stable minimal (il est obtenu \`a partir
de $X_{n,L}^+$, de mani\`ere fonctorielle, en \'eclatant les points non r\'eguliers
pour obtenir le mod\`ele minimal \'etale semi-stable, puis en \'eclatant
les points d'autointersection de ce mod\`ele).

Consid\'erons les diagrammes commutatifs suivants de morphismes de courbes $L$-analytiques\footnote{
Nous omettrons l'exposant $({\jcdot})^{\an}$ s'il n'y a pas de risque de confusion.} 
et de sch\'emas formels, respectivement.
  $$
  \xymatrix@R=.6cm{
  \mv_{n,L}\ar[r]^{f_1}\ar[d]^{p_2} & \Omega_{\rm Dr,L}\ar[d]^{p_1}\\
  X_{n,L}\ar[r]^{f_2} & X_L
  }\hskip2cm
  \xymatrix@R=.5cm{ (\mv_{n,L})^{\circ}\ar[r]\ar[d]^-{p_2} & (\mv_{n,L})^{+} \ar@{.>}[r]^{f_1}\ar@{.>}[d]^-{p^+_2}& \Omega_{\rm Dr,L}^{+}\ar[d]^-{p_1}\\
  X_{n,L}^{\circ}\ar[r] & X_{n,L}^{+}\ar[r]^{f_2} & X_L^{+}}
  $$
  Le premier diagramme est commutatif. 
En ce qui concerne
le second diagramme, 
l'application $f_2$ est l'extension de $f_2: X_{n,L}\to X_L$.
 Le  mod\`ele $ (\mv_{n,L})^{+} $ est d\'efini comme le produit fibr\'e de $X_{n,L}^{+}$ et
 $\Omega_{\rm Dr,L}^{+}$ au-dessus de $X_L^{+}$. Il est \'etale au-dessus
de  $X_{n,L}^{+}$ car $\Omega_{\rm Dr,L}^{+}$ l'est au-dessus de $X_L^+$;
 c'est donc un mod\`ele semi-stable g\'en\'eralis\'e de $\mv_{n,L}$. 
De m\^eme, le mod\`ele $ (\mv_{n,L})^{\circ} $ est d\'efini comme le produit fibr\'e de
$X_{n,L}^{\circ}$ et $\Omega_{\rm Dr,L}^{+}$ au-dessus de $X_L^{\circ}$. 
Comme il est \'etale au-dessus de  $X_{n,L}^{\circ}$ c'est un mod\`ele semi-stable de $\mv_{n,L}$. 

\begin{prop}
$(\mv_{n,L})^{\circ} $ est un mod\`ele semi-stable de $\mv_{n,L}$
auquel l'action de $G\times\check G\times \G_F$ s'\'etend de mani\`ere compatible
\`a l'action sur $\Omega_{{\rm Dr},L}^+$.
\end{prop}
\begin{proof}
On a d\'ej\`a vu que $(\mv_{n,L})^{\circ} $ est un mod\`ele semi-stable de $\mv_{n,L}$;
il suffit donc de v\'erifier que l'action de ${\rm Aut}_F(\mv_{n,L})$ s'\'etend \`a ce mod\`ele.
On va commencer par prouver que l'action s'\'etend \`a $(\mv_{n,L})^{+}$.

Comme $X_{n,L}^+$ est le mod\`ele stable de $X_{n,L}$, la triangulation
$S(X_{n,L}^+)$ qu'elle d\'efinit est \'egale \`a l'ensemble des
noeuds $\Sigma^{\rm an}(X_{n,L})$ du squelette analytique de $X_{n,L}$.

Comme l'action de 
$\Gamma$ sur l'arbre de  Bruhat-Tits est libre, il en est de m\^eme de son action sur 
$\mv_{n,L}$ et sur $(\mv_{n,L})^+$.  
Si $x\in\mv_{n,L}$,
le corps r\'esiduel de $x$ est donc \'egal \`a celui de son image
$p_2(x)$ et $p_2$, \'etant \'etale, est un isomorphisme local en $x$ \cite[3.1.5]{JvP}.  
On en d\'eduit que:

$\bullet$ la triangulation $S(X_{n,L}^+)$ est le quotient de $S((\mv_{n,L})^+)$ par $\Gamma$.

$\bullet$ 
$x$ appartient au squelette analytique (resp.~est un noeud) de $\mv_{n,L}$ si et seulement
si $p_2(x)$ appartient \`a celui (resp.~est un noeud) de $X_{n,L}$, et donc
$\Sigma^{\rm an}(X_{n,L})=\Gamma\backslash \Sigma^{\rm an}(\mv_{n,L})$.

On a $\Sigma^{\rm an}(\mv_{n,L})\subset S((\mv_{n,L})^+)$, et comme, d'apr\`es ce qui pr\'ec\`ede,
 les quotients
par $\Gamma$ sont \'egaux, on en d\'eduit que $\Sigma^{\rm an}(\mv_{n,L})= S((\mv_{n,L})^+)$.
Cela implique que $S((\mv_{n,L})^+)$ est invariante par ${\rm Aut}_F(\mv_{n,L})$,
et donc que l'action de ${\rm Aut}_F(\mv_{n,L})$ s'\'etend au mod\`ele
$(\mv_{n,L})^+$.  Comme $(\mv_{n,L})^\circ$ est obtenu fonctoriellement
\`a partir de $(\mv_{n,L})^+$, l'action de ${\rm Aut}_F(\mv_{n,L})$
s'\'etend aussi \`a $(\mv_{n,L})^\circ$.
\end{proof}

\begin{rema}\label{extraordinary}
Si $Y$ est une courbe Stein, connexe, ayant un mod\`ele semi-stable sur $\O_K$, alors
les fonctions born\'ees sur $Y$ sont constantes. En effet, si $f\in\O(Y)$ est born\'ee,
on peut multiplier $f$ par une puissance d'une uniformisante $\pi$ de $K$ de telle sorte que
$f\in\O^+(Y)$.  Alors $f$ modulo $\pi$ induit une fonction $\overline f$
sur la fibre sp\'eciale.  Or les composantes irr\'eductibles de la
la fibre sp\'eciales \'etant propres, $\overline f$ est constante
sur chaque composante irr\'eductible, et donc sur toute la fibre sp\'eciale
puisque celle-ci est connexe.
Si $c$ est un rel\`evement de $\overline f$ dans $\O_K$, on peut appliquer ce qui pr\'ec\`ede
\`a
$\pi^{-1}(f-c)$, et r\'eit\'erer, pour en d\'eduire que $f$ est constante modulo $\pi^n$ pour tout $n$.
\end{rema}
 
\end{appendix}

\end{document}